\documentclass[11pt]{article}
\usepackage{amsthm,amssymb,amsmath,amsfonts,amscd}
\usepackage{graphicx,epsfig,latexsym}
\usepackage{cite,calc,xcolor,subfigmat,mathrsfs,dsfont}
\usepackage{hyperref}
\usepackage[lined,commentsnumbered]{algorithm2e}
\usepackage{tabularx}

\newcommand{\LST}{\mathscr{L}}

\newcommand{\N}{\mathbb{N}}

\newcommand{\NZ}{{\mathbb{N}_0}}
\newcommand{\Span}{{\rm span}}
\newcommand{\ad}{{\rm ad}}

\newcommand{\0}{\mathbf{0}}

\newcommand{\m}{\mathbf{m}}

\newcommand{\Z}{\mathbb{Z}}
\newcommand{\sign}{\rm sign}
\newcommand{\ie}{{\em i.e.,} }
\newcommand{\eg}{{\em e.g.,} }

\newtheorem{thm}{Theorem}[section]

\newtheorem{lem}[thm]{Lemma}

\theoremstyle{definition}
\newtheorem{defn}[thm]{Definition}

\newtheorem{rem}[thm]{Remark}

\numberwithin{equation}{section}
\def\Blem {\begin{lem}}
\def\Elem {\end{lem}}
\def\be {\begin{equation}}
\def\ee {\end{equation}}
\def\ba {\begin{eqnarray}}
\def\ea {\end{eqnarray}}
\def\bes {\begin{equation*}}
\def\ees {\end{equation*}}
\def\bas {\begin{eqnarray*}}
\def\eas {\end{eqnarray*}}
\def\bpr {\begin{proof}}
\def\epr {\end{proof}}
\begin{document}
\baselineskip=18pt

\renewcommand {\thefootnote}{ }

\pagestyle{empty}

\begin{center}
\leftline{}
\vspace{-0.500 in}
{\Large \bf Bifurcation controller designs for the generalized cusp plants of Bogdanov--Takens singularity with an application to ship control } \\ [0.3in]

{\large Majid Gazor\(^{\dag}\)\footnote{$^\dag$Corresponding author. Phone: (98-31) 33913634; Fax: (98-31) 33912602; Email: mgazor@cc.iut.ac.ir; Email: n.sadri@math.iut.ac.ir.} and Nasrin Sadri}

\vspace{0.105in} {\small {\em Department of Mathematical Sciences, Isfahan University of Technology
\\[-0.5ex] Isfahan 84156-83111, Iran }}



\noindent
\end{center}

\vspace{-0.20in}

\baselineskip=16pt

\:\:\:\:\ \ \rule{5.88in}{0.012in}

\begin{abstract}
Nonlinear controlled plants with Bogdanov-Takens singularity may experience surprising changes in their number of equilibria, limit cycles and/or their stability types when the controllers slightly vary in the vicinity of critical parameter varieties. Each such a change is called a local bifurcation.
We derive novel results with regards to truncated parametric normal form classification of the generalized cusp plants. Then, we suggest effective nonlinear bifurcation control law designs for precisely locating and accurately controlling many different types of bifurcations for two measurable plants from this family. The first is a general quadratic plant with a possible multi-input linear controller while the second is a \(\mathbb{Z}_2\)-equivariant general plant with possible multi-input linear (\(\mathbb{Z}_2\)-symmetry preserving) and quadratic (symmetry-breaking) controllers. The bifurcations include from primary to quinary bifurcations of either of the following types: saddle-node, transcritical and pitchfork of equilibria, \(\mathbb{Z}_2\)-equivariant bifurcations of multiple limit cycles through Hopf, homoclinic, heteroclinic, saddle-node, and saddle-connection, and finally their one-parameter symmetry breaking bifurcations. Using our parametric normal form analysis, we propose a new approach for efficient treatment of tracking and regulating engineering problems with smooth manoeuvering possiblities. Due to the nonlinearity of a ship maneuvering characteristic, there is a need for a controller design in a ship steering system so that the ship follows a desired sea route. The results in bifurcation control analysis are applied to two nonlinear ship course models for such controller designs. Symbolic implementations in {\sc Maple} and numerical simulations in {\sc MATLAB} confirm our theoretical results and accurate predictions.
\end{abstract}

\vspace{0.05in}


\vspace{0.05in} \noindent {\it 2010 Mathematics Subject Classification}:\, Primary: 34H20, 34K18, 34C20; Secondary: 58E25.

\vspace{0.05in}

\section{Introduction }

A system is called {\it singular} when there exists a small perturbation of its governing equations so that the {\it qualitative} dynamics of the perturbed and unperturbed systems are different. Each change in the qualitative dynamics of a singular system in the vicinity of an equilibrium is called a {\it local bifurcation}. The most common local bifurcations in nonlinear planar systems are the changes in the number of equilibria, limit cycles, homoclinic and heteroclinic cycles, and changes in their stabilities and/or the existence of bi-stabilities; \eg see \cite{GuckenheimerDangelmayr,DumortierGeneric3,DumortierRoussarie,LangfordHuygens,Slotine,GazorSadri}. An important challenging problem in engineering control is to predict and locate possible trajectories of these types in an engineering singular plant. For the case of a control system, these also include the local linear controllability, stabilizablity, observability, and accessibility of the equilibria; see \cite{ControlBifurcation,Kang04,KangIEEE}. We recall that a {\it typical} control system is composed of four parts: a {\it plant} to be controlled, a possible {\it sensor for measurement} used for {\it feedback control}, a {\it controller} usually determined by a computer programming, and an {\it actuator} for control action; see \cite[page 169]{Slotine}. In this paper we refer to a differential system as a plant when it stands to be controlled through a controller design approach.

The classical bifurcation theory treats an unfolding obtained from the normal form of the folded system. Folding here means excluding the original parameters by setting them to zero. This analysis provides a prediction of possibilities in the system's qualitative dynamics, but it does not represent an actual qualitative and {\it quantitative} dynamics in terms of the original input parameters of the system. Hence, the only useful normal form approach for their possible engineering applications is through parametric normal forms and this has been recently addressed for a few cases: Hopf, single zero, the generalized saddle-node case of Bogdanov-Takens, and a case of Hopf-zero singularity; see \cite{ChenBifuControl,GazorMoazeni,GazorSadri,GazorYuSpec,YuLeung03}.
In this paper we treat the generalized cusp case of Bogdanov-Takens singular systems. We further apply our approach described in \cite[section 6]{GazorSadri} and suggest effective controllers for the families considered in this paper for bifurcation control. Our prime goal in this paper is to give an algorithmically computable approach via our normal form results for both {\it quantitative} and {\it qualitative} analysis of the actual bifurcations occurring in a real life
problem. This provides a comprehensive quantitative understanding about a parametric system's dynamics in terms of its original parameters. This is what the classical bifurcation theory fail to accomplish. We put this contribution into the context of a controller design, since we can choose the numerical values for input parameters of the system according to their associated dynamics. Also, an argument for the necessity of controlling a singular system is as follows. Uncontrolled bifurcations often lead to substantial {\it quantitative} changes into the solutions. Furthermore, there are always small errors (perturbations) in all real world mathematical models due to unavoidable modeling imperfections, data measurements, and computations. Hence, the singular nature of a singular real world model and the small errors are the actual causes for bifurcations. Thereby, the model's solutions do not represent the actual problem in the singular cases, since the errors radically influence bifurcations and solutions. 

\pagestyle{myheadings} \markright{{\footnotesize {\it M. Gazor and N. Sadri \hspace{3.4in} {\it Bogdanov-Takens bifurcation control }}}}

Numerous nonlinear methods have been developed for designing controllers in nonlinear control theory; see \cite{Slotine}. The most well-known methods are feedback linearization and back-stepping methods. Feedback linearization method aims at {\it tracking} control law problem, that is,
introducing a feedback controller so that the solutions of the controlled system asymptotically approach a desired solution. The back-stepping method is an alternative approach usually designed for {\it regulating} (stabilizing) an equilibrium, that is, using a control law so that the equilibrium becomes asymptotically stable. We supply a new feedback controller design approach for solving a family of tracking and regulating problems. This approach is an application of our tools in parametric normal forms and bifurcation control analysis and is an alternative to the existing methods in modern nonlinear control theory.

In this paper we are concerned with a nonlinear parametric system
\be\label{Eq1}
\dot{x}:= F(x, y, \mu), \quad \dot{y}:=-x+ G(x, y, \mu), \quad
\ee for \((x,y)\in \mathbb{R}^2, \mu\in \mathbb{R}^p,\) \(F(0, 0, \0)= G(0, 0, \0)=0.\) This is considered either as a small perturbation or as a controlled system for the Bogdanov-Takens singular plant
\be\label{Eq2}
\dot{x}:= f(x, y), \quad \dot{y}:=-x+ g(x, y),
\ee where \(f(x, y):= F(x, y, \0),\) \(g(x, y):= G(x, y, \0)\) do not have linear and constant terms.
The singular controlled plant \eqref{Eq1} may experience bifurcations in their dynamics by small static variation of the controller parameters \(\mu\) in the vicinity of the origin. This is called {\it control bifurcations}; see \cite{ControlBifurcation,Kang04,KangIEEE,Kang98,AbedFu86}. Thus, this paper investigates the use of a recently introduced notion of truncated universal asymptotic unfolding normal form \cite{GazorSadri} for bifurcation control of a family of Bogdanov-Takens singularity. (A truncated universal asymptotic unfolding normal form here refers to a suitable truncation of the simplest parametric normal form.) In this direction, we develop a systematic approach for normal form computation of this family of systems that they are novel in symbolic implementation in terms of both symbolic constant coefficients and small parametric coefficients. Note that we refer to small size unfolding and control parameters as either parameters or controller inputs, while symbolic coefficients refer to unknown constants and can generally be arbitrarily distanced from zero. Our results are implemented in {\sc Maple} and is being integrated with our {\sc Maple} library,
{\tt Singularity} \cite{GazorKazemiUser,GazorKazemi}, available for bifurcation analysis of smooth maps and differential systems.

The system \eqref{Eq2} can be transformed to a first level normal form
\be\label{ClsNF}
\dot{x}:= \sum^\infty_{k=1} a_ky^{k+1}+\sum^\infty_{k=1} b_kxy^k, \quad \dot{y}:=-x+ \sum^\infty_{k=1} b_ky^{k+1},
\ee
for example see \cite[Lemma 2.2]{GazorMoazeni} and \cite{baidersanders,ZoladekNF2015,wlhj}. Define
\bas
r:=\min \{a_i : a_i \neq 0\} \quad \hbox{ and } \quad s:=\min \{b_i : b_i \neq 0\}.
\eas
Throughout this paper we assume that
\be \label{rs}
r<2s.
\ee This case is called the generalized cusp case of Bogdanov-Takens singularity. Complete normal form and orbital normal form classifications of Bogdanov-Takens singularity have been recently obtained in a series of research results by Stroyzyna and Zoladek \cite{ZoladekNF2015,Zoladek03,Stroyzyna,Zoladek02}; also see \cite{GazorMoazeni,ChenWang,Algaba,kw,BaidSand91,chendora,baidersanders,baiderchurch,yy2001b,yy2001,GazorMokhtari3D}. We prove that the \((s+1)\)-degree truncated simplest parametric normal form of \eqref{Eq1} is given by
\bas
\dot{x}&=&a_r y^{r+1}+b_s x y^s+\sum_{1\leq i \leq r} \nu_i y^{i-1}+\sum_{0\leq i\leq s-1, \,\,i\neq-1\pmod{r+2}} \omega_i x y^{i},\\\nonumber
\dot{y}&=&-x+b_s y^{s+1}+\sum_{0\leq i\leq s, \,\,i\neq-1\pmod{r+2}} \omega_i  y^{i+1},
\eas where \(\nu_i\) and \(\omega_i\) are polynomial functions in terms of \(\mu\in \mathbb{R}^p.\) When \(\nu_i\) and \(\omega_i\) are treated as unfolding parameters, this system is indeed a \((s+1)\)-degree truncated universal asymptotic unfolding normal form for the generalized cusp plants of Bogdanov-Takens singularities.

We show that our parametric normal form approach is sufficient for bifurcation control of equilibria, limit cycles, homoclinic and heteroclinic cycles, saddle-connections, and saddle-node bifurcations of limit cycles for the two most generic generalized cusp cases, \ie \(r=s=1\) and \(r=s=2\). The qualitative dynamics associated with these cases of Bogdanov-Takens singularity have already been studied by several authors; see \cite{Perko94,DumortierGeneric3,DumortierRoussarie,GuckenheimerDangelmayr,PerkoBook,Perko95,Takens74}. Yet we briefly present various normal form {\it bifurcation varieties} with the following targeted goals:
\begin{enumerate}
\item Our derived formulas are computed in terms of normal forms in Hamiltonian-Eulerian decomposition. This is an important normal form style and is different from the styles used in \cite{Perko94,DumortierGeneric3,DumortierRoussarie,GuckenheimerDangelmayr,PerkoBook,Perko95,Takens74}.
\item The bifurcation varieties are derived in terms of symbolic unknown coefficients and input parameters. This is useful in a systematic symbolic study of many applications in real life problems.
\item The difficulties of bifurcation control rise in {\it secondary, tertiary, quaternary,} {\em etc.,} types of bifurcations. Further, a too small size neighborhood validity of the controller inputs (parameters) significantly reduces their possible engineering applications. Hence,
    some bifurcation varieties are derived for a higher degree truncated parametric normal forms than their analogue in the literature. Furthermore, some estimated bifurcation varieties are given in higher orders than those in the literature. Either of these are particularly aimed at enlarging the neighborhood validity associated with the original controller parameters.
\end{enumerate} We remark that a {\it (local) variety} here refers to (a neighborhood subset of) the zero set of a polynomial system and a {\it bifurcation variety} refers to a variety in the parameter space where the singularity occurs. The bifurcation control here aims at designing polynomial controllers with small size coefficient parameters in order to control the local bifurcations and prevent {\it control bifurcations} caused by modeling imperfections. Indeed, the control law is assumed to be a small polynomial perturbation in state variables of the plant so that the small controller parameters can play the role of asymptotic unfolding parameters and thus, the control system is not dominated by the modeling imperfections; see \cite{GazorSadri}. The bifurcation analysis provides a possible list of qualitative asymptotic dynamics for two most generic cases of such systems. These consist of all qualitatively different dynamics that they are persistent to small perturbations of parameters. Further for each desired choice from the persistent dynamics' list, there exist appropriate choices for small controller inputs so that the controlled system follows the chosen asymptotic dynamics. Our approach provides a systematic method for finding these appropriate choices of the controller inputs. Due to the small size nature of the controller inputs, our approach has the potential for designing a flexible and non-expensive control law for some engineering, bio and economical applications. An instance of this is illustrated in subsection \ref{ShipCourse}; \eg see Remark \ref{Rem7.1}. An important characteristic of our approach is that the control polynomial law can be adopted based on the physics of the problem. In fact we can choose the multiple-input polynomial controller from a list of monomial controllers appearing in a parametric normal form of a truncated-jet system. Then we are able to distinguish the suitable control parameters and choose appropriate values from parameters of the original differential singular system, where it is modeled based on physics of the problem; \eg see Remark \ref{Rem7.1}. These highlight our claimed contribution to the controller designs in the real life problems.

The rest of this paper is organized as follows. Section \ref{sec2} derives the necessary formulas for time rescaling and orbital normal form hypernormalization steps. A truncated simplest parametric normal form for the generalized cusp plants of Bogdanov-Takens singularity is computed in section \ref{sec3}. A list of estimated bifurcation varieties for two most generic cases (\(r=s=1\) and \(r=s=2\)) are presented in sections \ref{sec5} and \ref{sec6}, respectively. In section \ref{sec7}, we apply our suggested approach to a general quadratic-jet plant with a multi-input linear controller and a \(\mathbb{Z}_2\)-equivariant general cubic-jet plant with two types of quadratic controllers; one is \(\mathbb{Z}_2\)-preserving and the other is a one-parameter \(\mathbb{Z}_2\)-symmetry violating controller. We use our recently proposed approach to find the parameters of a parametric system that they can play the role of distinguished parameters. Next we demonstrate how our approach helps to find appropriate choices of parameters to precisely control and locate any possibly desired dynamics. Subsection \ref{BifCont} introduces an application of bifurcation control for regulating and tracking control engineering problems. Then in subsection \ref{ShipCourse}, our bifurcation control results are applied to solve a ship course tracking control problem; that is to design a controller so that the controlled ship course follows a desired path on a sea.
We consider two ship course models: one only takes the nonlinearity between ship's yaw and rudder angles while the second adds the steering gear's dynamics.

\section{ Orbital normal forms }\label{sec2}

The classical normal forms are not the simplest normal forms and hyper-normalization are generally possible by using the available near identity change of state variables and also by using time rescaling transformations. This section and the next section are devoted to the orbital and parametric hypernormalizations beyond the classical normal form theory; see \cite{MurdockFinal,GazorMoazeni,Zoladek03,Zoladek02,ZoladekNF2015,BaidSand91,GazorYuSpec,MurdBook,GazorMokhtariEul,Murd04} and the references therein. We follow \cite{GazorMoazeni,baidersanders} and recall the following notations
\ba\nonumber
A_{k}^{l}&:=&\frac{k-l+1}{k+2}x^{l+1}y^{k-l}\frac{\partial}{\partial
x}-\frac {l+1}{k+2}x^ly^{k-l+1}\frac{\partial}{\partial y} \hspace{1.3cm} (-1\leq l\leq k+1),\\\label{ABTerms}
B_{k}^{l}&:=& x^{l+1}y^{k-l}\frac{\partial}{\partial x}+x^ly^{k-l+1}\frac{\partial}{\partial y}\hspace{4cm} (0\leq l\leq k),
\\\nonumber
Z^l_k&:=& x^{k-l}y^l  \hspace{7.5cm} (0\leq l\leq k).
\ea The monomial \(h^l_k:= \frac{1}{k+2}x^{l+1}y^{k-l+1}\) is the Hamiltonian for the Hamiltonian vector field \(A^l_k.\) We follow \cite{GazorMoazeni} to denote \(\LST\) for the Lie algebra generated by formal vector fields of type \eqref{ClsNF}. Thus, the Lie algebra \(\LST\) can be expressed as formal sums in combinations of terms \(A_{k}^{l}, B_{k}^{l}\) for \(k\geq 1\) and \(A^1_0.\) The ring of formal power series generated by \(Z^l_k\) for \(k\geq1\) is denoted by \(\mathcal{R}.\) Any \(T\in \mathcal{R}\) corresponds to a time rescaling generator, that is, a near identity time rescaling \(\tau:=(1+T)t\) that sends old time variable \(t\) to new time variable \(\tau\) and a vector field \(v\) into \(v+Tv\). Hence, \(\LST\) holds a \(\mathcal{R}\)-module structure. The following lemma recalls the Lie algebra and \(\mathcal{R}\)-module structure constants from \cite[Theorem 3.7]{baidersanders} and \cite[Lemma 2.1]{GazorMoazeni}, respectively.

\begin{lem}\label{StrucCons}
The \(\mathcal{R}\)-module and the Lie algebra structure constants of \(\LST\) are given by
\begin{eqnarray}\nonumber
Z^m_n B^l_k&=& B^{m+l}_{n+k}, \qquad\qquad\;\quad\qquad\quad Z^m_n A^l_k= A^{m+l}_{n+k}+ \frac{(k+2)m-n(l+1)}{(k+2)(k+n+2)}B^{m+l}_{n+k},\\
{\left[B^m_n, B^l_k\right]}&=&(k-n)B^{m+l}_{n+k},\qquad \quad\qquad {\left[A^m_n, A^l_k\right]}=(n+k+2)\left(\frac{l+1}{k+2}-\frac{m+1}{n+2}\right)A^{m+l}_{n+k},\\\nonumber
{\left[A^m_n, B^l_k\right]}&=&\frac{(k+2)}{(k+n+2)}\left(l-\frac{k(m+1)}{n+2}\right)
B^{m+l}_{n+k}-nA^{l+l}_{n+k}.
\end{eqnarray}
\end{lem}

Normal form theory uses the concept of a graded Lie algebra, that is, \(\LST:=\sum_k \LST_k\) where \(\LST_k\) denotes the vector space of homogenous vector fields of grade \(k\) and \([\LST_k, \LST_l]\subseteq \LST_{k+l}.\) Further we use graded ring structure for \(\mathcal{R}:=\sum_k\mathcal{R}_k,\) where \(\mathcal{R}_l\LST_k\subseteq \LST_{l+k}\) and \(\mathcal{R}_l\mathcal{R}_k\subseteq \mathcal{R}_{l+k}.\) Given an normalizing vector field \(v:=\sum^\infty_{k=0} v_k, v_k\in \LST_k,\) we define \(d^{k, 1}(T_k, S_k):= T_kv_0+[S_k, v_0]\) for \((T_k, S_k)\in \mathcal{R}_k\times \LST_k.\) Now for natural numbers \(r\geq 2\) and \(k\) we inductively define the map
\(d^{k, r}: \ker d^{k-1, r-1}\times \mathcal{R}_k\times \LST_k\rightarrow \LST_k\) by
\ba\label{d}
&d^{k, r}(T^{r-1}_{k-r+1}, \cdots, T^{r-1}_{k-1}, T_k,
S^{r-1}_{k-r+1}, \cdots, S^{r-1}_{k-1}, S_k):= \sum^{r-1}_{i=1} \left(T^{r-1}_{k-i}v_i+ [S^{r-1}_{k-i}, v_i]\right) + T_kv_0+[S_k, v_0],&
\ea where \((T^{r-1}_{k-r+1}, \cdots, T^{r-1}_{k-1}, S^{r-1}_{k-r+1}, \cdots, S^{r-1}_{k-1})\in \ker d^{{k-1}, r-1}.\) Here, some rearrangements of grade-homogenous components \(T^{r-1}_{i}\) and \(S^{r-1}_{i}\) from \(\mathcal{R}_i\) and \(\LST_i\) are made for our convenience. By \cite[Theorem 4.3 and Lemma 4.2]{GazorYuSpec} we can transform the vector field \(v\) into a \(r\)-th level partially extended orbital normal form \(v^{r}:= \sum v_k^{r}\), where \(v^r_k\in C^r_k\) and \(C^r_k\) is a complement space to \({\rm im}\, d^{k, r}\) for any \(k.\) A normal form style refers to a rule on how to choose the complement space \(C^r_k\) for each \(k.\) We here give the priority of elimination to lower grade-homogenous \(A\) and \(B\) terms than those of higher grades. Further for terms of the same grade, the priority of elimination is with \(B\)-terms rather than \(A\)-terms. The Pochhammer notation \((a)_b^n:= a(a+b)(a+2b)\cdots (a+(n-1)b)\) is used in this paper.

Given the formulas in equations \eqref{ABTerms}, the classical normal form \eqref{Eq2} can be written as
\be\label{clsAB}
v^{(1)}= A^1_0+\sum^\infty_{k=r} a_kA^{-1}_k+ \sum^\infty_{k=s} b_kB^{0}_k.
\ee We follow Baider and Sanders \cite[Page 219]{BaidSand91} and define \(\Gamma_r\) by
\(\ad_{A^{1}_{0}}\circ\ad_{\Bbb{A}_{r}},\) where \(\mathbb{A}_r:=A^1_0+a_r A^{-1}_{r}\).
Then,
\begin{equation}
\ker(\Gamma_r)=\Span \Big\{\mathcal{A}^{k+1}_k, \mathcal{B}^{k}_k : k=0,1,2,\cdots\Big\}.
\end{equation}
Here
\begin{equation}
\mathcal{A}^{k+1}_{k}:=\sum^{\lfloor\frac{k}2\rfloor}_{l=-1}\dfrac{{a_r}^{l+1}\big(k+r(l+1)+2\big)\big(k\big)^l_{-2}}{\big(r+2\big)^{l+1}_{r+2}}
A^{k-(2l+1)}_{k+r(l+1)},
\end{equation}
and \(\mathcal{B}^{k}_{k}\) is given by
\begin{eqnarray}
\sum^{\lfloor\frac{k-2}2\rfloor}_{l=-1}\dfrac{{a_r}^{l+1}(k+2)\big(k)^{l+1}_{-2}}{\big(k+lr+r+2\big)^{l+1}_{r}\big(r+2\big)^{l+1}_{r+2}}B^{k-2(l+1)}_{k-r(l+1)}
-\sum^{\lfloor\frac{k-1}2\rfloor}_{l=0}\dfrac{{a_r}^{l+1}r(k+lr+r+2)\big(k-1\big)^{l}_{-2}}{(k+r+2)\big(r+3\big)^{l+1}_{r+2}}A^{k-2(l+1)}_{k+r(l+1)}&&\\\nonumber
-\sum^{\lfloor\frac{k-2}2\rfloor}_{q=0}\dfrac{{a_r}^{q+2}r (k+2)\big(k\big)^{q+1}_{-2}}{\big(r+2\big)^{q+1}_{r+2}\big(k+r+2+qr\big)}\sum^{\lfloor\frac{k-2q-3}2\rfloor}_{l=0}\frac{{a_r}^l \big(k+r(q+l+2)+2\big)\big(k-2q-4\big)^{l}_{-2}}{\big(k+r(q+2)+2\big)\big(2q+r(q+2)+5\big)^{l+1}_{r+2}}A^{k-2(q+l+2)}_{k+r(q+l+2)}.&&
\end{eqnarray}

Define
\begin{equation}
\mathfrak{A}^m_n:=\sum^{\lfloor\frac{m}2\rfloor}_{l=0}\dfrac{-{a_r}^{l} (m)^{l}_{-2}(n+lr+2)}{(n+2)\big(n-m+2\big)^{l+1}_{r+2}}A^{m-(2l+1)}_{n+lr}
\end{equation}
and
\begin{eqnarray}\nonumber
\mathfrak{B}^m_n&:=&
\sum^{\left\lfloor\frac{m-1}2\right\rfloor}_{l=0}\dfrac{-{a_r}^l\big(n+2\big)^l_r\big(m-1\big)^l_{-2}}{\big(n-m+1\big)^{l+1}_{r+2}\big(n+r+2\big)^l_r}
B^{m-(2l+1)}_{n+lr}+\sum^{\left\lfloor\frac{m-1}2\right\rfloor}_{q=0}\dfrac{{a_r}^{q+l+1}r\big(n+2\big)^q_r\big(m-1\big)^q_{-2}}{\big(n-m+1\big)^{q+1}_{r+2}
\big(n+r+2\big)^{q+1}_r}\\
&&\sum^{\left\lfloor\frac{m-2q-3}2\right\rfloor}_{l=0}\dfrac{\big(n+r(l+q+1)+2\big)\big(m-2(q+1)\big)^l_{-2}}{\big(n-m+r(q+1)+2q+4\big)^{l+1}_{r+2}}
A^{m-2(q+l)-3}_{n+r(q+l+1)}.
\end{eqnarray}

\begin{lem}\label{AReduce} For nonnegative integers \(n, l\) and \(n>0\), we have
\begin{equation}
A^{2l}_n+\left[\mathfrak{A}^{2l}_n,\mathbb{A}_r\right]=0, \qquad
A^{2l-1}_n+\left[\mathfrak{A}^{2l-1}_n,\mathbb{A}_r\right]=\dfrac{{a_r}^{l}(2l-1)^{l-1}_{-2}\big(n+2+rl\big)}{(n+2)\big(n-2(l-1)\big)^{l}_{r+2}}A^{-1}_{n+lr},
\end{equation}
while
\begin{equation}
B^{2l}_n+\left[\mathfrak{B}^{2l}_n, \mathbb{A}_r\right]=\dfrac{{a_r}^{l}\big(n+r(l-1)+2\big)\big(n+2\big)^{l-1}_{r}\big(2l-1\big)^{l-1}_{-2}}{(n+lr+2)\big(n-2l+1\big)^{l}_{r+2}
\big(n+r+2\big)^{l-1}_{r}}B^{0}_{n+lr},
\end{equation}
and for \(l>0\) the following holds:
\begin{eqnarray*}
B^{2l-1}_n+\left[\mathfrak{B}^{2l-1}_n,\mathbb{A}_r\right]&=&\sum^{l-1}_{i=0}\frac{-{a_r}^l r\big(n+lr+2\big)\big(n+2\big)^{i}_{r}\big(2l-2\big)^{i}_{-2}(3)^{l-i-2}_2}{\big(n+2-2l\big)^{i+1}_{r+2}\big(n+r(i+1)+2\big)^{i-1}_{r}
\big(n-m+1+r(l-1)\big)^{l-i-1}_{-(r+2)}}A^{-1}_{n+lr}.\quad
\end{eqnarray*}
\end{lem}
\bpr
A direct computation verifies our derived formulas.
\epr
The formulas given in Lemma \ref{AReduce} provide a method for deriving the transformation generators that they can eliminate \(A^i_k\) and \(B^j_k\) for \(i\neq -1\) and \(j\neq 0\) in the hypernormalization steps.

\begin{lem}\label{ABReduce}
For each nonnegative integer \(l\), we have
\begin{equation}
\left[\mathcal{A}^{2l+1}_{2l}, \mathbb{A}_r\right]=0, \qquad
\left[\mathcal{A}^{2l}_{2l-1}, \mathbb{A}_r\right]=-\dfrac{{a_r}^{l+1}\big((r+2)(l+1)-1\big)}{\big(r+2\big)^{l}_{r+2}}A^{-1}_{l(r+2)+r-1} \neq 0,
\end{equation}
while
\begin{equation}
\left[\mathcal{B}^{2l-1}_{2l-1}, \mathbb{A}_r\right]=-\dfrac{{a_r}^{l}\big(l(r+2)-r+1\big)\big(2l+1\big)^{l-1}_{r}\big(2l-1\big)^{l-1}_{-2}}{\big(l(r+2)+1\big)\big(r+2\big)^{l-1}_{r+2}
\big(2l+r+1\big)^{l-1}_{r}}B^0_{l(r+2)-1} \neq 0
\end{equation}
and
\begin{equation}
\left[\mathcal{B}^{2l}_{2l}, \mathbb{A}_r\right]=\sum^{l-1}_{i=-1}\dfrac{2 {a_r}^{l+1}r (l+1)^2 \big(2l\big)^{i+1}_{-2}(3)^{l-2-i}_2}{\big(2l+r(i+1)+2\big)^2_{r}\big(r+2\big)^{i}_{2r+4}\big(r(i+2)+2i+5\big)^{l-i-1}_{r+2}}A^{-1}_{l(r+2)+r} \neq 0.
\end{equation}
\end{lem}

\bpr The proof follows direct computations by using the structure constants in Lemma \ref{StrucCons}.
\epr Define
\begin{eqnarray}\label{ztrans}
\mathcal{Z}^{m}_{n}:=\dfrac{1}{(m-n-1)}A^m_n+\frac{1}{n+2}B^m_n- \sum^{\left\lfloor\frac{m-1}2\right\rfloor}_{l=0}\dfrac{ {a_r}^{l+1}\big(n+r(l+1)+2\big)\big(m-1\big)^{l}_{-2}}{\big(n-m+1\big)^{l+2}_{r+2}} A^{m-2(l+1)}_{n+r(l+1)}.
\end{eqnarray}

\begin{lem}\label{ZReduce}
For any natural number \(l,\)
\begin{equation}
Z^{2l-1}_n \mathbb{A}_r+\big[\mathcal{Z}^{2l-1}_{n}, \mathbb{A}_r\big]=0,
\end{equation}
\begin{equation}\label{Z2ln}
Z^{2l}_n \mathbb{A}_r+\big[\mathcal{Z}^{2l}_{n}, \mathbb{A}_r\big]=\frac{{a_r}^{l+1}\big(n+r(l+1)+2\big)\big(2l-1\big)^{l-1}_{-2}}{\big(n-2l+1\big)^{l+1}_{r+2}}A^{-1}_{n+r(l+1)} \neq 0.
\end{equation}
\end{lem}

\begin{thm} The \((r+1)\)-th level partially extended orbital normal form of \(v^{(1)}\) is
\begin{equation}\label{rlevel}
v^{(r+1)}:=A^1_0+ a_{r}A^{-1}_r +\sum_{k=s}^{\infty} b_{k}B^{0}_{k},
\end{equation}
where \(b_{k}=0\) for \(k= -1\pmod{r+2}.\)
\end{thm}
\begin{proof}
The proof readily follows Lemmas \ref{AReduce}, \ref{ABReduce} and \ref{ZReduce}.
\end{proof}

\begin{rem}
\begin{itemize}
\item[a.] The number \(s\) must be updated in the \((r+1)\)-th level hypernormalization, \ie \(s:=\min \{k: b_k \neq 0\}\), where \(b_k\) denotes the coefficients given in equation \eqref{rlevel}.
\item[b.]\label{rem(b)} Let \(a\) and \(b\) be arbitrary real numbers. By using the change of coordinates and time rescaling
\begin{eqnarray}\nonumber
X&:=&\left(\frac{a_r}{a}\,\sign\left(\frac{a_r}{a}\right)^{r+1}\right)^{\dfrac{s+1}{2s-r}} \left(\frac{b}{b_s}\,\sign\left(\frac{a_r}{a}\right)^{rs}\,\sign\left(\frac{b}{b_s}\right)\right)^{\dfrac{r+2}{2s-r}} x,\\
Y&:=&\left(\frac{a_r b^2}{a {b_s}^2}\,\sign\left(\frac{a_r}{a}\right)^{r+1}\right)^{\dfrac{1}{2s-r}}y,\\\nonumber
\tau&:=&\left(\frac{a}{a_r}\,\sign\left(\frac{a_r}{a}\right)^{r+1}\right)^{\dfrac{s}{2s-r}} \left(\frac{b_s}{b}\,\sign\left(\frac{a_r}{a}\right)^{rs}\,\sign\left(\frac{b}{b_s}\right)\right)^{\dfrac{r}{2s-r}} t,
\end{eqnarray}
we can change the coefficient \(a_r\) to \(a \,\sign(a a_r)^{r+1}\) and \(b_s\) to \(b \,\sign(b b_s)\,\sign(a a_r)^{rs}\).
\end{itemize}
\end{rem}
Similar to \cite{baidersanders,BaidSand91,GazorMoazeni,GazorSadri} we define
\(\Phi_k:=\ad_{A^{-1}_{0}}\circ d^{{k+r}, {r+1}},\) where \(d^{{k+r}, {r+1}}\) is given by equation \eqref{d}. Hence, \(\ker \Phi_k\) is given by
\begin{equation}
\Span \bigg\{(\mathcal{A}^{k+1}_k, \0),(\mathcal{B}^{k}_k, \0),\left( \mathcal{B}^{2k}_{2k}+\frac{1}{k(r+2)+2}B^0_{k(r+2)}-\frac{1}{k(r+2)+1}A^0_{k(r+2)}, Z^0_{k(r+2)}\right): k \in \mathbb{N}\bigg\}.
\end{equation}
\begin{lem}
In the \((2s+1)\)-th level orbital normal form, term \(B^0_k\) is eliminated for \(k=s \pmod{r+2}\) when \(k>r+s+2\). The term \(B^0_{r+s+2}\) is eliminated in this level when \(3r+4 \neq 2s\).
\end{lem}

\bpr
Since
\begin{eqnarray}
\left[\mathcal{B}^{2l}_{2l}, \mathbb{A}_r\right]&=&\sum^{l-1}_{i=-1}\dfrac{2^{i+2} {a_r}^{l+1}r (l+1)^2 \big(l\big)^{i+1}_{-1}\big(2l-2i-3\big)^{l-2-i}_{-2}}{\big(r(i+1)+2(l+1)\big)^2_{r}\big(2r+4\big)^{i}_{r+2}\big((r+2)(i+2)+1\big)^{l-i-1}_{r+2}}A^{-1}_{l(r+2)+r}
\end{eqnarray}
and
\begin{eqnarray}
Z^0_{l(r+2)} \mathbb{A}_r&+&\left[\frac{1}{l(r+2)+2}B^0_{l(r+2)}-\frac{1}{l(r+2)+1}A^0_{l(r+2)}
, \mathbb{A}_r\right]=\dfrac{a_r (l+1)(r+2)}{l(r+2)+1}A^{-1}_{l(r+2)+r},
\end{eqnarray}
we may use a linear combination of \(\mathcal{B}^{2l}_{2l}\) and \( \bigg(\frac{-1}{l(r+2)+1}A^0_{l(r+2)}+\frac{1}{l(r+2)+2}B^0_{l(r+2)},Z^0_{l(r+2)}\bigg)\) in the \((s+1)\)- level for possible hypernormalization. On the other hand for any natural number \(l\),
\begin{eqnarray}\nonumber
Z^0_{l(r+2)} B^0_s&=& B^0_{l(r+2)+s},\\
\left[\frac{-1}{l(r+2)+1} A^{0}_{l(r+2)}, B^0_s\right]&=&\dfrac{-l(r+2)}{l(r+2)+1} A^{0}_{l(r+2)+s}+\dfrac{s(s+2)}{\big(l(r+2)+1\big)^2_1\big(l(r+2)+s+2\big)} B^{0}_{l(r+2)+s},\qquad \\\nonumber
\left[\frac{1}{l(r+2)+2} B^0_{l(r+2)}, B^0_s\right]&=&\dfrac{s-l(r+2)}{l(r+2)+2} B^{0}_{l(r+2)+s},
\end{eqnarray}
 while the impact \(\left[\mathcal{B}^{2l}_{2l},B^0_s\right]\) can be computed via
\begin{eqnarray}\nonumber
\left[A^{2l-2(i+1)}_{2l+r(i+1)},B^0_s\right]&=&{\frac {\left( s\right)^2_2  \left( 2i-2l+1 \right) }{ \left( 2l+ir+r+2 \right)^2_s }}B^{2l-2i
-2}_{2l+r \left( i+1 \right) +s} - \left( 2l+r \left( i+1 \right)  \right) A^{2l-2i-2}_{2l+r
 \left( i+1 \right) +s},\\
\left[B^{2l-2(i+1)}_{2l+r(i+1)},B^0_s\right]&=&\left( s-2l-r \left( i+1 \right)  \right) B^{2l-2i-2,}_{2l+r
 \left( i+1 \right) +s},\\\nonumber
\left[A^{2l-2(q+i+2)}_{2l+r(q+i+2)},B^0_s\right]&=&\frac {\left( s\right)^2_2  \left( 2i-2l+2q+3 \right) }{\left( 2l+r \left( q+i+2 \right) +2\right)^2_s }B^{2(l-q-i-2)}_{2l+r(q+i+2)+s}- \left( 2l+r \left( q+i+2 \right)\right)
 A^{2(l-q-i-2)}_{2l+r(q+i+2)+s}.\quad
\end{eqnarray}
By Lemmas \ref{AReduce} and \ref{ABReduce} there exists a \(S_l\in \LST_{2l(r+2)+2s-r}\) so that the effect of
\begin{equation}
\mathcal{B}^{2l}_{2l}+\frac{-1}{l(r+2)+1}A^0_{l(r+2)}+\frac{1}{l(r+2)+2}B^0_{l(r+2)}, S_l, Z^0_{l(r+2)},
\end{equation} in the \((2s+1)\)-th level generates \(P(a_r) B^0_{l(r+2)+s}\in {\rm im}\, d^{{2l(r+2+s)},{2s+1}}\)
 where \(P(a_r)\) is a polynomial expression of degree \((l+1)\) in \(a_r\). Indeed for \(l\geq 2\), the smallest power of \({a_r}\) in \(P(a_r)\) is \(2\) and its coefficient  is given by
\bas
\dfrac{4 r s l (s+2) (l+1)^2 \big(lr+r+4\big)^2_{-r} \big(2l-4\big)^{l-2}_{-2}}{(lr+2l+1)\big(2l+r+2\big)^2_r \big(2r+5\big)^{l-1}_{r+2}}\neq 0.
\eas
When \(l=1\), the coefficient of \({a_r}^2\) in \(P(a_r)\) is
\bas
-{\frac {2\left(s+2\right) \left(3r+2s+8\right)  \left(3r-2s+4 \right)}{\left(s+1\right) \left(r+s+4\right)\left(r+3\right)  \left(r+2\right) }}.
\eas
The latter is non-zero when \(3r+4 \neq 2s\). Since without loss of generality we can choose \(a_r\) as a non-algebraic number by remark \ref{rem(b)}(b), \(P(a_r)\) is non-zero. This completes the proof.
\epr
\begin{thm}\label{Orb2s1}
The \((2s+1)\)-th level partially extended orbital normal form of the system  \eqref{Eq2} is given by
\begin{equation}
v^{(2s+1)}=A^1_0+ a_{r}A^{-1}_r +\sum_{k=s}^{\infty} b_{k}B^{0}_{k}.
\end{equation}
Here \(b_{k}=0\) when \(k=-1\pmod{r+2},\) and for \(k>s\) when \(k=s \pmod{r+2}.\)
\end{thm}
\bpr
By Lemmas \ref{ABReduce} and \ref{ZReduce} we have
\be\label{ZZ}
 \left(Z^{2l+1}_n,\mathcal{Z}^{2l+1}_{n}\right)\in \ker d^{2n+2lr+2r, {r+1}} \hbox{ and } \left(A^{2l+1}_{2l},0\right) \in \ker d ^{r+1, 2l+2lr+2r}.
\ee
On the other hand
\bes
\left[\mathcal{A}^{2l+1}_{2l}, {B}^0_s\right] \in \Span \{ A^{-1}_{2l+s+r(1+l)}\}.
\ees
Hence,
\bes
d^{2n+2s+2lr-r, 2s+1}\left(Z^{2l-1}_{n},\mathcal{Z}^{m}_{n}\right)\in \Span \{A^{-1}_{n+s+lr}\}  \quad \hbox{ and } \quad
Z^{2l-1}_n {B}^0_s+\big[\mathcal{Z}^{2l-1}_{n}, \mathbb{A}_r\big]\in \Span \{{ A^{-1}_{n+s+lr}}\}
\ees imply that the kernel terms in \eqref{ZZ} do not contribute to further hypernormalization in \((2s+1)\)-th level.
\epr

\section{ Parametric normal forms }\label{sec3}

This section is devoted to the computation of parametric normal forms for the generalized cusp case of Bogdanov--Takens singularity, \ie the differential system \eqref{Eq1}, where the system \eqref{Eq2} satisfies equation \eqref{rs}; also see \cite{Stroyzyna,Zoladek02}. We denote \(\m\) for \((m_1, \ldots, m_p)\in \NZ^p, \NZ:=\N\cup\{0\},\) \(\mu^\m\) for \({\mu_1}^{m_1}\cdots {\mu_p}^{m_p},\) and \(|\m|:= \sum^p_{i=1} m_i.\)

\begin{defn}
The grading function
\begin{equation}
\delta(\mu^\mathbf{m}A^l_k)= \delta(\mu^\mathbf{m} B^l_k)=lr+2k+(2r+1)|\mathbf{m}|
\end{equation}
extends the grading function introduced in \cite[equation 4.3]{baidersanders} into parametric cases; also see \cite[equation 3.5]{GazorMoazeni}.
\end{defn}

\begin{lem}
There exist polynomial maps \(v_i(\mu)\)  and \(w_i(\mu)\) so that equation (\ref{Eq1}) can be transformed into
\begin{eqnarray}
v^{(r+1)}&=&A^1_0+a_r A^{-1}_r +\sum_{1\leq i \leq r} v_{i}(\mu) A^{-1}_{i-2}+\sum_{0\leq i < s} w_{i}(\mu)B^{0}_{i} +\sum_{i=s}^{\infty} (b_{i}+w_{i}(\mu))B^{0}_{i},
\end{eqnarray} where \(b_{i}=0\) for \(i=-1\pmod{r+2}.\) Further, \(w_{i}(\0)=0\) and \(v_{i}(\0)=0\) for all \(i \geq 0\).
\end{lem}
\begin{proof}
By\cite[Lemma 5.1]{GazorMoazeni}, all parametric terms \(A^i_k \mu^\mathbf{m}\) and \(B^j_k \mu^\mathbf{m}\) for \(i \neq -1\), \(j \neq 0, k>0\) can be eliminated in the first level parametric normal form. Let \(a_r= a_{r\0}\).
Parametric terms of the form \(A^{-1}_r \mu^\mathbf{m}, A^0_0 \mu^\mathbf{m},\) and \(A^1_0\mu^\mathbf{m}\) for any \(\mathbf{m}=(m_1, m_2, \ldots, m_k) \in \NZ,\) when \(\sum m_i >0,\) can be simplified from the system (\ref{Eq1}) due to

\begin{equation}
Z^0_0 \mathbb{A}_r+\big[-A^0_0, \mathbb{A}_r\big]=\dfrac{a_r (r+4)}2A^{-1}_{r}, \quad \big[A^{-1}_0, A^1_0\big]=2 A^0_0 \quad \hbox{ and } \quad
\big[A^0_0,A^1_0\big]=A^1_0.
\end{equation}
Consider an arbitrary parametric map \(c(\mu)\) with \(c(\0)=0\). By primary shift of coordinates (see \cite[page 373]{MurdBook}), replacing \(x\) by \(x+c(\mu)\), we can eliminate terms of the form \(A^0_{-1}c(\mu)\). On the other hand by applying secondary shift of coordinates
(see \cite[page 373]{MurdBook}), \ie replacing \(y\) by \(y+c(\mu)\), on the first level normal form system,  we can eliminate terms of the form \(A^{-1}_{r-1}\mu^{n}\). This is due to the equation
\begin{eqnarray}
A^{-1}_r\big(y+c(\mu)\big)&=&\sum^{r+1}_{i=0} {r+1 \choose i} c(\mu)^i A^{-1}_{r-i}.
\end{eqnarray}
The impact of the secondary shift of coordinates on B-terms are given by
\begin{eqnarray}\label{B0kImp}
B^{0}_k\big(y+c(\mu)\big)&=&\sum^{k}_{i=0} {k \choose i} c(\mu)^i B^{0}_{k-i}+ \sum^{k}_{i=1} {k \choose i-1} c(\mu)^i \left(\frac{k-i+1}{k-i+2}B^0_{k-i}-A^{0}_{k-i}\right)-c(\mu)^{k+1} A^{0}_{-1}. \qquad
\end{eqnarray} In each grade we first apply the secondary shift of coordinates to simplify \(\mu ^\mathbf{m}A^{-1}_{r-1}.\) Then, we use the primary shift of coordinates to simplify \(\mu ^\mathbf{m}A^{0}_{-1}.\) The remaining terms, that are created through the equation \eqref{B0kImp}, can be eliminated via the changes of state variables like in the first level normal form. The latter transformations do not influence the coefficient of \(A^{-1}_0.\)

Since
\begin{equation}
\big[B^0_0,A^1_0\big]=0, \quad \big[B^0_0,A^{-1}_r\big]=a_r A^{-1}_{r},\quad \hbox{ and } \quad
Z^0_0 A_r+\left[-A^0_0, \mathbb{A}_r\right]=\dfrac{a_r (r+4)}2A^{-1}_{r},
\end{equation}
we have

\begin{equation}\label{eq3.7}
\left(\dfrac{a_r (r+4)}2B^0_0, a_r r A^0_0, -a_rr Z^0_0\right)\mu^\m \in \ker d^{r+|\m|(2r+1), r+1}.
\end{equation} This implies that we can not simplify any more term in the \((r+1)\)-th level. Thus the constant and linear terms in state variables
of \(A^0_0,\) \(A^1_0,\) \(A^0_{-1},\) \(A^{-1}_r,\) and \(A^{-1}_{r-1}\) are simplified in the \((r+1)\)-th level parametric normalization step.
\end{proof}

Now we update \(s:=\min\{k: b_{k} \neq 0\}\), and similar to \cite{GazorMoazeni,GazorSadri} define
\bes
\{k_i\,|\,i\in \mathbb{N}\}:=\left \{k\,| \,0\leq k<s, \frac{k+1}{r+2}\not\in \mathbb{N}\right\}\cup\left\{k\,|\, s< k,  \frac{k+1}{r+2} \hbox{ and } \frac{k-s}{r+2} \not\in \mathbb{N}\right\},
\ees
and \(N:= s-\left\lfloor\frac{s}{r+2}\right\rfloor \).

\begin{thm}\label{Par2s1Thm}
There exist maps \(\nu_i(\mu)\) and \(\omega_i(\mu)\) such that the \((2s+1)\)-th level partially extended parametric normal form of \eqref{Eq1} is given by
\begin{eqnarray}\label{2s+1parametric}
v^{(2s+1)}&=&A^1_0+a_r A^{-1}_r +b_s B^0_s+ \sum_{1\leq i \leq r} \nu_i A^{-1}_{i-2}+\sum_{i=1}^N \omega_i B^{0}_{k_i} +\sum_{i=N+1}^{\infty}(b_{k_i}+\omega_i) B^{0}_{k_i}.
\end{eqnarray}
Each parametric Bogdanov-Takens singular system \eqref{Eq1} of the generalized cusp can be transformed into the \((s+1)\)-degree truncated simplest parametric normal form
\be\label{Par2s1}
\dot{x}=a_r y^{r+1}+b_s x y^s+\sum_{1\leq i \leq r} \nu_i y^{i-1}+\sum_{i=1}^N \omega_i x y^{k_i},\qquad
\dot{y}=-x+b_s y^{s+1}+\sum_{i=1}^N \omega_i  y^{k_i+1},
\ee where \(\nu_i(\mu)\) and \(\omega_i(\mu)\) are polynomials. Furthermore, the normal form coefficients \(a_r, b_s\) and polynomial functions \(\nu_i(\mu)\) and \(\omega_i(\mu)\) are uniquely determined by the differential system \eqref{Eq1}.
\end{thm}

\begin{proof}
Since
\begin{equation}
\big[B^0_0,B^0_s\big]=s B^0_s, \qquad \big[-A^0_0,B^0_s\big]=\dfrac{s}2 B^0_s, \qquad Z^0_0 B^0_s= B^0_s,
\end{equation} and the equation \eqref{eq3.7} holds, we can eliminate \(B^0_s\mu^{n}\) for any \(\0\neq\mu\in {\NZ}^p\).

The parametric versions of the formulas in the proof of Theorem \ref{Orb2s1} imply that terms of the form \(B^0_k\mu^\m\) for \(k\neq s\pmod{r+2}\) can be eliminated in the \((s+1)\)-th level. The only remaining terms in the kernel of \(d^{r+|\m|(2r+1), r}\) is \((A^1_0+a_r A^{-1}_r)\mu^\m\). This does not omit any further term due to the equation
\begin{equation}
\left[A^1_0+a_r A^{-1}_r, b_s B^0_s\right]=-s b_s B^0_s - a_r b_s A^{-1}_{r+s}.
\end{equation} The proof is now complete by \cite[Theorem 4.3 and Lemma 4.2]{GazorYuSpec}.
\end{proof}

\section{Bifurcation analysis of truncated normal form when \(r= s=1\)}\label{sec5}

We consider the bifurcation analysis of truncated parametric normal forms for two most generic cases of the generalized cusp case of Bogdanov-Takens singularity, \ie \((r, s)= (1, 1),\) and \((2, 2).\) This section is devoted to the most generic case \(r= s=1\). The local qualitative dynamics of this case are well-known: it consists of a saddle-node, a Hopf and a homoclinic bifurcation.
We consider the 4-jet asymptotic unfolding normal form system
\be\label{r1s1NF}
\dot{x}=\nu_1+\nu_2 x+a_1 y^2+b_1 x y+b_3 x y^3, \quad\qquad \dot{y}=-x+\nu_2 y+ b_1 y^2+b_3 y^4.
\ee
A saddle-node bifurcation of equilibria occurs along the cusp variety
\begin{eqnarray*}
\nu_1=-\dfrac{2+b_1\nu_2\left(12+15{b_1}{\nu_2}-2{{b_1}^2\nu_2}^2\right)-2\sqrt{1+b_1 \nu_2 \left(12+51 b_1\nu_2\right)+{b_1}^3{\nu_2}^3\left(88+51 b_1 \nu_2\right)+{b_1}^6{\nu_2}^6}}{27{b_1}^4}.
\end{eqnarray*}
Assuming that \(\nu_1<0,\) \(a_1=1,\) and \(\nu_2= o(\sqrt{-\nu_1}),\) two local equilibria around the origin for the system \eqref{r1s1NF} are estimated by
\bas\nonumber
&x_\pm=-5{b_1}^9{\nu_1}^3-13{b_1}^6{\nu_1}^2\nu_2-12{b_1}^4b_3{\nu_1}^3\pm{\frac {21{b_1}^7}{8}} {\nu_1}^2\sqrt{-\nu_1}
+\frac{3}{2}{b_1}^5{\nu_1}^2-8{b_1}^3\nu_1{\nu_2}^2\pm2b_1\sqrt {-\nu_1}{\nu_2}^2+b_3{\nu_1}^2&\\\nonumber
&\;\qquad \pm{\frac {45{b_1}^4}{8}}
\nu_1\sqrt{-\nu_1}\nu_2-7b_1b_3{\nu_1}^2\nu_2+\frac{5{b_1}^2}{2}\nu_1\nu_2-{b_1}^3\nu_1 \sqrt{-\nu_1}\pm4{b_1}^2b_3 {\nu_1}^2\sqrt{-\nu_1}-\frac{1}{2}{\nu_2}^3-b_1\nu_1
\mp\sqrt {-\nu_1}\nu_2,&
\\\nonumber
&y_\pm={\nu_1}^2\nu_2
\mp{\frac {231{b_1}^8\nu_1^2\sqrt{-\nu_1}}{128}}\mp{\frac {35{b_1}^5\nu_1\sqrt{-\nu_1}\nu_2}{8}}-2{b_1}^3\nu_1\nu_2\pm\frac{5{b_1}^4\nu_1\sqrt{-\nu_1}}{8}-{\frac {{\nu_2}^4}{\nu_1}}+{\frac {{\nu_2}^5}{{\nu_1}^2}}-{\frac {{\nu_2}^6}{{\nu_1}^3}}\mp\frac{7{b_1}^3b_3\nu_1^2\sqrt{-\nu_1}}{2}+{\nu_2}^3&\\\nonumber
&\qquad -{b_1}^6{\nu_1}^2\pm b_1\sqrt{-\nu_1}\nu_2\mp\frac{9{b_1}^2\sqrt{-\nu_1}{\nu_2}^2}{4}-b_1b_3{\nu_1}^2+\frac{{b_1}^2\nu_1}{2}-{\nu_1}^3-\nu_1{\nu_2}^2\mp b_3\nu_1\sqrt{-\nu_1}\nu_2-\frac{{\nu_2}^2}{2}\mp\sqrt {-\nu_1}.&
\eas
The equilibrium \((x_+, y_+)\) is always a saddle point. The equilibrium \((x_-, y_-)\) is a stable/unstable focus when \(b_1\) is positive/negative before it undergoes a Hopf singularity. Indeed, \((x_-, y_-)\) changes its stability type when the parameters cross the transition set of Hopf bifurcation, that is approximated by:
\be\label{Hopfr1s1}
T_H:= \{(\nu_1, \nu_2): \eta=0\}, \;\; \hbox{ where }\;\; \eta:=\nu_2+\frac{3}{2}b_1\sqrt {-\nu_1}+\frac{3}{16}{b_1}^{3}\nu_1-{\frac {267}{64}}{b_1}^{5}\sqrt {-\nu_1}\nu_1-\frac{5}{2}b_3\sqrt{-\nu_1}\nu_1.
\ee Using a symbolic {\sc Maple} programming (see \cite[Theorem 5.1]{GazorYuFormal} and \cite{YuLeung03,GazorYuSpec}), a three-universal asymptotic unfolding normal form amplitude equation is derived as
\bes
\dot{\rho_-}=-\eta\nu_1{\rho_-}+\left(\frac{3}{16}b_3\nu_1+\frac{3}{128}b_1\right){\rho_-}^3.
\ees The bifurcated limit cycle from \((x_-, y_-)\) is destroyed through a homoclinic bifurcation.
For a sufficiently accurate estimation of the homoclinic transition set, we use the parametric normal form
\be\label{r1s1NF2}
\dot{x}= \tilde{\nu}_1+ \tilde{\nu}_2 x+ \tilde{a}_1 y^2+ \tilde{b}_1 x y+ \tilde{a}_2 y^3, \qquad \dot{y}= -x+\tilde{\nu_2} y+ \tilde{b}_1 y^2,
\ee where the time rescaling \(Z^0_1\) and state transformation generator \(\mathcal{Z}^0_1\) through equations \eqref{ztrans} and \eqref{Z2ln} are not used to simplify \(A^{-1}_2\)-term. This is to enlarge the neighborhood truncated normal form validity for the homoclinic bifurcation control in subsection \ref{sec6}. Next we use the rescaling transformations \(x:=\epsilon^3 x,\) \(y:=\epsilon^2y,\) \(t:=\epsilon^{-1}\tau,\) \(\tilde{\nu_1}:=-\epsilon^4,\) and \(\tilde{\nu_2}:=\epsilon^2(\gamma_0+\gamma_1 \epsilon)\) in order to transform the system \eqref{r1s1NF2} into
\be
\dot{x}=-1+\tilde{a}_1y^2+\epsilon(\gamma_0 x+ \epsilon \gamma_1 x+\tilde{b}_1 x y+\tilde{a}_2\epsilon y^3), \quad\qquad\dot{y}=-x +\epsilon(\gamma_0+ \epsilon \gamma_1 y+ y+\tilde{b}_1 y^2).
\ee Here, the polynomial \(H(x,y):=\frac{1}2x^2-y+\frac{\tilde{a}_1}3y^3\) is a first integral when \(\epsilon=0.\) The roots of the first and second order Melnikov integrals (along the curve \(H(x, y)= \frac{2}{2\sqrt{\tilde{a}_1}}\))
\be
\displaystyle\int^{\frac{2}{\sqrt{\tilde{a}_1}}}_{\frac{-1}{\sqrt{\tilde{a}_1}}} \left(\dfrac{y(1-y)(\gamma_0+\tilde{b}_1 y)}{\sqrt{-\frac23(y-2)}}-(y+1)(\gamma_0+\tilde{b}_1 y) \sqrt{-\frac23(y-2)} \right)dy=-\frac{24}{5}\sqrt2\gamma_0-\frac{36}{7}\sqrt 2\tilde{b}_1,
\ee
\be
\displaystyle\int^{\frac{2}{\sqrt{\tilde{a}_1}}}_{\frac{-1}{\sqrt{\tilde{a}_1}}} \left(\left(\gamma_0 x+\tilde{b}_1 x^2\right)\left(2 \gamma_0+3 \tilde{b}_1 x\right)+3 \tilde{a}_2 x^2+\frac{1}{3}\gamma_1 x\sqrt{-6 x(\tilde{a}_1 x^2-3)}-3 \gamma_1 \frac{(1-\tilde{a}_1 x^2)}{\sqrt{-6 x(\tilde{a}_1 x^2-3)}} x\right)dx
\ee give rise to \(\gamma_0=-\frac{15\tilde{b}_1}{14\sqrt{\tilde{a}_1}},\) \(\gamma_1= -\frac{225}{3136}\sqrt{2}\,\frac{18{\tilde{b_1}}^2+49 \tilde{a}_2}{\sqrt[4]{{\tilde{a}_1}^5}},\) \(\tilde{\nu_1}=-\frac{1}{{\gamma_0}^2} {\tilde{\nu_2}}^2,\) and finally the estimated homoclinic bifurcation set \(T_{HmC}\) in the parameter space \((\tilde{\nu_1}, \tilde{\nu_2})\) is given by
\be\label{HomocV}
\tilde{\nu_2}=\frac{{\sqrt{-2\tilde{a_1}\tilde{\nu_1}
\left(225792{\tilde{a_1}}^2{\tilde{b_1}}^2-3360\sqrt 2{\tilde{a_1}}^{\frac{5}{4}}\tilde{b_1}
 \left(49{\tilde{a}}_{2}+18{\tilde{b_1}}^2 \right) \sqrt [4]{-\tilde{\nu_1}}+25
\sqrt {\tilde{a_1}} \left( 49{\tilde{a}}_{2}+18{\tilde{b_1}}^2 \right)
^2\sqrt {-\tilde{\nu_1}} \right)}}}{-627.2{\tilde{a_1}}^2}.
\ee

\section{Truncated normal form analysis when \(r= s=2\)}\label{sec6}

In this section we discuss the bifurcation analysis for the three-asymptotic unfolding normal form system
\be\label{r2s2Eq1}
\dot{x}= \nu_1+\nu_2 y+\nu_3 x+\nu_4 x y+a_2 y^3 +b_2 x y^2,  \quad\qquad \dot{y}= -x+\nu_3 y+\nu_4 y^2 + b_2 y^3.
\ee The system \eqref{r2s2Eq1} is a \(\Z_2\)-equivariant system with respect to the reflection around the origin, when \(\nu_1:=0\) and \(\nu_4:=0.\) In this section we first consider the \(\Z_2\)-equivariant dynamics and then study a one-parameter dynamical symmetry breaking.

\subsection{\(\Z_2\)-equivariant bifurcation analysis }

We assume that \(\nu_1=\nu_4=0\) hold in equation \eqref{r2s2Eq1}. The origin is an equilibrium with eigenvalues \(\nu_3\pm \sqrt{-\nu_2}.\) On the bifurcation variety
\be\label{SNr2s2nu10} T_{P}:=\{(\nu_2, \nu_3)|\, \nu_2=-{\nu_3}^2\},\ee the origin changes its stability type from a stable/unstable node to a saddle point. The origin is a saddle for \(\nu_2<-{\nu_3}^2<0.\) When \(-{\nu_3}^2<\nu_2<0,\) the origin is an unstable/stable node for \(\nu_3>0\) and \(\nu_3<0,\) respectively. The transition variety \(T_{P}\) indeed corresponds to a pitchfork bifurcation under which two equilibria
\begin{eqnarray}\label{Epmnu40}
E_\pm: \qquad (x_\pm,y_\pm)=\left(\nu_3y_\pm+b_2{y_\pm}^3, \pm \dfrac{\sqrt{2}}{2 b_2}\sqrt{-a_2-2 b_2 \nu_3 + \sqrt{{a_2}^2+4 a_2 b_2 \nu_3 -4 b_2^2 \nu_2}}\right),
\end{eqnarray} bifurcate. These equilibria are always saddle when \(a_2<0.\)

On the variety
\be\label{Focus}
T_F:=\{(\nu_2, \nu_3) |\, \nu_2=0, \nu_3\neq0\},
\ee the stable/unstable node nature of the origin changes to a focus point.
For positive values of \(\nu_2\), the origin undergoes a Hopf bifurcation at
\be\label{HopfOrgn} T_H:=\{(\nu_2, \nu_3) |\, \nu_2>0, \nu_3=0\}\ee
whose a three-degree truncated asymptotic unfolding amplitude equation is
\begin{eqnarray}\label{NFHopf0}
\dot{\rho}&=&\nu_3\rho+\frac{b_2}{2}\rho^3.
\end{eqnarray} Hence for parameters crossing the Hopf bifurcation variety \eqref{HopfOrgn} when \(b_2\nu_3<0,\) one limit cycle bifurcates from and surrounds the origin; see Figures \ref{transnu10bn} and \ref{Fig6(b)}. When \(a_2>0,\) this limit cycle will vanish through a quinary saddle-node bifurcation of limit cycles. However for the case of \(a_2<0,\) this limit cycle will disappear through a heteroclinic cycle bifurcation.

\begin{figure}
\begin{center}
\subfigure[\(a_2= b_2=1\)\label{transnu10bna}]
{\includegraphics[width=.24\columnwidth,height=.2\columnwidth]{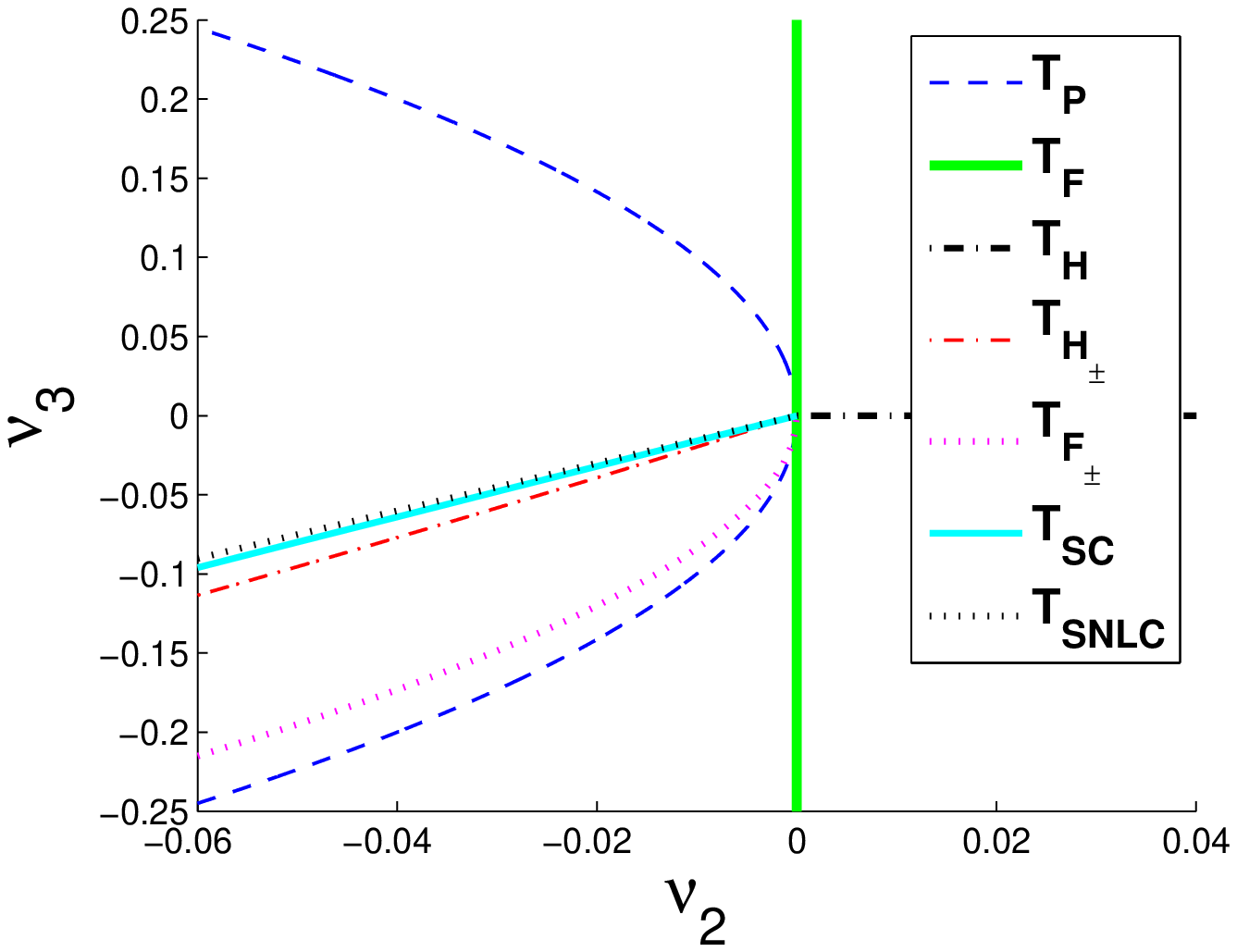}}
\subfigure[\(a_2=1, b_2=-1\)\label{transnu10bnb}]
{\includegraphics[width=.24\columnwidth,height=.2\columnwidth]{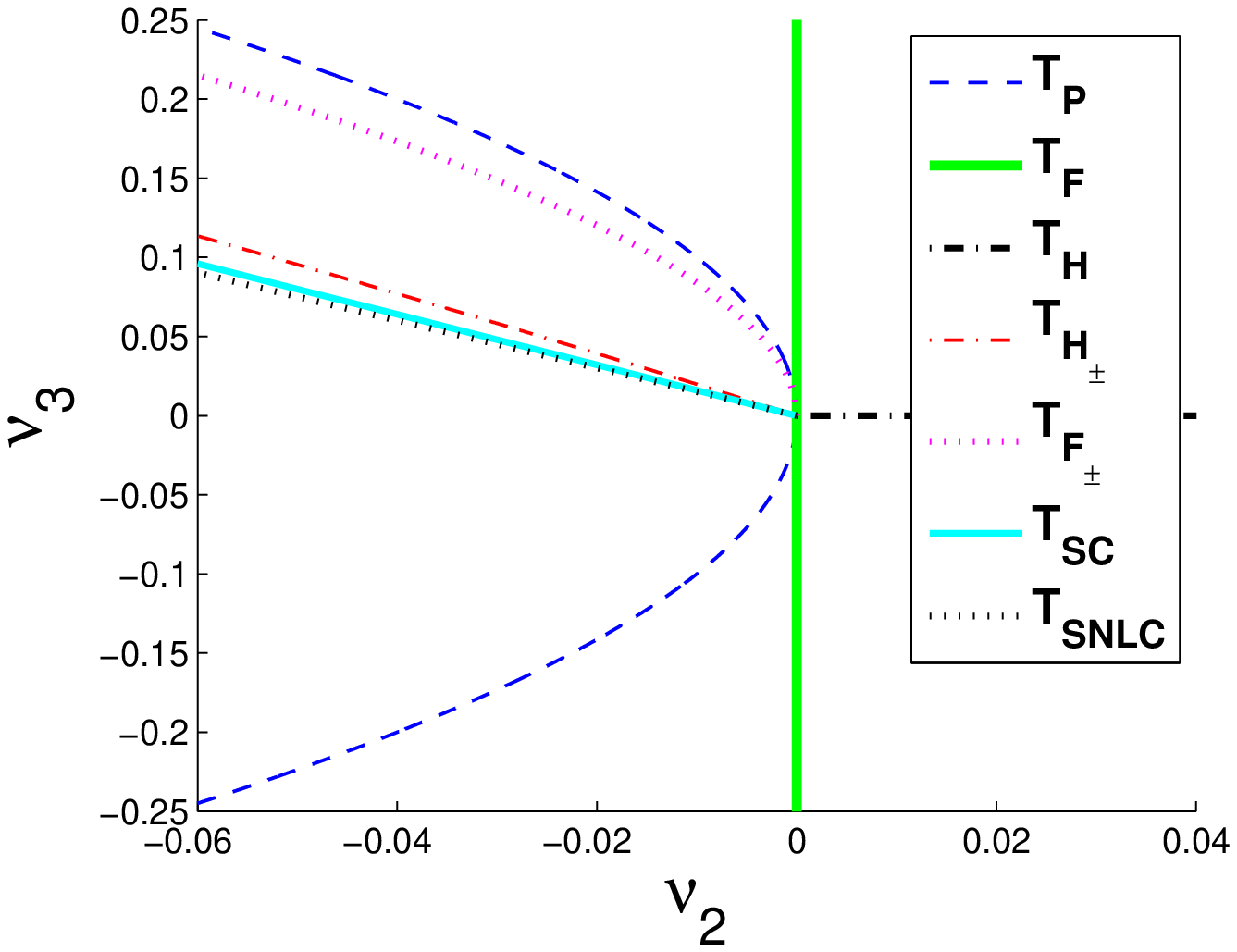}}
\subfigure[\(a_2=-1, b_2=1\)\label{transnu10bnc}]
{\includegraphics[width=.24\columnwidth,height=.2\columnwidth]{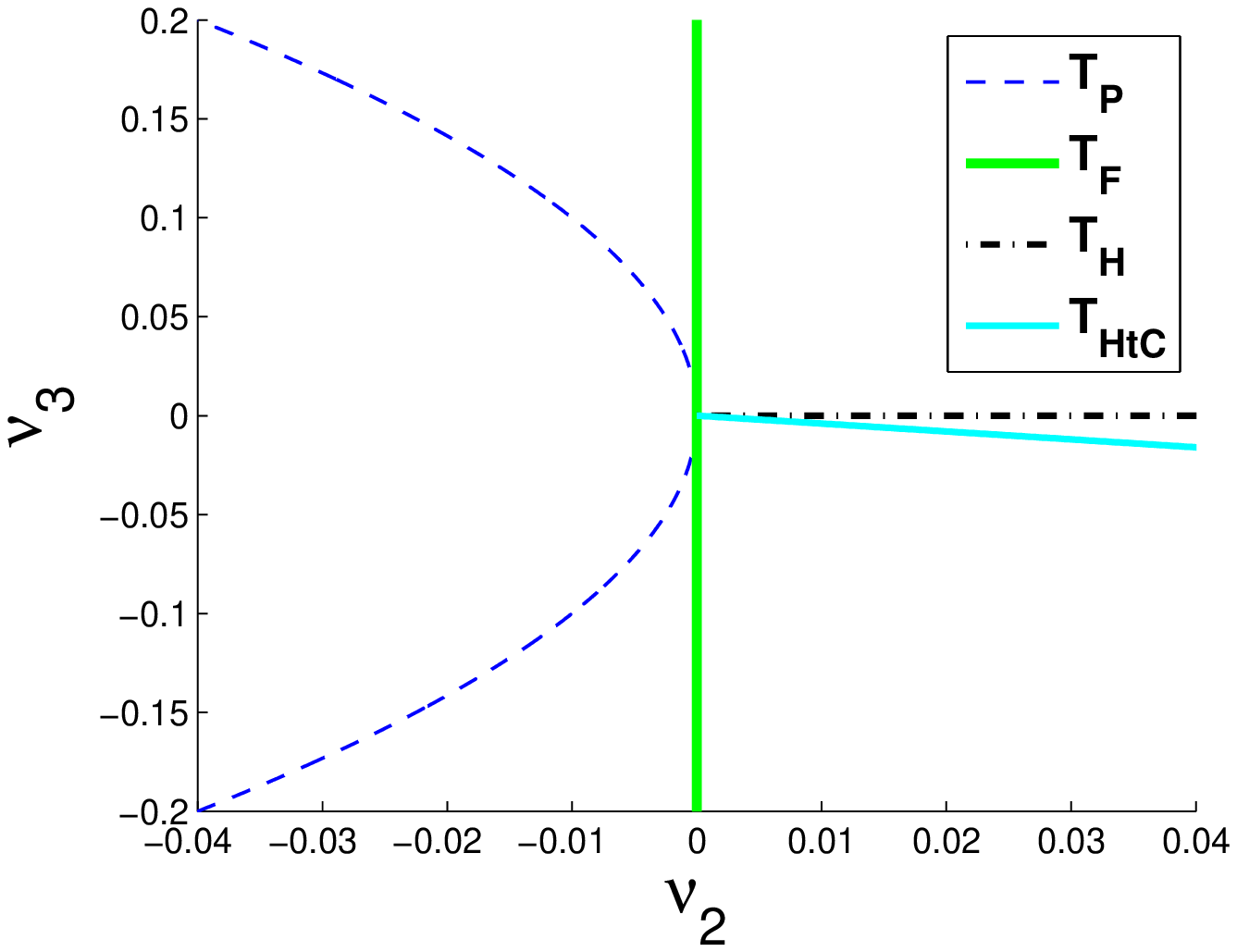}}
\subfigure[\(a_2=-1, b_2=-1\)\label{transnu10bnd}]
{\includegraphics[width=.24\columnwidth,height=.2\columnwidth]{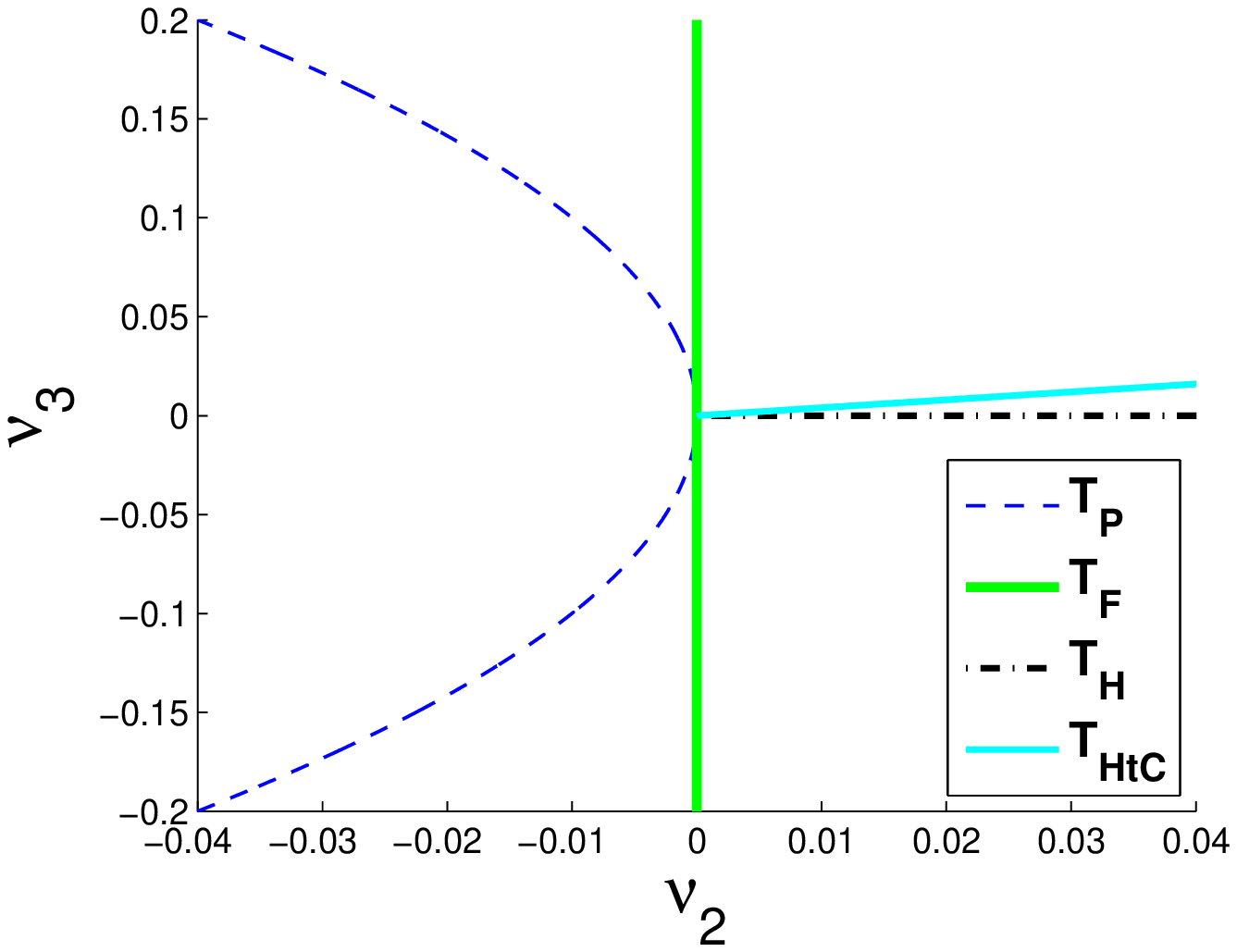}}
\caption{Estimated transition sets for the \(\mathbb{Z}_2\)-equivariant normal form system \eqref{r2s2Eq1}, the case \(r= s=2\)  for \(a_2=\pm1, b_2= \pm1\) and \(\nu_1:=\nu_4:=0.\)}\label{transnu10bn}
\end{center}
\end{figure}

For the case \(a_2>0,\) the equilibria \eqref{Epmnu40} change their stability type from a node to a focus through the curve
\bes
T_{F\pm}:=\left\{ \left(\nu_2, \nu_3\right)|\, 4 {a_2}^2\nu_2 +6 {a_2}^2{\nu_3}^2+4 a_2 b_2{\nu_3}^3  +{b_2}^2{\nu_3}^4=0\right\}.
\ees These equilibria are stable/unstable when \(\eta\) given by
\begin{eqnarray}\label{eta1}
\eta:=2 a_2 \nu_3- 4 b_2 \nu_2-b_2\nu_3^2
\end{eqnarray}
 is positive/negative. In fact each of the equilibria \(E_\pm\) holds a tertiary Hopf bifurcation at
\bes
T_{H\pm}:= \{(\nu_2, \nu_3)|\, \eta=0\}, \ \hbox{ where } \eta \hbox{ is given by equation \eqref{eta1}} \quad \hbox{ and }\quad a_2>0.
\ees
One small limit cycle bifurcates from each of the equilibria \(E_\pm\) on the corresponding variety \(T_{H\pm}\) when \(b_2\eta>0\) and \(a_2>0.\) Each of these bifurcating small limit cycles grows in size until when they get simultaneous {\it collisions} with the origin (a saddle). This results in a {\it saddle-connection} (double homoclinic cycle), that is, a quaternary saddle-connection bifurcation. Therefore, we use the transformations \(\nu_2=-\epsilon^2\), \(\nu_3=\epsilon^2 \gamma\), \(x=\epsilon^2 x\), \(y=\epsilon y\), \(t=\epsilon^{-1} \tau\) and the first and second order Melnikov integrals to estimate the saddle-connection bifurcation variety as
\be\label{SC2+}
T_{SC}:= \left\{(\nu_2, \nu_3)|\, \nu_3=\dfrac{8}{5} \frac{b_2}{a_2} \nu_2+ o\big({|\nu_2|}^{2.5}\big)\right\}\quad\quad\hbox{ and } \quad a_2>0.
\ee
The homoclinic cycles simultaneously break and give birth into a quaternary bifurcation of a limit cycle. This gives rise into two limit cycles surrounding the three equilibria. These two limit cycles disappear through a quinary saddle-node bifurcation of limit cycles at the estimated variety
\bes
T_{SNLC}:= \left\{(\nu_2, \nu_3)|\, \nu_2=\dfrac{125 a_2}{188 b_2} \nu_3\right\}.
\ees

When \(a_2<0,\) the limit cycle bifurcated from the origin grows in size and coalesces with the equilibria \(E_\pm.\) They construct a heteroclinic cycle. The estimated heteroclinic bifurcation variety
\be\label{Hetero}
T_{HtC}:= \left\{(\nu_2, \nu_3)|\, \nu_3=\dfrac{2b_2}{5a_2} \nu_2+ o\big({|\nu_2|}^{\frac{3}{2}}\big)\right\}\quad\hbox{ and } \quad a_2<0,
\ee is derived through the roots associated with the first and second order Melnikov integrals and the rescaling transformations \(\nu_2=\epsilon^2\), \(\nu_3=\epsilon^2 (\gamma_0+ \gamma_1\epsilon),\) \(x=\epsilon^2 x\), \(y=\epsilon y,\) \(t=\epsilon^{-1} \tau.\)

Figures \ref{transnu10bna}--\ref{transnu10bnd} demonstrate the estimated transition varieties for the \(\mathbb{Z}_2\)-equivariant system \eqref{r2s2Eq1} for \(\nu_1:=\nu_4:=0\). Figures \ref{transnu10bna} and  \ref{transnu10bnb} include a pitchfork variety \(T_P\) from which the equilibria \(E_\pm\) are bifurcated, Hopf varieties \(T_H\) and \(T_{H\pm}\) which are respectively associated with the origin and \(E_\pm\), saddle-connection \(T_{SC}\) given in \eqref{SC2+} and the saddle-node bifurcation of limit cycles \(T_{SNLC}.\) Figures \ref{transnu10bnc} and \ref{transnu10bnd} are associated with \(a_2=-b_2=-1\) and \(a_2= b_2=-1,\) respectively. These include a pitchfork bifurcation \(T_P,\) a Hopf variety \(T_H\) from which a limit cycle bifurcates from the origin and, finally a heteroclinic variety \(T_{HtC}\) given by \eqref{Hetero} through which the bifurcated limit cycle vanishes. The transition varieties \(T_F\) and \(T_{F\pm}\) in figures \ref{transnu10bn} represent the changes of stabilities from nodes to foci.

\subsection{One parameter symmetry breaking \(\nu_4\neq0\)}

A small non-zero variation of \(\nu_4\) leads to a \(\mathbb{Z}_2\)-symmetry breaking in the system \eqref{r2s2Eq1}. However, this does not change our estimated transition varieties for the primary bifurcations from the origin, \ie the pitchfork bifurcation of equilibria \(E_\pm\) through the variety \eqref{SNr2s2nu10} and Hopf bifurcation \eqref{HopfOrgn}. The parameter \(\nu_4\) in equation \eqref{r2s2Eq1} forces asymmetric qualitative dynamics and asymmetric formulas for {estimated} \(E_\pm\) and their associated follow-up bifurcation varieties including Hopf bifurcation, homoclinic bifurcation, and limit cycle saddle-node bifurcation varieties. Assuming that
\be\label{Restrct}\nu_2= o(||\nu_3,\nu_4||^2)\quad\hbox{ and }\quad\nu_1:=0,\ee a symbolic root approximation associated with equilibria of the system \eqref{r2s2Eq1} is given by
\begin{eqnarray}\label{Epm}
E_\pm: \qquad (x_\pm,y_\pm):=\left(\nu_3y_\pm+\nu_4{y_\pm}^2+b_2{y_\pm}^3, \mp\dfrac{\left(-2 a_2+2 b_2 \nu_3 +{\nu_4}^2\right)\sqrt{\nu_2-{\nu_3}^2}}{2 a_2 \sqrt{a_2}}\mp\frac{\nu_3 \nu_4}{a_2}\right).
\end{eqnarray}

\noindent Let \(a_2>0.\) Then, the tertiary Hopf bifurcation transition sets associated with \(E_\pm\) are approximated by
\bes
T_{H\pm}:= \{(\nu_2, \nu_3, \nu_4)|\, \eta_\pm=0\},
\ees
where
\bes
 \small{\eta_\pm:=2{a_2}^\frac{7}{2}\nu_3-3 {a_2}^\frac{5}{2} \nu_3 {\nu_4}^2-4b_2 {a_2}^{\frac{3}{2}} (a_2-2b_2 \nu_3-{\nu_4}^2)(\nu_2+{\nu_3}^2)\pm\frac{{a_2}^2}{2}(6 a_2-22 b_2\nu_3-3{\nu_4}^2)\nu_4\sqrt{-\nu_2-{\nu_3}^2}}.
\ees
An estimated five-asymptotic unfolding amplitude normal form equation is given by
\begin{eqnarray}
\dot{\rho}= \eta_\pm(-\nu_2-{\nu_3}^2)\rho-\left(b_2(-\nu_2-{\nu_3}^2)\pm\frac{9}{16}\nu_4 \sqrt{a_2} \sqrt{-\nu_2-{\nu_3}^2}\right)\rho^3-\dfrac{7 a_2 b_2}{128}\rho^5.
\end{eqnarray} Despite the degeneracy of Hopf singularity, the parameter restriction leads to bifurcation of at most one limit cycle from either of \(E_\pm.\) This implies that parameters \((\nu_2, \nu_3, \nu_4)\) with restrictions \eqref{Restrct} are not enough for fully unfolding a Bautin bifurcation around \(E_\pm.\) The system \eqref{r2s2Eq1} are naturally expected to undergo a bifurcation of two limit cycles through a Bautin bifurcation in the vicinity of these equilibria when the restrictions \eqref{Restrct} are removed. For the homoclinic bifurcation transition set, we use the rescaling transformations
\(\nu_2=-\epsilon^2\), \(\nu_3=\epsilon^2 \gamma_3\), \(\nu_4=\epsilon \gamma_4\), \(x=\epsilon^2 x\), \(y=\epsilon y\), and \(t=\epsilon^{-1} \tau.\) These give rise to
\bes
\dot{x}=-y+a_2 y^3+\epsilon(\gamma_4 x y+\gamma_3 x+b_2 x y^2), \qquad \dot{y}=-x+\epsilon(\gamma_4  y^2+\gamma_3 y+b_2 y^3),\nonumber
\ees where \(H(x,y)=\frac{1}{2}x^2-\frac{1}{2} y^2+\frac{1}{4} a_2 y^4\) is a first integral when \(\epsilon=0.\) Hence, the roots of the Melnikov  integrals
\bas
&\displaystyle\int_{\pm\sqrt{\frac{2}{a_2}}}^{0} \left(\dfrac{ 2(1-y^2)\left(\gamma_4 y^2 +\gamma_3 y+b_2 y^3\right)}{ \sqrt{4-2 a_2 y^2}} -\frac{1}{2} y \sqrt{4-2 a_2 y^2} \left(\gamma_4 y+\gamma_3+ b_2 y^2\right)\right) dy=\pm \dfrac{3}{8}\sqrt{2} \pi\gamma_4+\dfrac{32}{15} b_2+\dfrac{4}{3}&
\eas
lead to the quaternary homoclinic bifurcation sets surrounding \(E_{\pm}\). These are estimated by
\bes
T_{HmC_{{\pm}}}:= \left\{(\nu_2, \nu_3,\nu_4)|\, \nu_3=\dfrac{8b_2}{5 a_2}\nu_2\mp \frac{9 \sqrt{2}\pi}{32\sqrt{a_2}}\nu_4\sqrt{-\nu_2}\right\}\quad\hbox{ and } \quad a_2>0.
\ees There is a secondary homoclinic bifurcation on the transition set \(T_{HmC}\) through which the limit cycle bifurcated via the variety \eqref{HopfOrgn} disappear. Similar to the above Melnikov integral computations, an estimated transition set \(T_{HmC}\) in the 3D-parameter space \((\nu_2, \nu_3, \nu_4)\) follows the formula given in equation \eqref{SC2+}.

\begin{figure}
\begin{center}
\subfigure[\( a_2= b_2=1, \nu_4=0.1\)]
{\includegraphics[width=.23\columnwidth,height=.2\columnwidth]{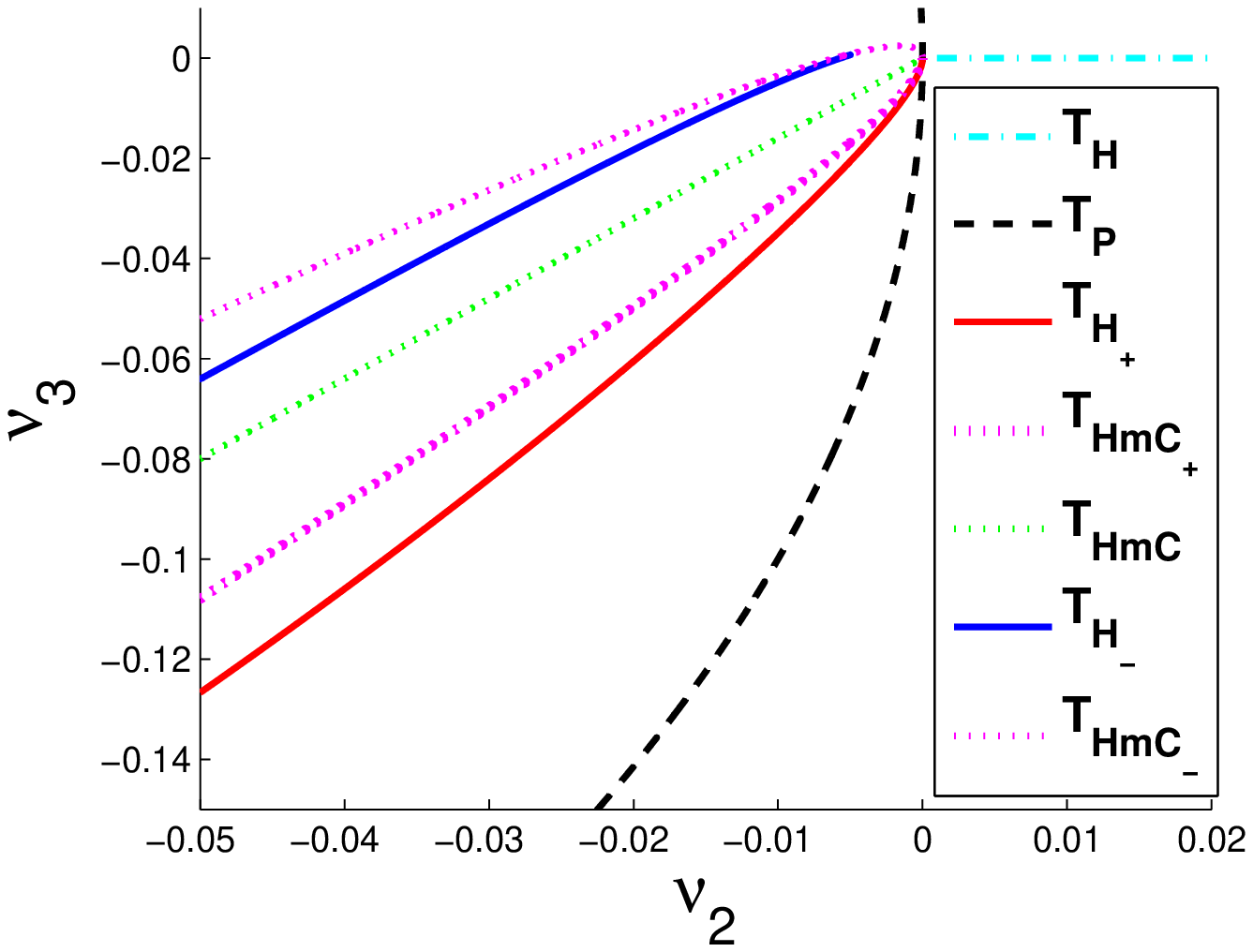}}
\subfigure[\( a_2=- b_2=1, \nu_4=0.1\)]
{\includegraphics[width=.23\columnwidth,height=.2\columnwidth]{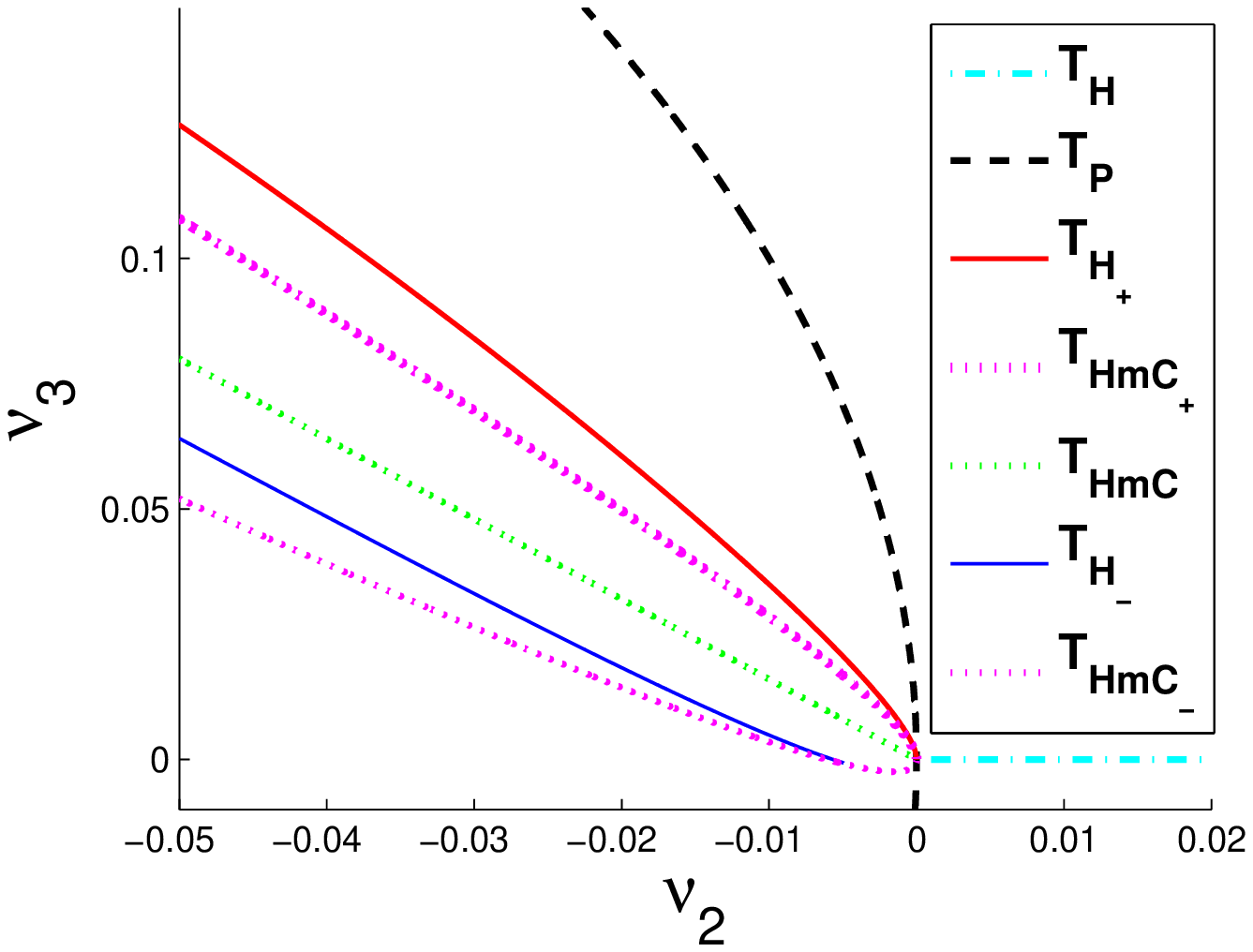}}
\subfigure[\( a_2= b_2=1, \nu_4=-0.1\)]
{\includegraphics[width=.23\columnwidth,height=.2\columnwidth]{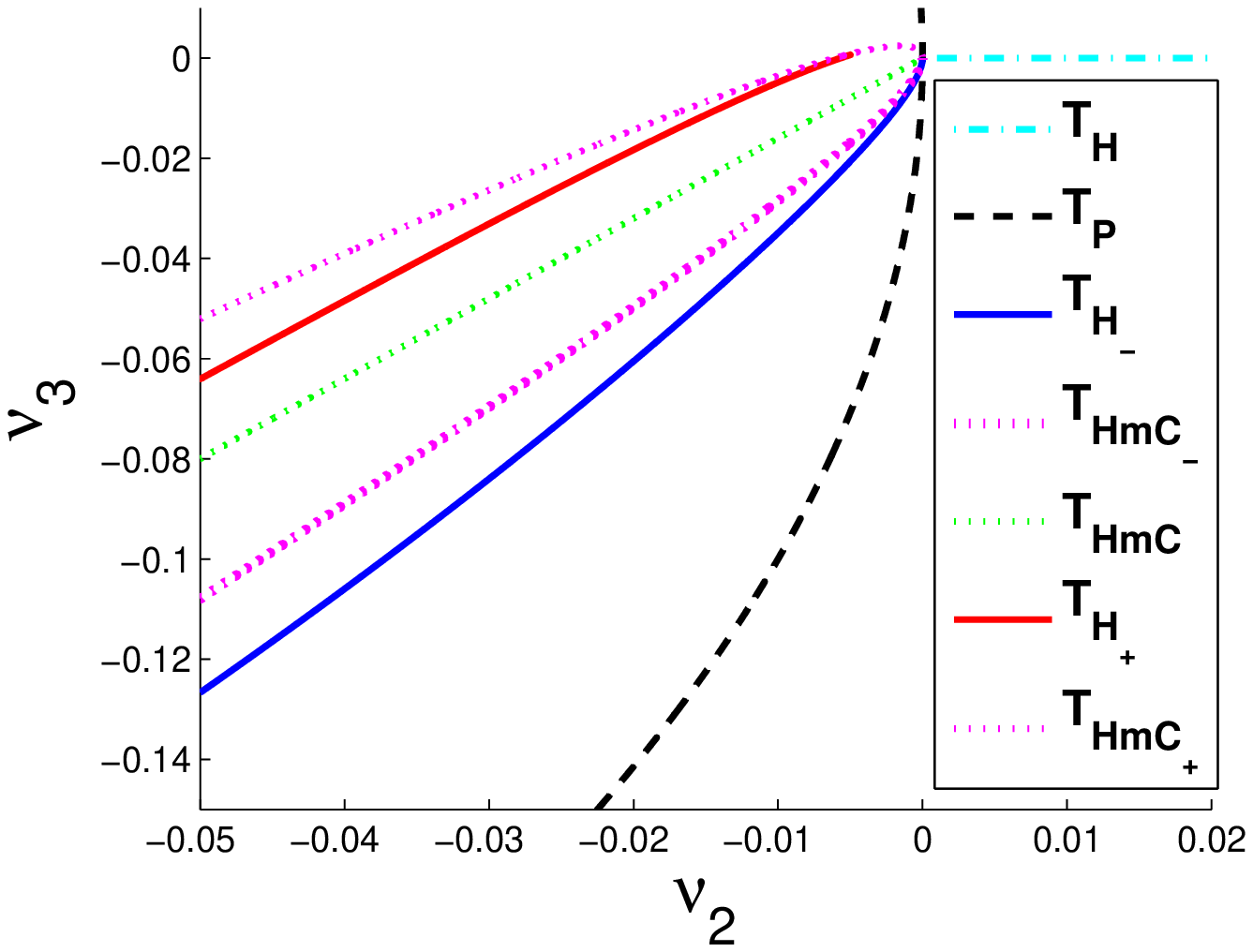}}
\subfigure[\(a_2= -b_2=1,\nu_4=-0.1\)]
{\includegraphics[width=.23\columnwidth,height=.2\columnwidth]{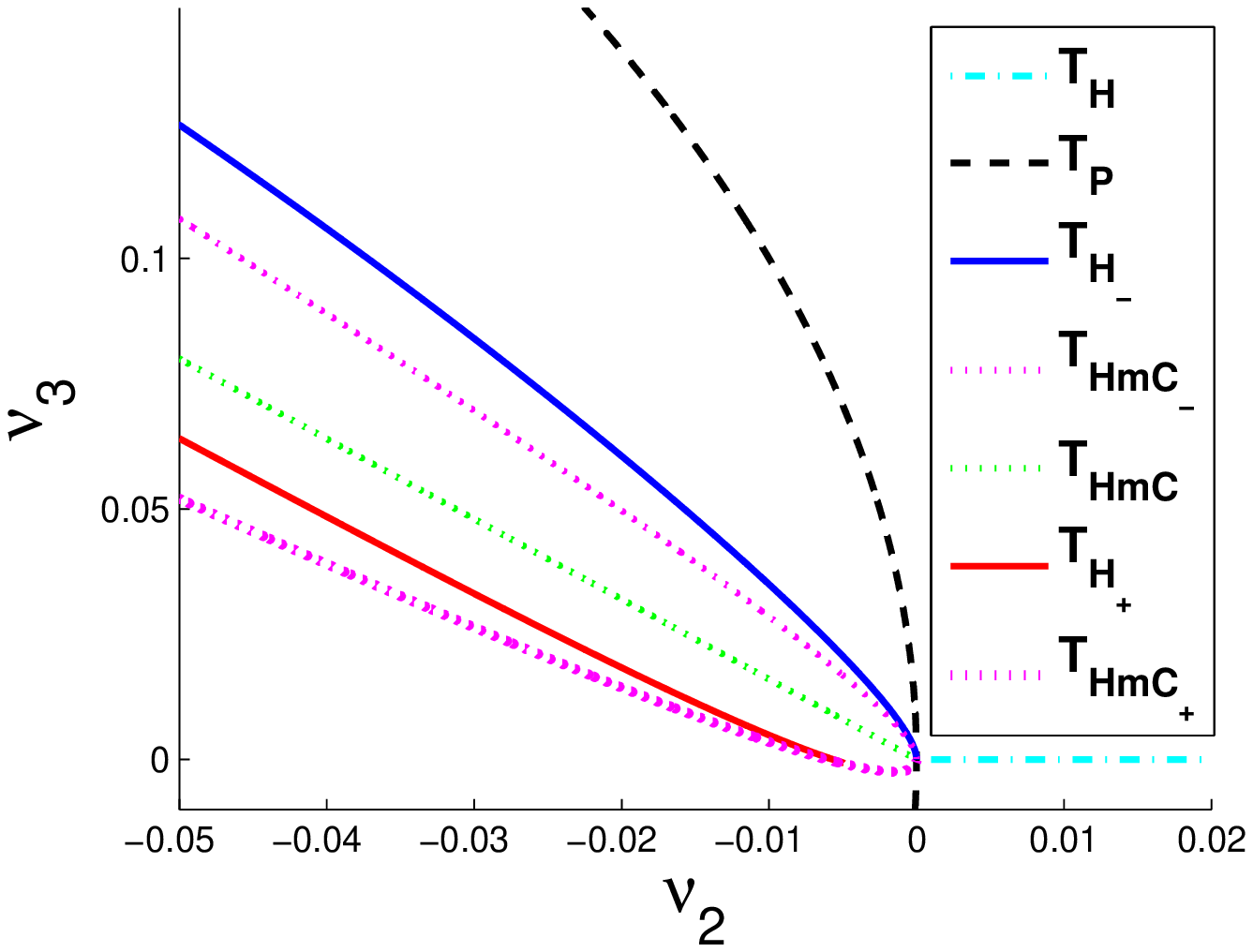}}
\caption{Estimated bifurcation varieties for the normal form system \eqref{r2s2Eq1}, the case \(r= s=2\) and \(\nu_1:=0.\)}\label{nu4pnu10}
\end{center}
\end{figure}

For the case of \(a_2<0,\) similar to the \(\mathbb{Z}_2\)-symmetric case and with its rescaling transformations along with
\(\nu_4=\epsilon \gamma_4,\)
the first and second order Melnikov integrals give rise to the approximated heteroclinic variety given in equation \eqref{Hetero}.

Figure \ref{nu4pnu10} demonstrate the estimated transition sets associated with the system \eqref{r2s2Eq1} when \(\nu_4:=\pm 0.1\) and the restrictions \eqref{Restrct} hold. These figures include a pitchfork bifurcation variety \(T_P\) from which two equilibria \(E_\pm\) are bifurcated from the origin. Hopf bifurcation varieties for the origin and the equilibria \(E_\pm\) are denoted by \(T_H,\) \(T_{H_+}\) and \(T_{H_-},\) respectively. Each of the bifurcated limit cycles from Hopf bifurcation varieties disappear when parameters pass through the homoclinic varieties \(T_{HmC},\) \(T_{HmC_+}\) and \(T_{HmC_-}.\)

\section{Nonlinear bifurcation control }\label{sec7}

This section demonstrates how our parametric normal forms help in the design of efficient control laws
for bifurcation control of two most generic generalized cusp plants of Bogdanov-Takens singularity. 

\subsection{Bifurcation controller design for the case \(r= s=1\)}\label{subsec7}


\begin{figure}
\begin{center}
\subfigure[There are a saddle point and a sink.\label{Fig5(a)}]
{\includegraphics[width=.18\columnwidth,height=.15\columnwidth]{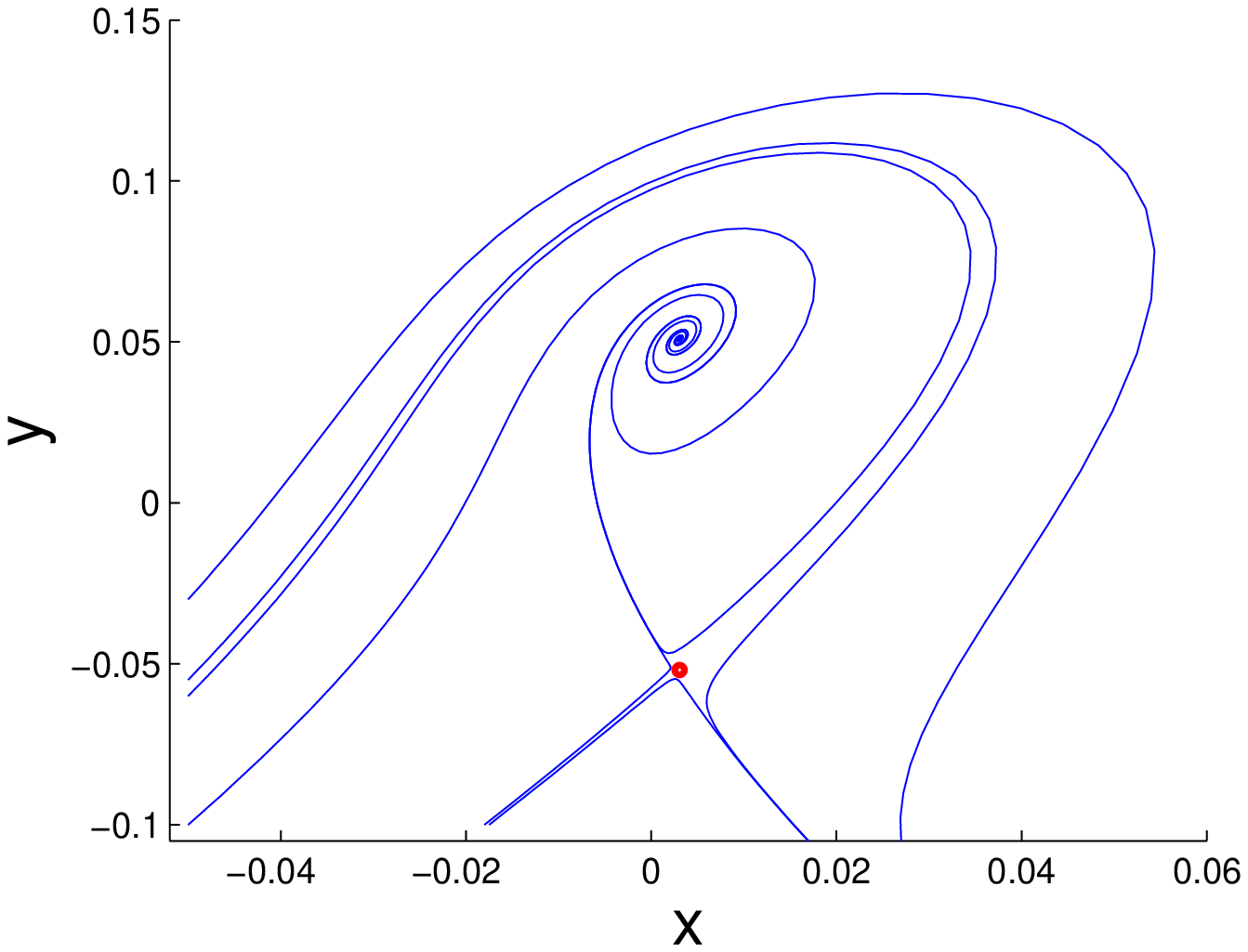}}
\subfigure[A spiral source, a saddle and a stable limit cycle.]
{\includegraphics[width=.18\columnwidth,height=.15\columnwidth]{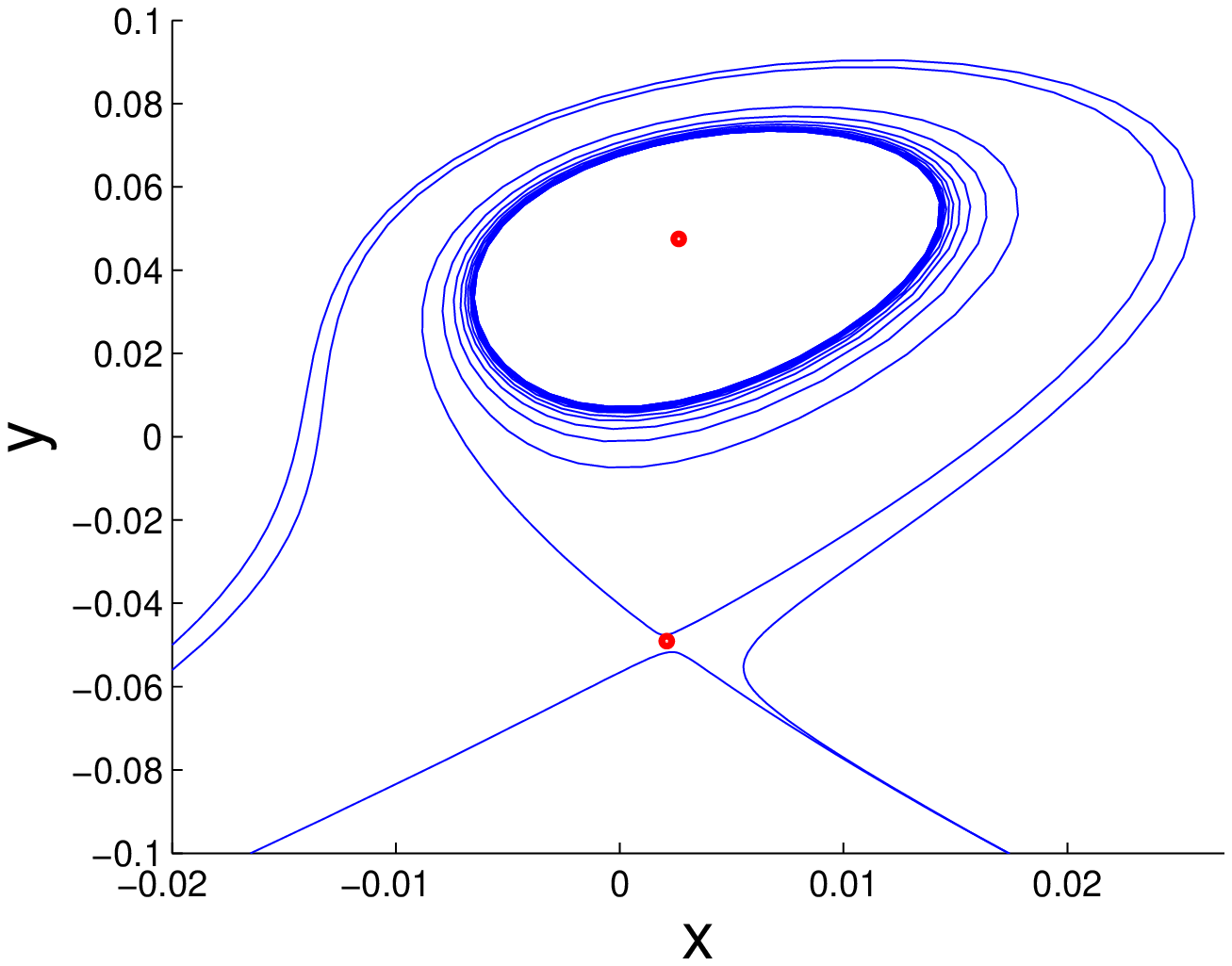}}
\subfigure[There are a saddle and a spiral source. ]
{\includegraphics[width=.18\columnwidth,height=.15\columnwidth]{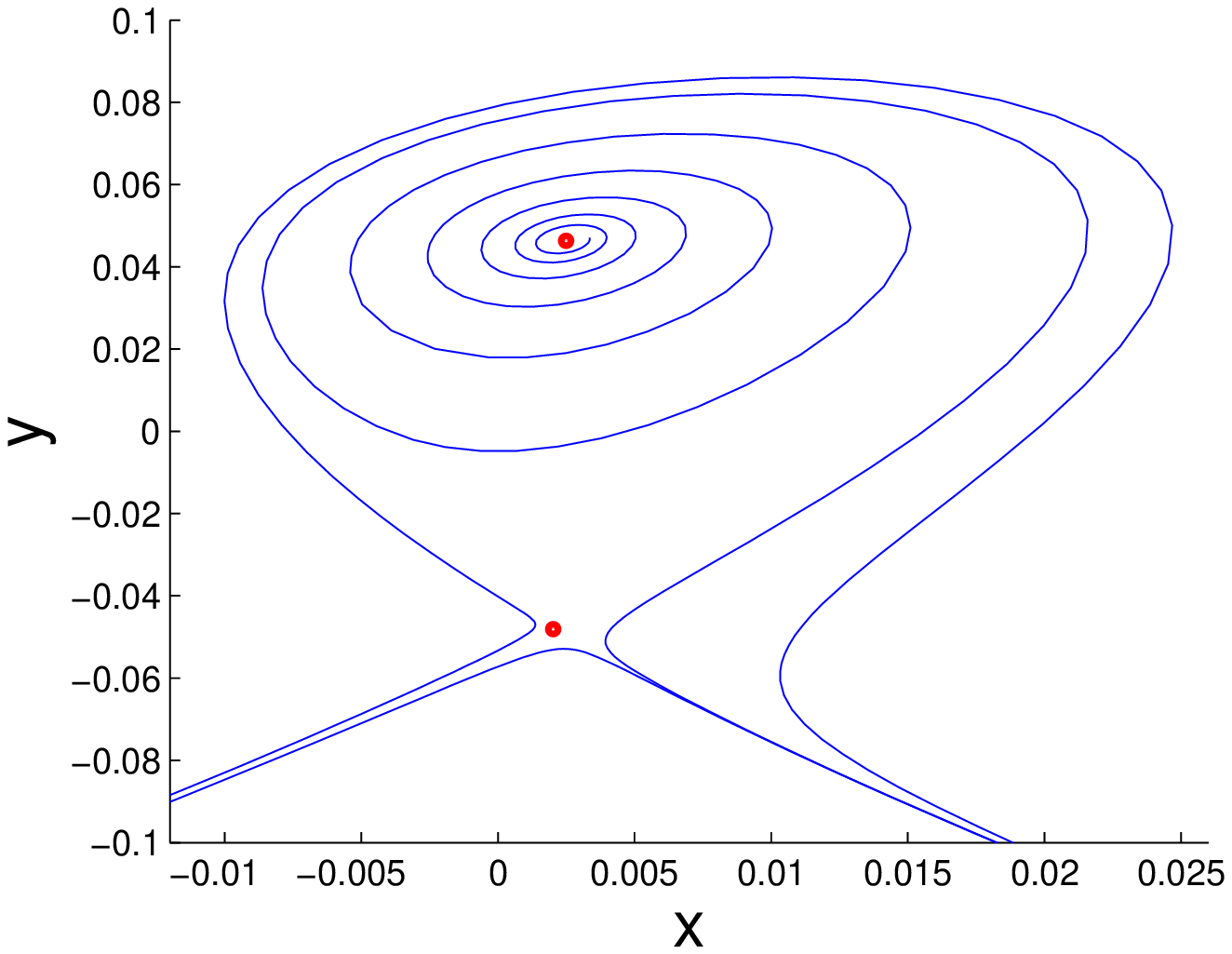}}
\subfigure[Six orbits starting from the line \(x=-0.04.\) \label{Fig5(d)}]
{\includegraphics[width=.18\columnwidth,height=.15\columnwidth]{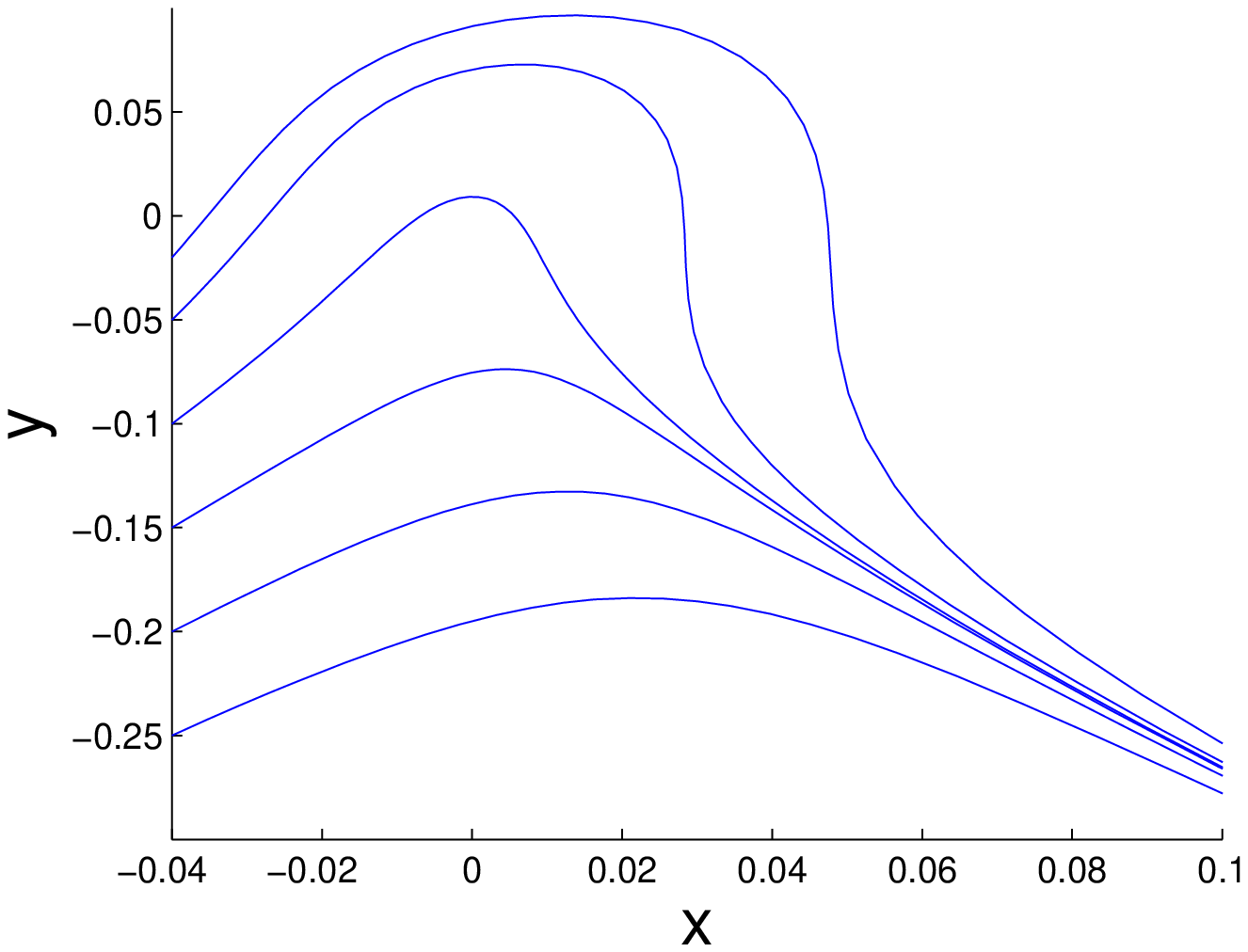}}
\subfigure[Highly accurate numerical controller transition sets.\label{transexample}]
{\includegraphics[width=.24\columnwidth,height=.15\columnwidth]{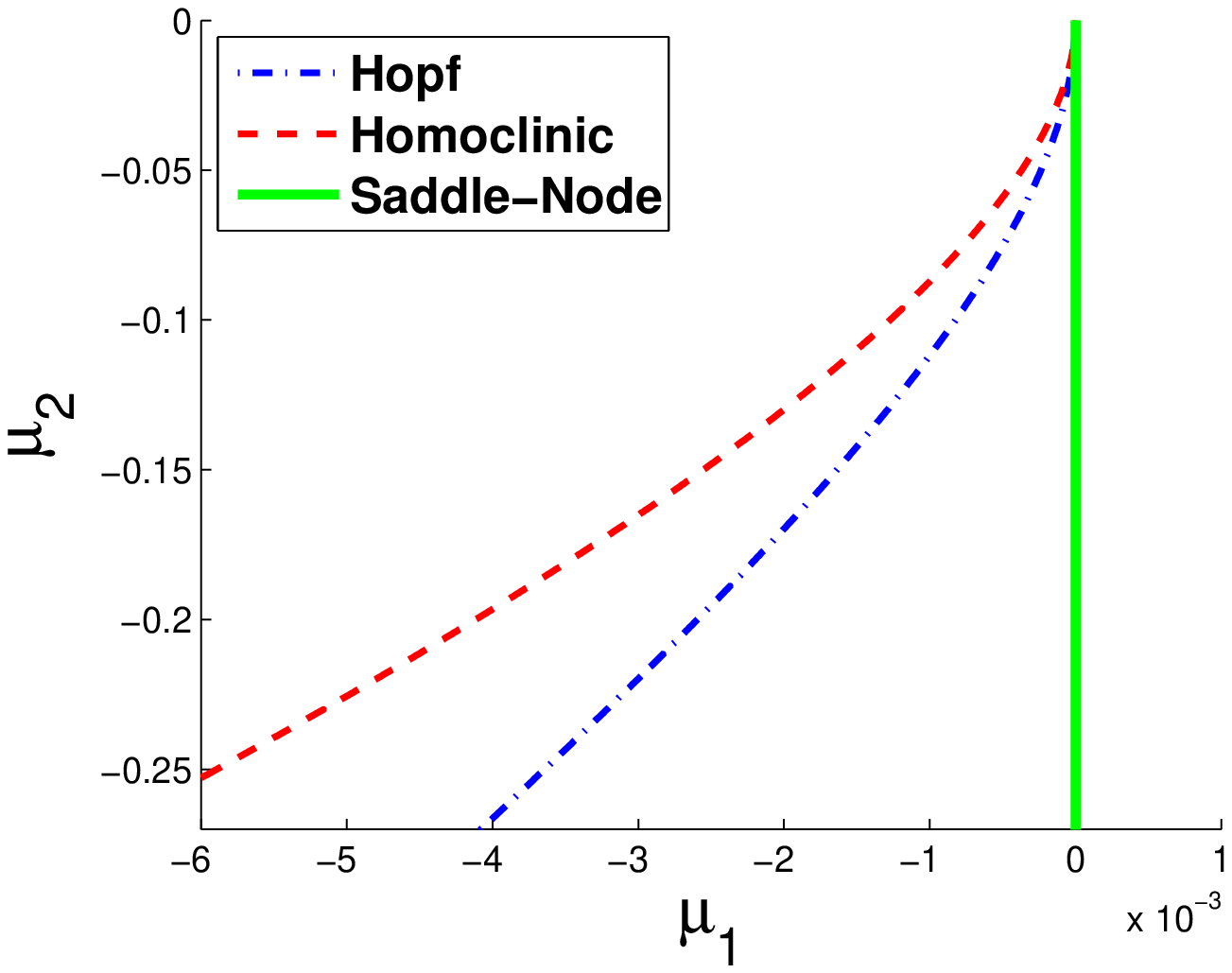}}
\caption{Controlled numerical transition sets and phase portraits of the system \eqref{BfControlr1s1} when \(d_1:= d_2:= d_4:= 1, d_3:=2, d_5:=3, d_6:=\frac{19}{28}\) and \(r=s=1.\) Figures \ref{Fig5(a)}-\ref{Fig5(d)} are associated with regions (a)-(d) in Figure \ref{transexample}, respectively.}
\label{BifContr1s1}
\end{center}
\end{figure}

Consider the quadratic-jet truncated controlled plant
\ba\label{BfControlr1s1}
\dot{x}&:=& d_1y^2+d_2xy+d_3x^2+ u_1, \qquad \dot{y}:=-x+d_4y^2+d_5xy+d_6x^2+ u_2,
\ea  with possible multi-input linear controllers
\bes
u_1:= \mu_1+\mu_2x+\mu_3y \qquad \hbox{ and } \qquad u_2:=\mu_4+\mu_5x+\mu_6y,
\ees where \(\mu_i\) for \(i=1, \ldots, 6\) are referred by controller inputs. Then, by \cite[Proposition 2]{GazorMoazeni} we obtain
\bes
r=s=1, \qquad \hbox{ when } \quad a_1= d_1\neq 0 \quad \hbox{ and } \quad b_1=\frac{1}{3}(d_2+2 d_4)\neq0.
\ees
\noindent Assuming these generic conditions, the parameter \(\mu_1\) and either of the parameters \(\mu_2, \mu_3, \mu_4,\) and \(\mu_6\) can play the role of distinguished parameters. For an instance we take \((\mu_1, \mu_2)\) as two distinguished parameters of the system \eqref{BfControlr1s1} and set \(\mu_3:= \mu_4:= \mu_5:= \mu_6:= 0.\) We remark that a different unfolding leads to slightly different dynamics than the case presented here; for example see the transition set in figure \ref{section72} and the bifurcation varieties given in figure \ref{Fig14(a)}. Yet our bifurcation control analysis is sufficient for a comprehensive study of all these cases; also see \cite[Proposition 6.2]{GazorSadri}.

For briefness in the derived formulas, we choose \(d_5:=3\) and  \(d_1:=d_2:=d_4:=1.\) Hence \(a_1=1\) and  \(b_1=1.\) Then, the four-degree truncated third level partially extended parametric normal form is given by equation \eqref{r1s1NF} where
\begin{eqnarray}\nonumber
\nu_1&=&\mu_1-\frac{477+366d_3- 84 d_6-80 d_3 d_6 +32 {d_3}^2}{30}{\mu_1}^2+2\mu_1\mu_2,\\
\nu_2&=&-{\frac {15}{4}}\mu_1-d_3\mu_1+\frac{1}{2}\mu_2-\frac{477+351 d_3+36 d_6+32 {d_3}^2-80 d_3 d_6}{240}\mu_1\mu_2+\frac{1}{4}{\mu_2}^2\\\nonumber
&&+\frac{239913+222744
d_3+70344 d_6+29238{d_3}^2+13680 d_3 d_6-1040{d_3}^3+3200 {{d_3}}^2d_6}{9600}{\mu_1}^2, \\\nonumber
b_3&=&{\frac {393}{200}}-{\frac {7}{25}}d_3-{\frac {51}{50}}d_6-\frac{4}{25}{d_3}^2+\frac{2}{5}d_3 d_6.
\end{eqnarray}
\noindent Thus, the associated bifurcation transition sets \(T_{SN}\) and \(T_{Hopf}\) are given by
\(T_{SN}=\{ (\mu_1, \mu_2)\,|\,\mu_1=0\},\) and
\bas\nonumber
&T_{Hopf}=\Big\{\big(\mu_1, \mu_2\big)\big |&
\frac{3}{2}\sqrt {-\mu_1}+\dfrac{1}{2}\mu_2-d_3\mu_1-{\frac {57}{16}}\mu_1+\frac{2907+816 d_{{6}}+320 d_6 d_3-224 d_3-128{d_3}^2}{320}{(-\mu_1)^\frac{3}{2}}\qquad\\\nonumber
&&+\frac{239913+70344d_{{6}}+
222744d_3+13680d_6 d_3+29238{d_3}^2+3200{d_3}^2d_6-1040{d_3}^3}{9600}{{\mu_1}^{2}}\\\nonumber
&&
-\frac{477+36 d_6+351 d_3-80 d_6 d_3+32{d_3}^2}{240}\mu_1\mu_2+\frac{1}{4}{\mu_2}^2=0 \Big\}.
\eas
In order to derive a sufficiently accurate transition set \(T_{HmC},\) we compute the parametric normal form \eqref{r1s1NF2}, where
\begin{eqnarray}
\tilde{a_1}&=&a_1=d_1, \qquad \tilde{b_1}=b_1=\frac{1}{3}(d_2+2 d_4), \qquad \tilde{a_2}=-\frac{1}{9}(9d_1d_5-5d_2d_4+{d_2}^2+4{d_4}^2), \\\nonumber
\tilde{\nu_1}&=&\mu_1+{\frac {1}{60}}{\mu_1}^2\left(459-492 d_3+168 d_6+160 d_3 d_6 -64 {d_3}^2\right),\\\nonumber
\tilde{\nu_2}&=&{\frac {1 }{9600}}{\mu_1}^2\left(35370-58320
d_3+1440 d_6+4716 {d_3}^2+6480 d_3 d_6-3440{d_3}^3+3200{{d_3}}^2d_6 \right) \\\nonumber
&&-{\frac {1}{240}}\mu_1\mu_2 \left( -567+156 d_3+36 d_6+32 {d_3}^2-80 d_3 d_6\right)+\frac{1}{4}{\mu_2}^2-{\frac {3}{4}}\mu_1-d_3\mu_1+\frac{1}{2}\mu_2.\nonumber
\end{eqnarray}
Let \(\mathcal{C}:= (64{d_3}^2-160d_3d_6+492d_3-168d_6-459){\mu_1}^2-60\mu_1.\) Hence, \(T_{HmC}\) in \((\mu_1, \mu_2)\)-space follows
\begin{small}
\begin{eqnarray*}&& \frac{\sqrt{5}}{6272}\sqrt{150528\mathcal{C}+19264\sqrt[4]{15^3}\mathcal{C}^\frac{5}{4}+9245\sqrt {15}\mathcal{C}^\frac{3}{2}}-\left(\Big({\frac {189}{80}}+\frac{1}{3}d_3d_6-{\frac {13}{20}}
d_3-\frac{2}{15}{d_3}^2-\frac 3{20}d_6\Big)\mu_1+\frac{1}{2}\right)\mu_2
\\&=&(\frac{3}{4}-d_3)\mu_1+\left( \frac{1}{3}{d_3}^2d_6-\frac {43}{120}{d_3}^3
 +\frac {27}{40} d_3d_6+\frac {393}{80}{d_3}^2
 -\frac {243}{40}d_3+\frac 3{20}d_6+\frac {1179}{320}
 \right){\mu_1}^2.
\end{eqnarray*}
\end{small}
For a numerical simulation, we choose
\be\label{di_s} d_1:=1, \quad d_2:=1, \quad d_3:=2, \quad d_4:=1, \quad d_5:=3, \quad d_6:=\frac{19}{28},\ee
and obtain the transition varieties depicted in Figure \ref{transexample}.
By choosing the input parameters \((\mu_1, \mu_2)\) as
\bes
(-0.002, -0.18), \; (-0.002, -0.15), \; (-0.002, -0.11), \; (0.001, -0.15)
\ees from regions (a)-(d) in Figure \ref{transexample}, we obtain the controlled phase portraits in Figures \ref{Fig5(a)}-\ref{Fig5(d)}, respectively.

\begin{figure}
\begin{center}
\subfigure[\(d_2:=d_3:=d_6:=1, a_2=b_2=1\)\label{Fig6(a)}]
{\includegraphics[width=.28\columnwidth,height=.22\columnwidth]{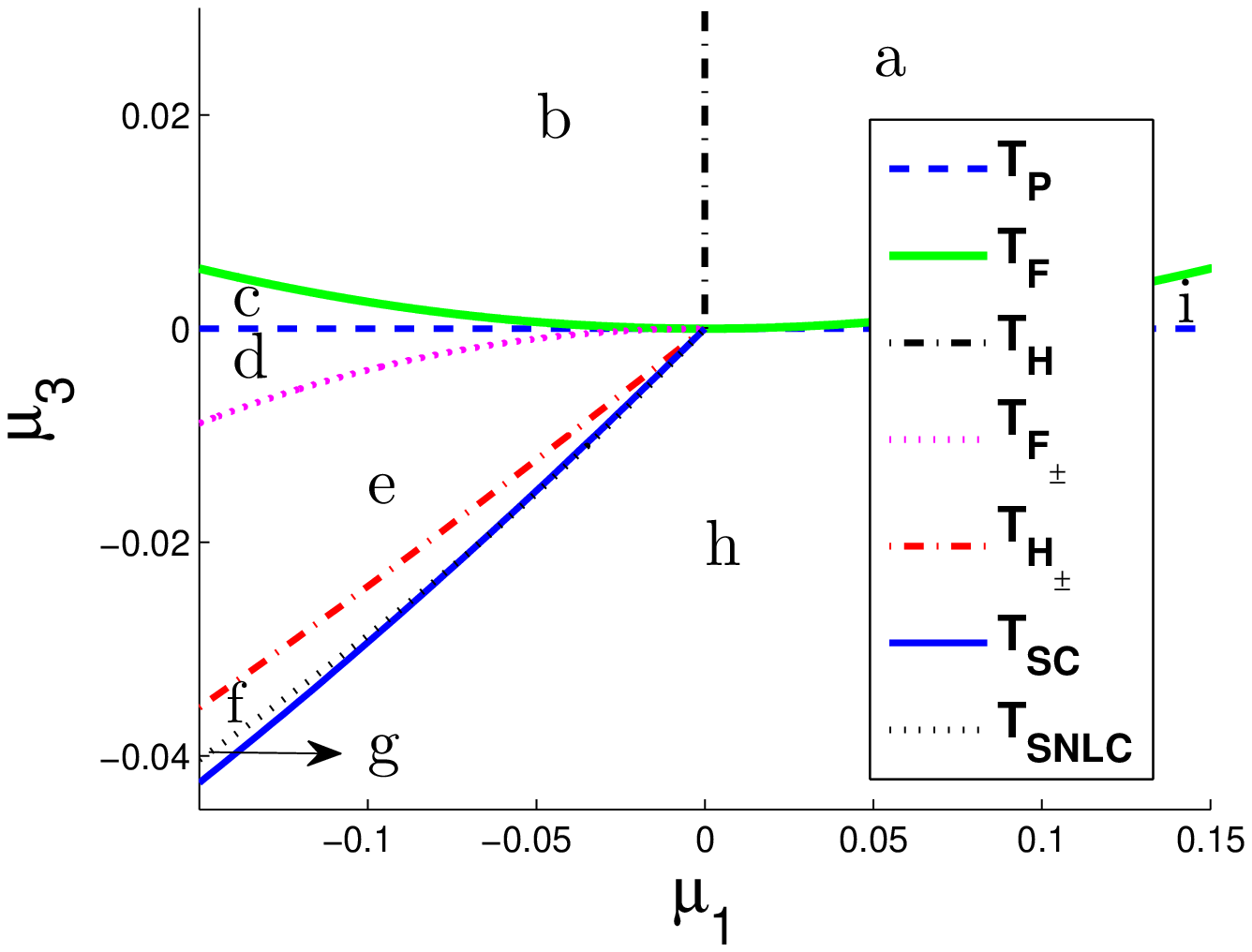}}
\subfigure[\( d_2:=1, d_3:=d_6:=-1, a_2=- b_2=1\)\label{Fig6(b)}]
{\includegraphics[width=.3\columnwidth,height=.22\columnwidth]{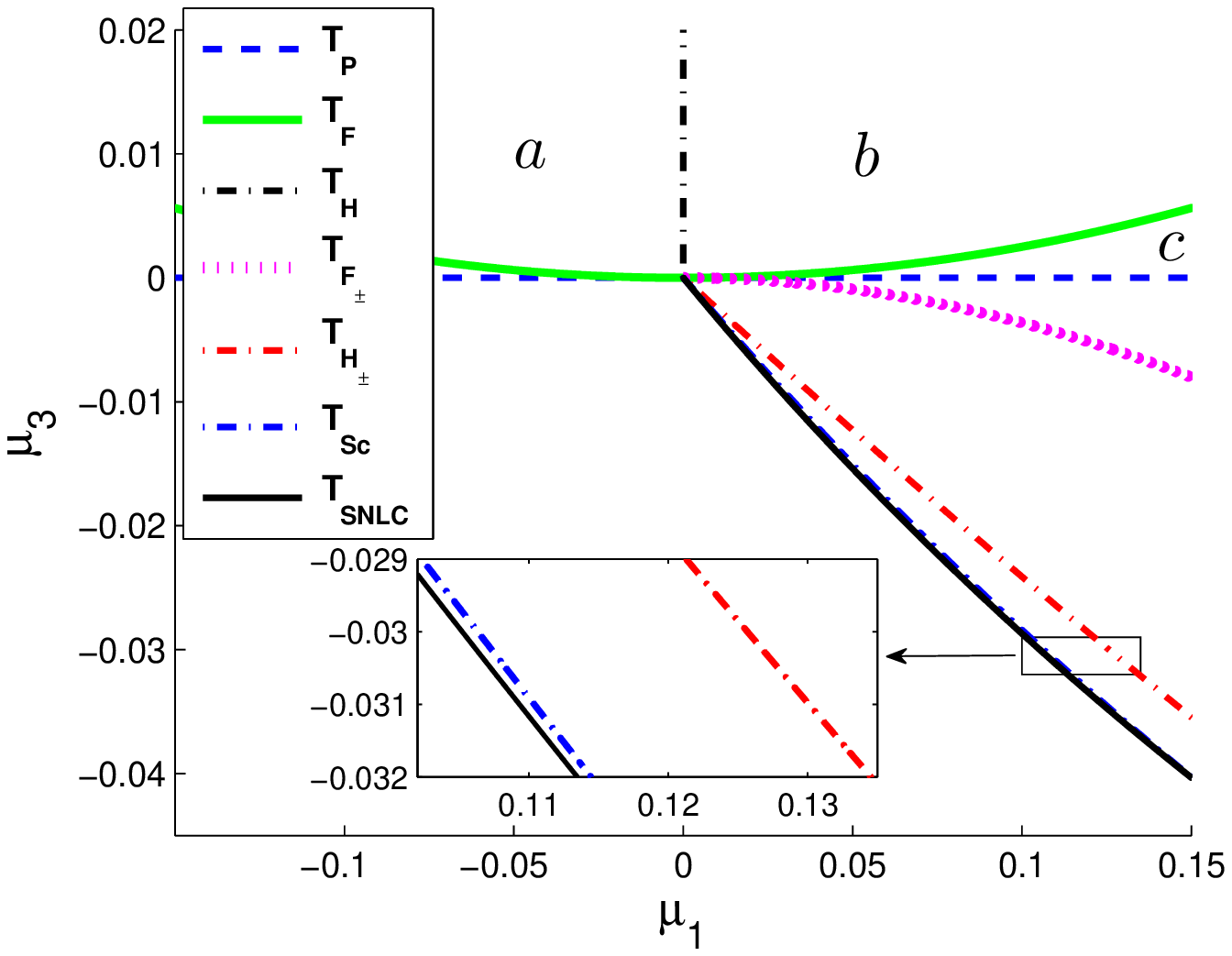}}
\subfigure[\(d_3:=d_6:=-d_2:=1, b_2=-a_2=1\) \label{Fig6(c)}]
{\includegraphics[width=.2\columnwidth,height=.22\columnwidth]{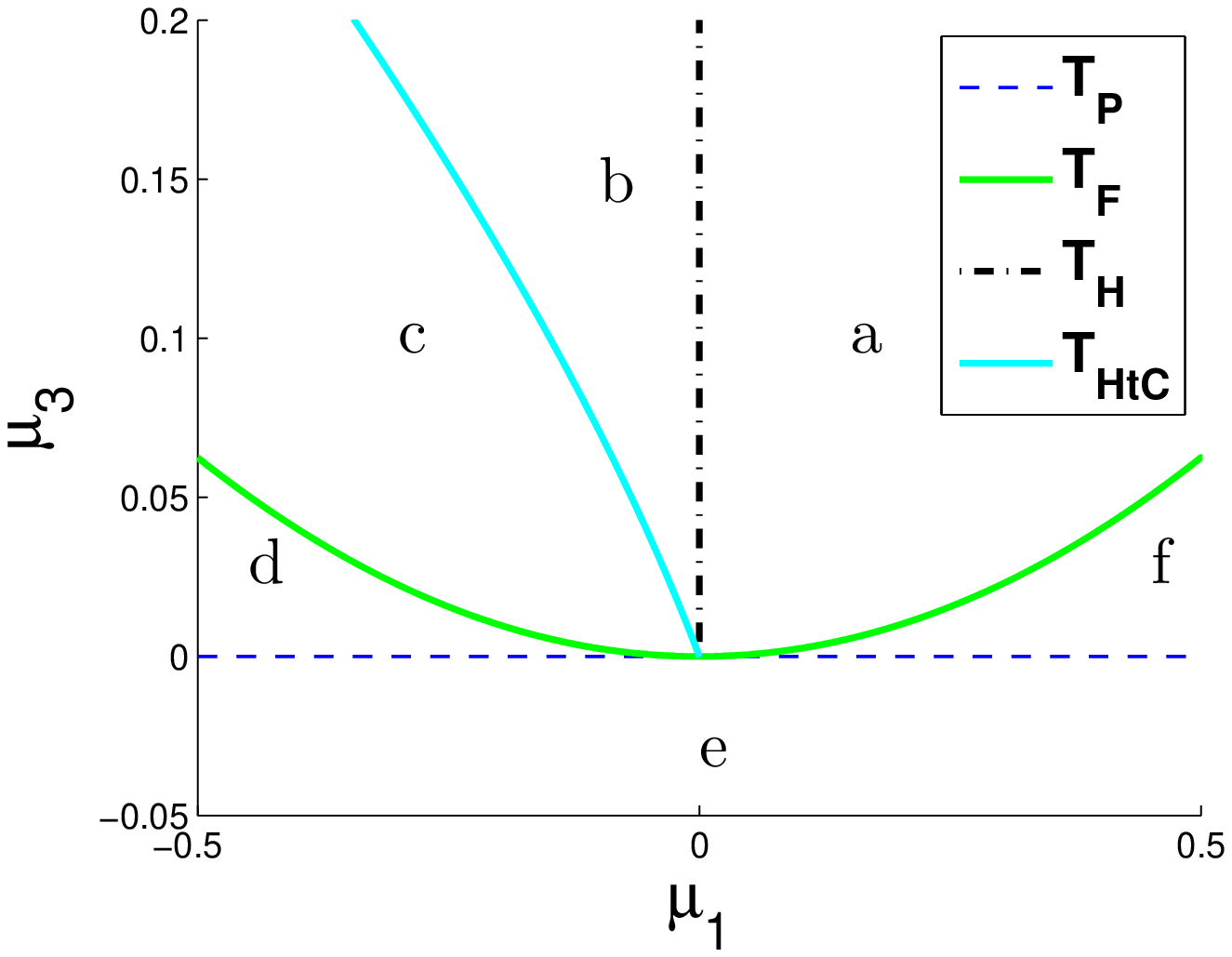}}
\subfigure[\(d_2:=d_3:=d_6:=-1, a_2= b_2=-1\)\label{Fig6(d)}]
{\includegraphics[width=.19\columnwidth,height=.22\columnwidth]{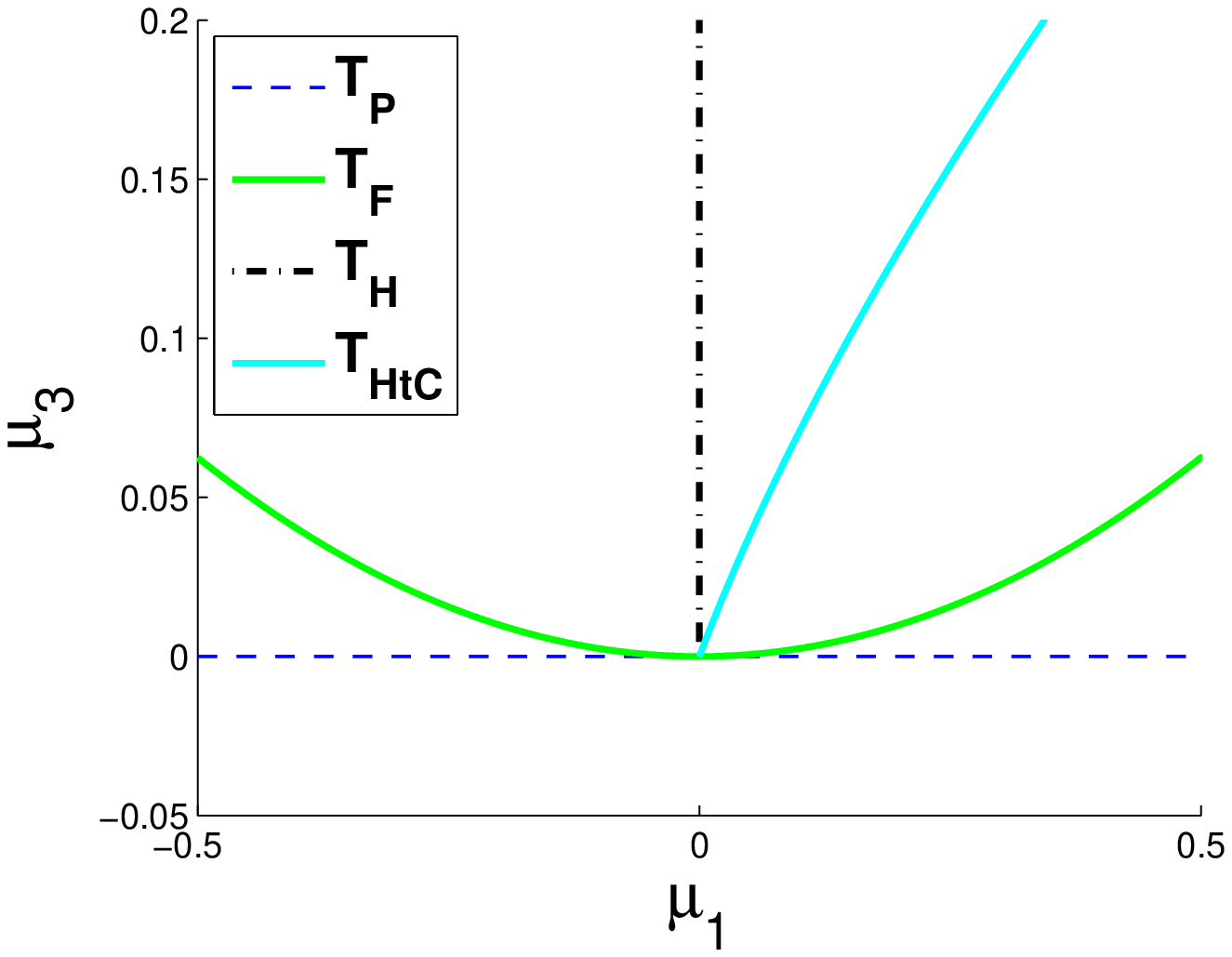}}
\caption{Numerical controller transition varieties for \(\mathbb{Z}_2\)-equivariant system \eqref{BfControlr2s2} when \(\mu_i:= 0\) for \(i\neq 1, 3,\) and \(r=s=2.\)} \label{ABCD}
\end{center}
\end{figure}

\subsection{Bifurcation controller design for the case \(r= s=2\)}

Although the generalized cusp case \((r=s=2)\) of Bogdanov-Takens singular system does not necessarily have to be a \(\mathbb{Z}_2\)-equivariant system, yet this case falls within the most generic case of \(\mathbb{Z}_2\)-equivariant systems with a Bogdanov-Takens singularity. Hence the \(\mathbb{Z}_2\)-equivariant system and its symmetry breaking are expected to occur more often in engineering problems than the occurrence of the case \((r=s=2)\) in non-\(\mathbb{Z}_2\)-equivariant engineering problems. Thus, in this subsection we only consider the cubic-jet of a  \(\mathbb{Z}_2\)-equivariant (controlled) plant
\be\label{BfControlr2s2}
\dot{x}:= d_1 x^3+d_2 y^3+d_3 x y^2+d_4 y x^2+u_1, \qquad\qquad \dot{y}:=-x+d_5 x^3+d_6 y^3+d_7 x y^2+d_8 y x^2+u_2,
\ee with possible multi-input quadratic controllers
\be\label{u1u2}
u_1:=\mu_1 x+\mu_3y+ \mu_5x^2+\mu_7xy+\mu_9y^2, \qquad u_2:=\mu_2 x+\mu_4 y+ \mu_6x^2+\mu_8xy+\mu_{10}y^2.
\ee This system is \(\mathbb{Z}_2\)-equivariant with respect to reflection around the origin when \(\mu_i=0\) for \(i\geq 5,\) while the controller parameters \(\mu_i\) for \(i\geq 5\) may contribute into the \(\mathbb{Z}_2\)-symmetry breaking. Our approach can be easily applied to non-\(\mathbb{Z}_2\)-equivariant plants. In particular we could readily include quadratic (in state variables) terms with non-zero constant coefficients in the plant associated with the system \eqref{BfControlr2s2}. For briefness, the latter is skipped in this paper.

\begin{figure}
\begin{center}
\subfigure[There are three orbits starting from the \(y\)-axis. \label{Fig7(a)}]
{\includegraphics[width=.19\columnwidth,height=.17\columnwidth]{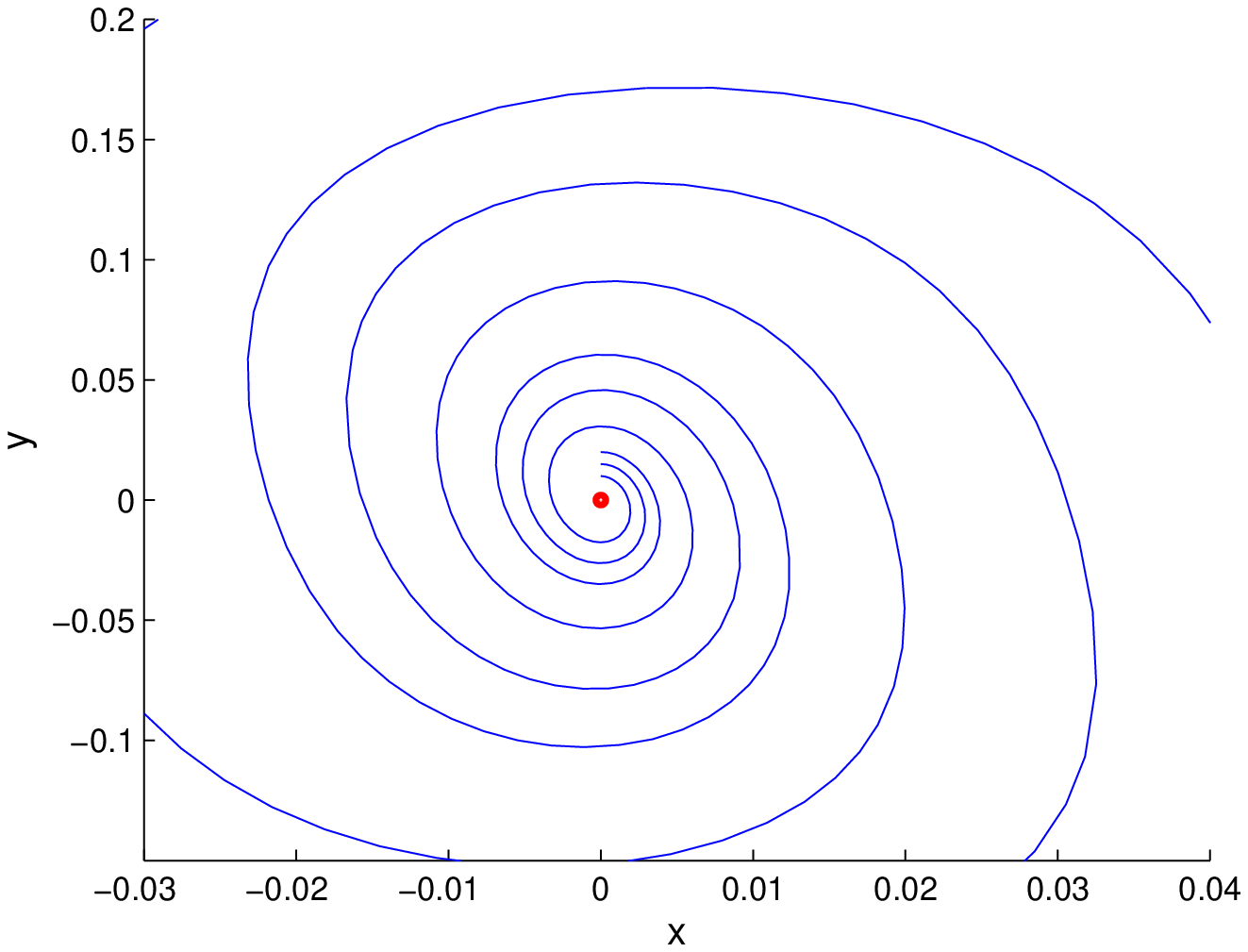}}
\subfigure[A spiral sink is surrounded by an unstable limit cycle.\label{Fig7(b)}]
{\includegraphics[width=.19\columnwidth,height=.17\columnwidth]{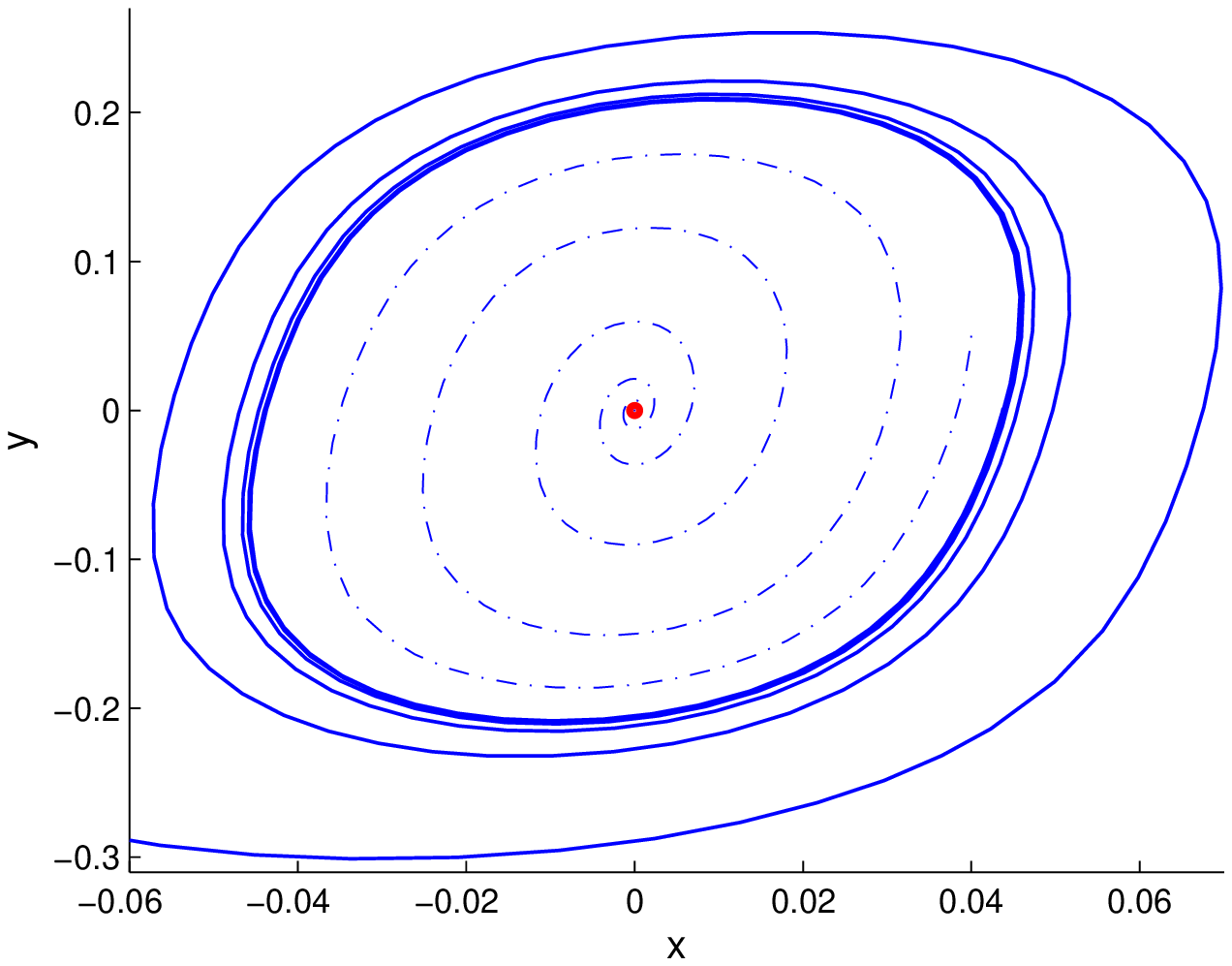}}
\subfigure[A nodal sink is surrounded by an unstable limit cycle.\label{Fig7(c)}]
{\includegraphics[width=.19\columnwidth,height=.17\columnwidth]{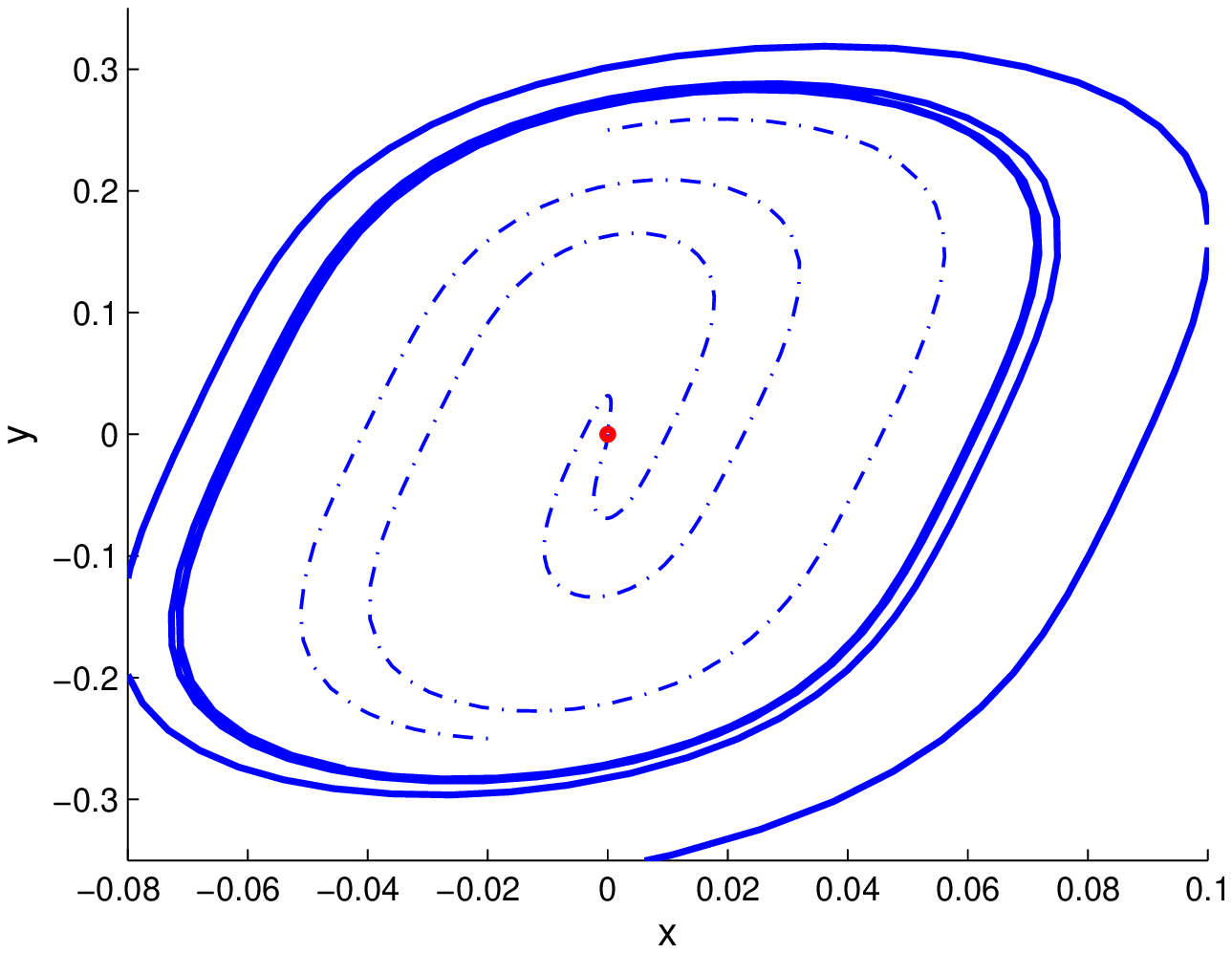}}
\subfigure[A saddle and two stable nodes exist inside an unstable limit cycle.]
{\includegraphics[width=.19\columnwidth,height=.17\columnwidth]{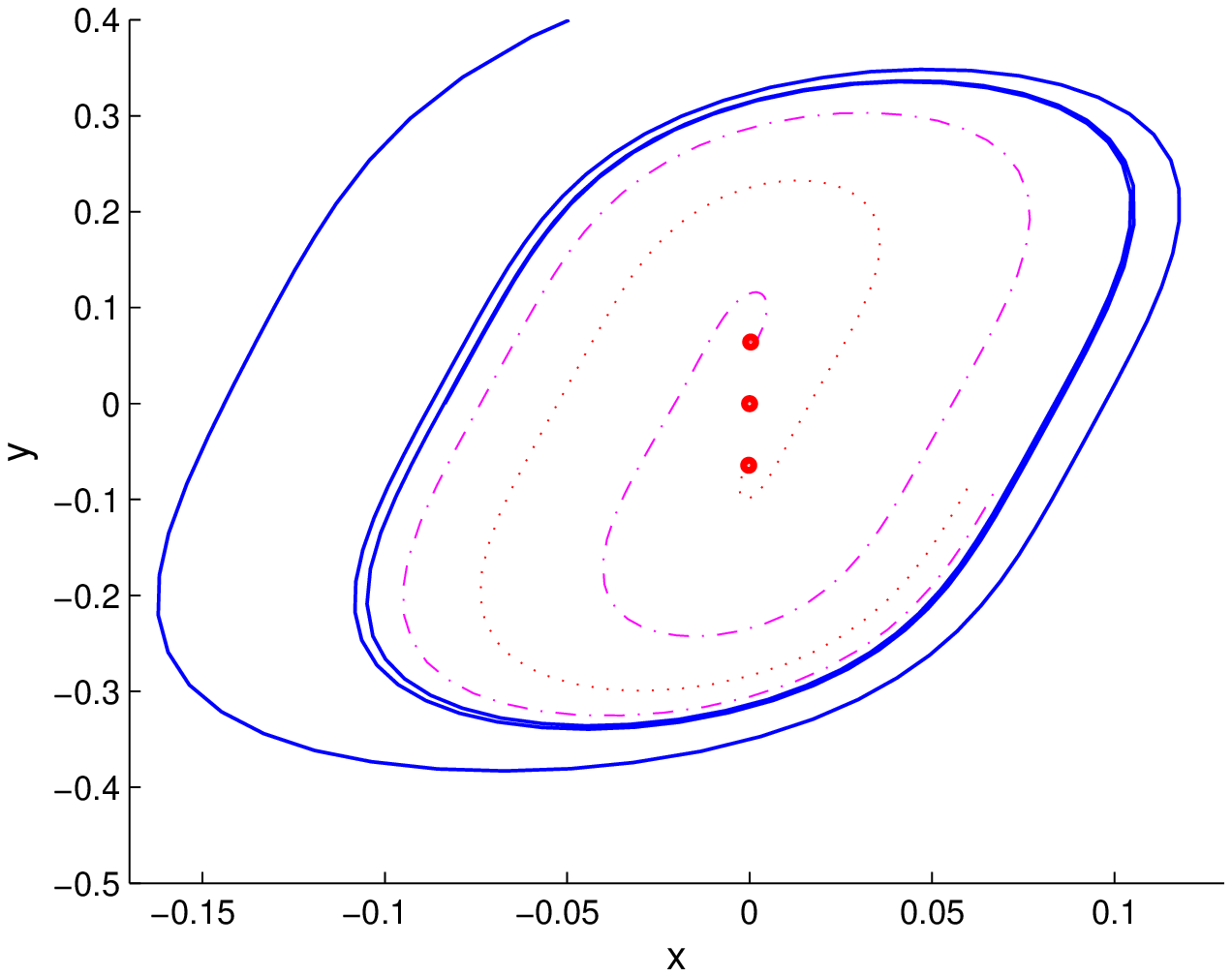}}
\subfigure[A saddle and two spiral sinks are inside an unstable limit cycle.]
{\includegraphics[width=.19\columnwidth,height=.17\columnwidth]{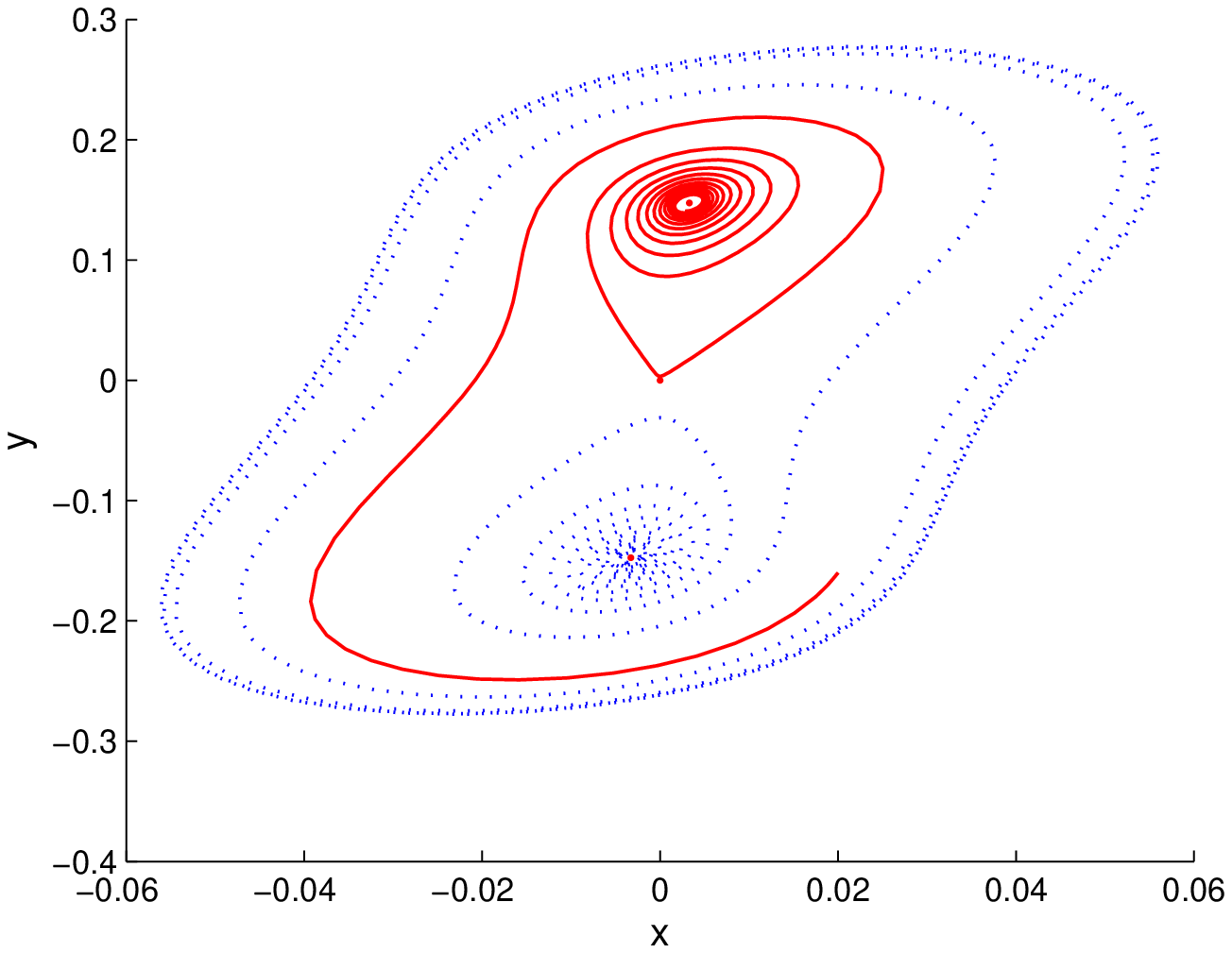}}
\subfigure[An unstable limit cycle surrounds two stable limit cycles.]
{\includegraphics[width=.26\columnwidth,height=.17\columnwidth]{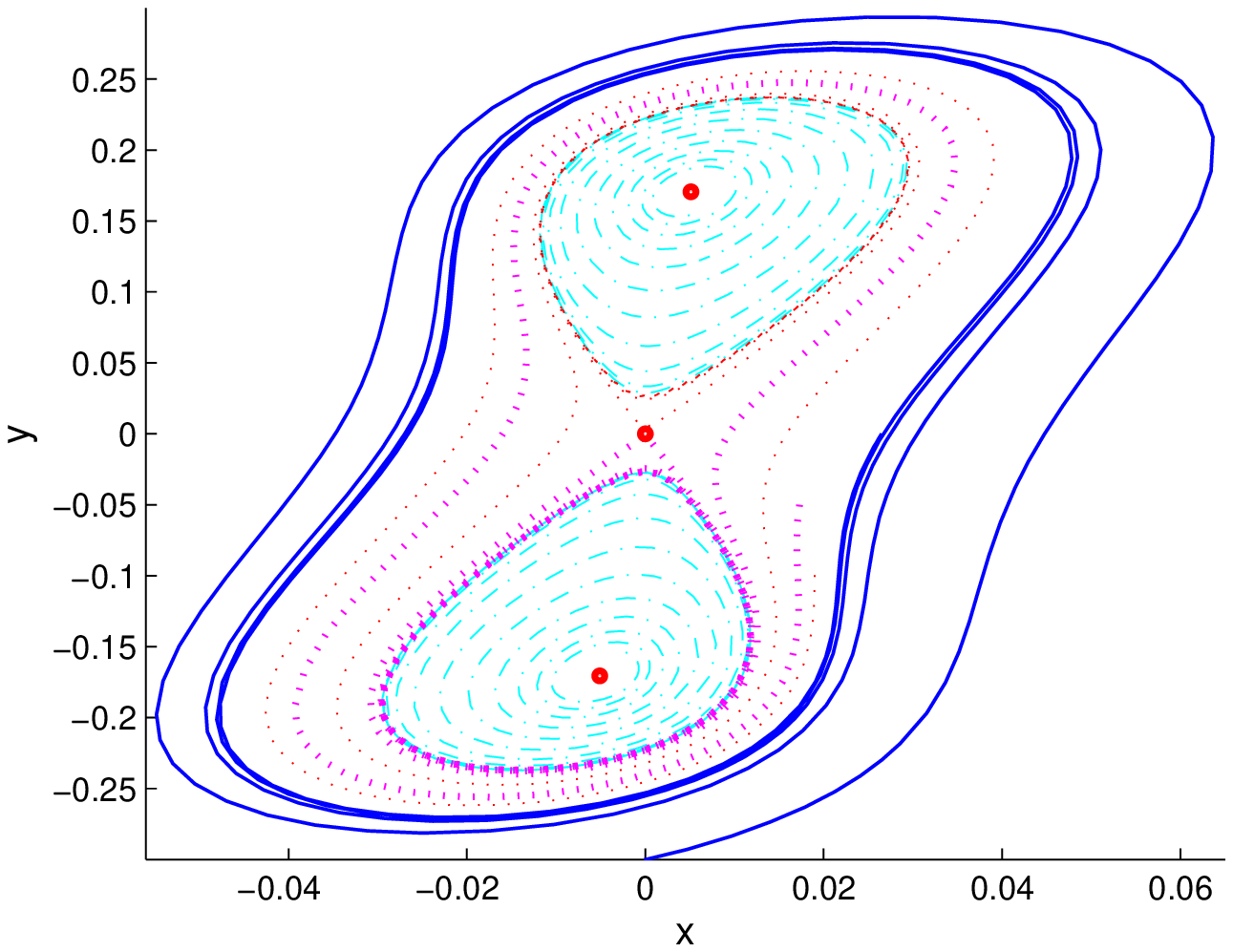}}
\subfigure[Two spiral sources and a saddle exist inside two limit cycles.]
{\includegraphics[width=.26\columnwidth,height=.17\columnwidth]{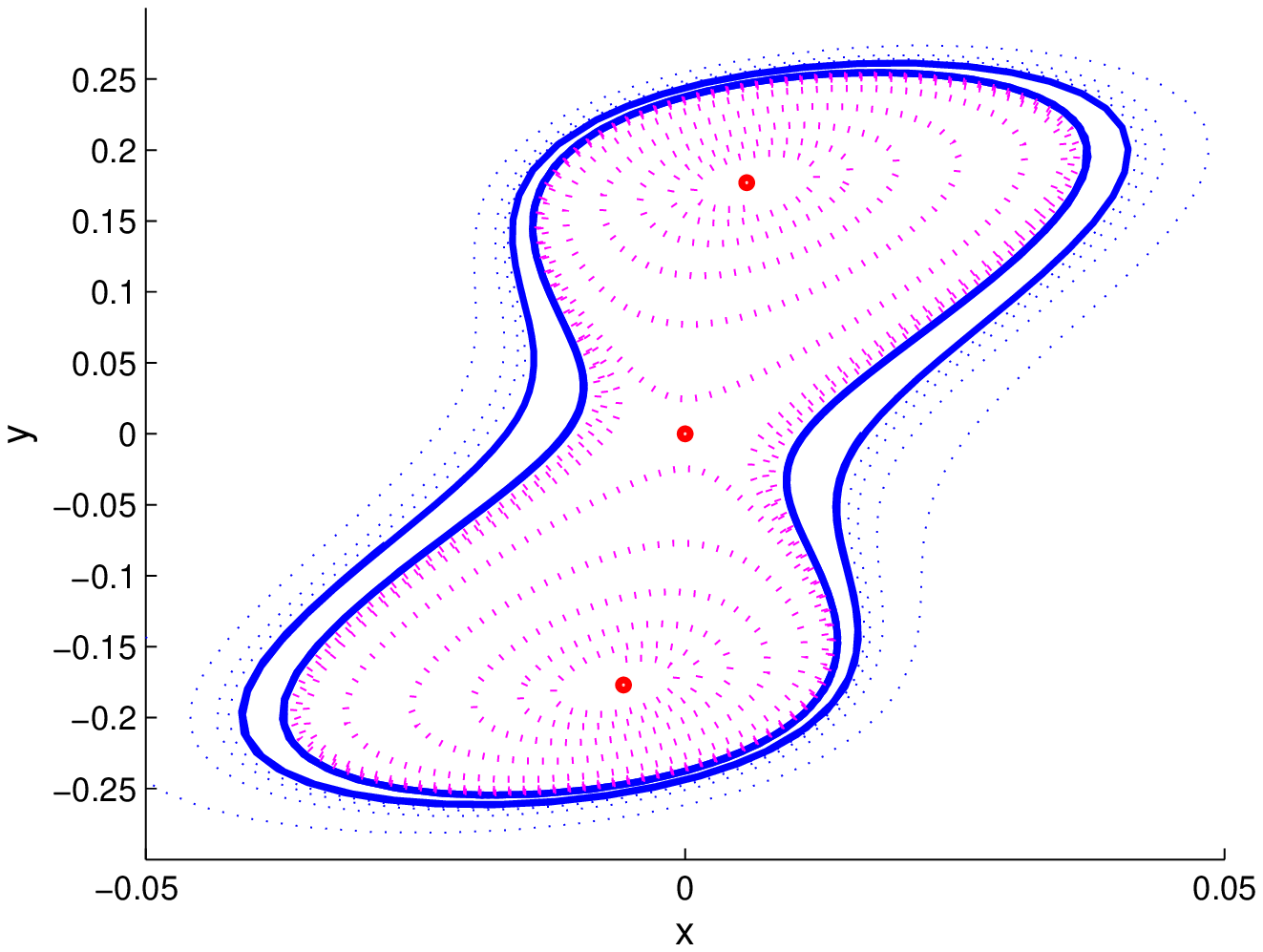}}
\subfigure[There are two spiral sources and a saddle.\label{Fig7(g)}]
{\includegraphics[width=.2\columnwidth,height=.17\columnwidth]{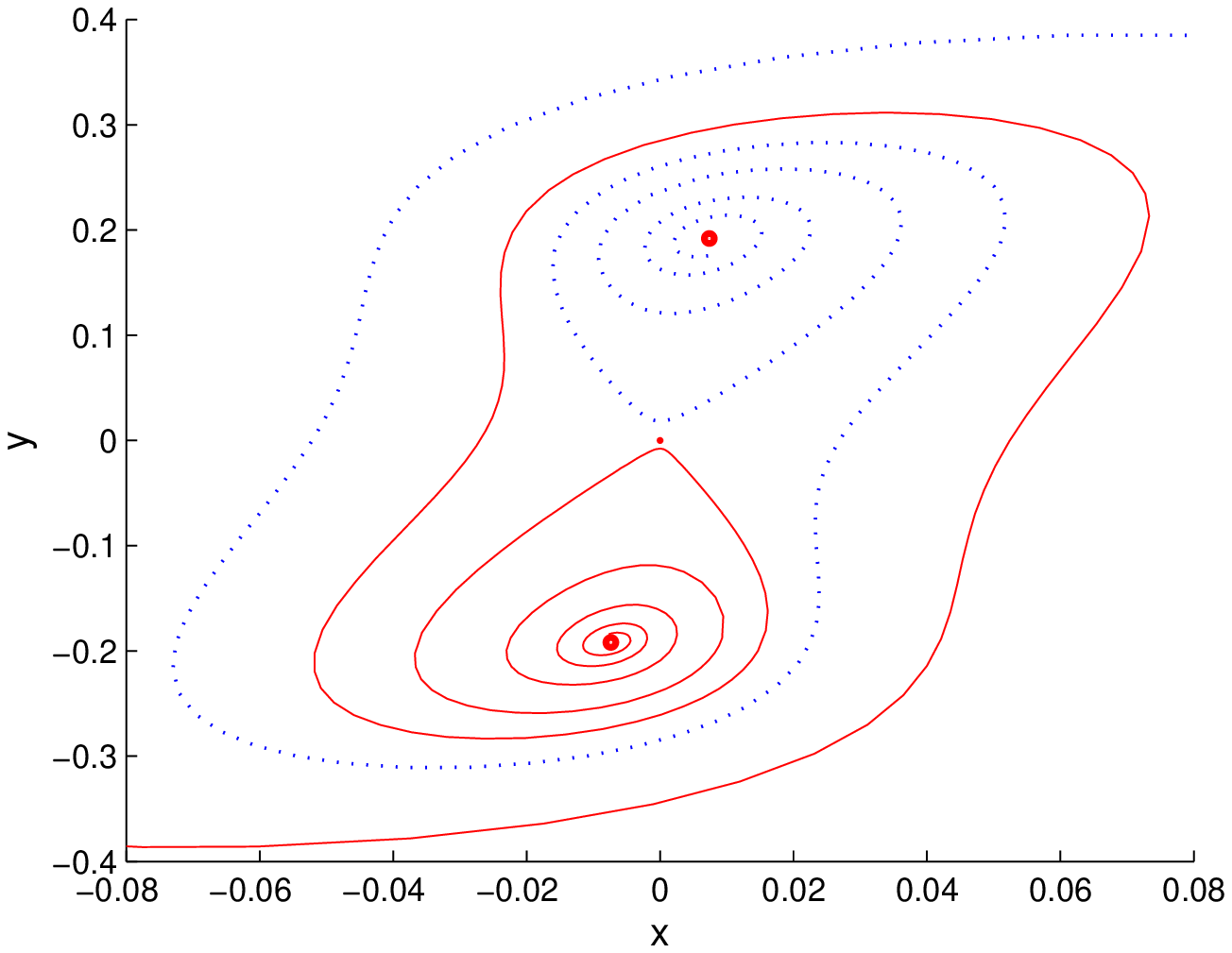}}
\subfigure[The equilibrium is a nodal source.\label{Fig7(i)}]
{\includegraphics[width=.2\columnwidth,height=.17\columnwidth]{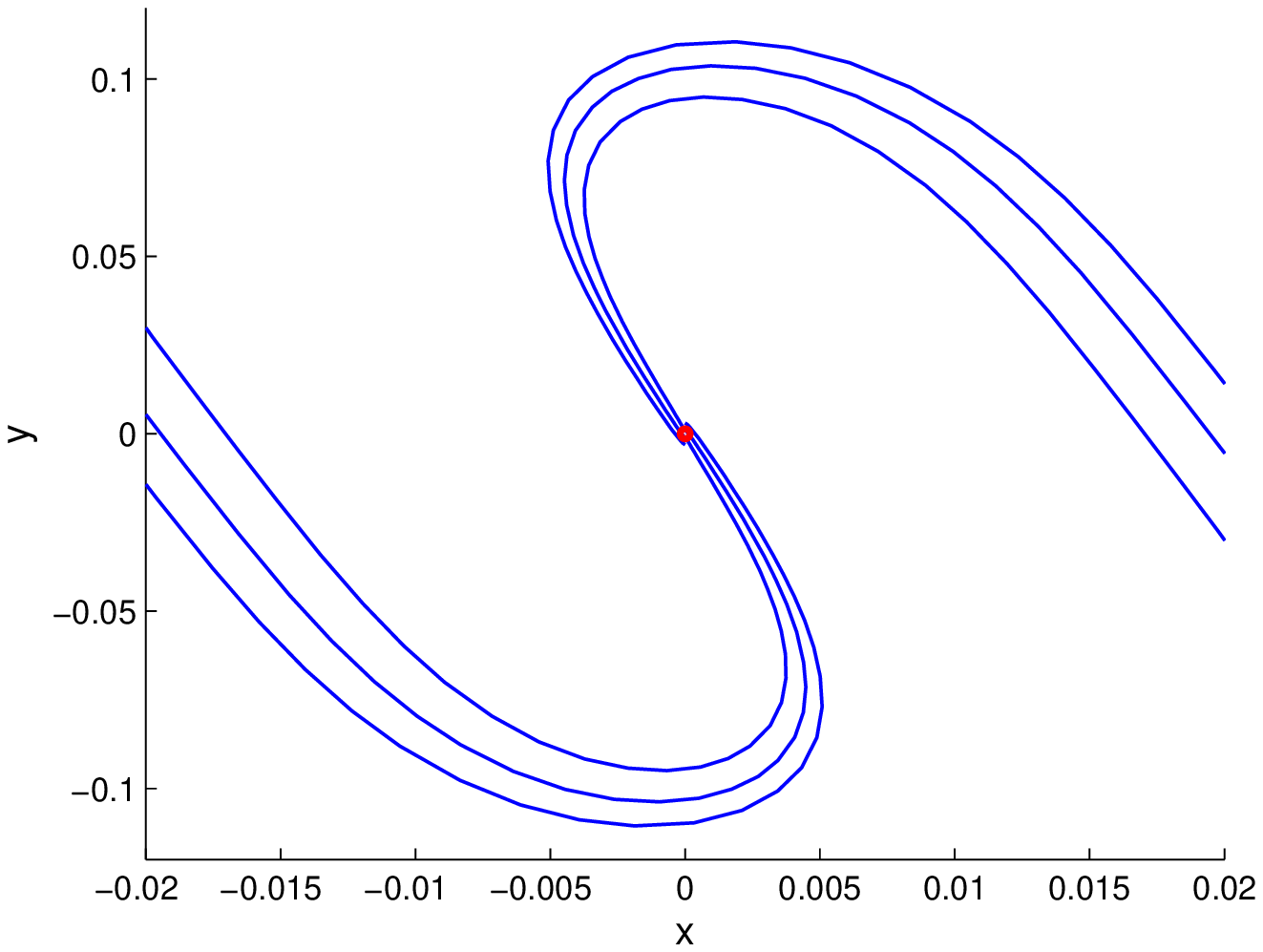}}
\caption{\(\mathbb{Z}_2\)-equivariant controlled numerical phase portraits for the plant \eqref{BfControlr2s2}. Here \(a_2= b_2= 1\), and Figures \ref{Fig7(a)}-\ref{Fig7(i)} are associated with regions (a)-(i) in Figure \ref{ABCD}(a), respectively.}
\end{center}
\end{figure}

\subsubsection{\(\mathbb{Z}_2\)-equivariant bifurcation controller design }\label{subsecZ2}

Let \(\mu_i:=0\) for \(i\geq 5.\) The cubic truncated simplest parametric normal form of the system \eqref{BfControlr2s2} is
\be\label{whynu40}
\dot{x}=\nu_2  y+\dfrac{1}{2} (\mu_1+\mu_4) x+d_2 y^3+\dfrac{1}{4}\left( d_3+3 d_6\right) x y^2, \quad \dot{y}= -x+\dfrac{1}{2}(\mu_1+\mu_4) y+\dfrac{1}{4}\left( d_3+3 d_6\right)y^3,
\ee where
\begin{eqnarray*}
\nu_2&:=&\mu_3-\dfrac{1}{{4}} {\mu_4}^2-\dfrac{1}{4}{\mu_1}^2+\dfrac{24 {d_2}^2 (3d_1+d_8)-(d_3+3 d_6)\big(\left(d_3-d_6\right)(d_3-9 d_6) +8 d_2 (d_4+d_7)\big)}{{12 {d_2}^2 \left( d_3+3d_6 \right) }}{\mu_3}^2+\mu_2\mu_3 \\
&&+\dfrac{1}{2}\mu_1\mu_4+\dfrac{ \left(3d_3+d_6\right)^2 +16(d_2 d_7-{d_6}^2)}{{16 d_2 \left( d_3
+3d_6 \right)}}\mu_1\mu_3+\dfrac{\left(d_6-d_3\right)\left(23d_3+81d_6\right)+16d_2\left(4d_4+5d_7\right)}{{16 d_2\left(d_3
+3d_6\right)}}\mu_3\mu_4.
\end{eqnarray*}

\begin{figure}
\begin{center}
\subfigure[ There are two saddle points and a spiral source at the origin.\label{Fig8(a)}]
{\includegraphics[width=.31\columnwidth,height=.2\columnwidth]{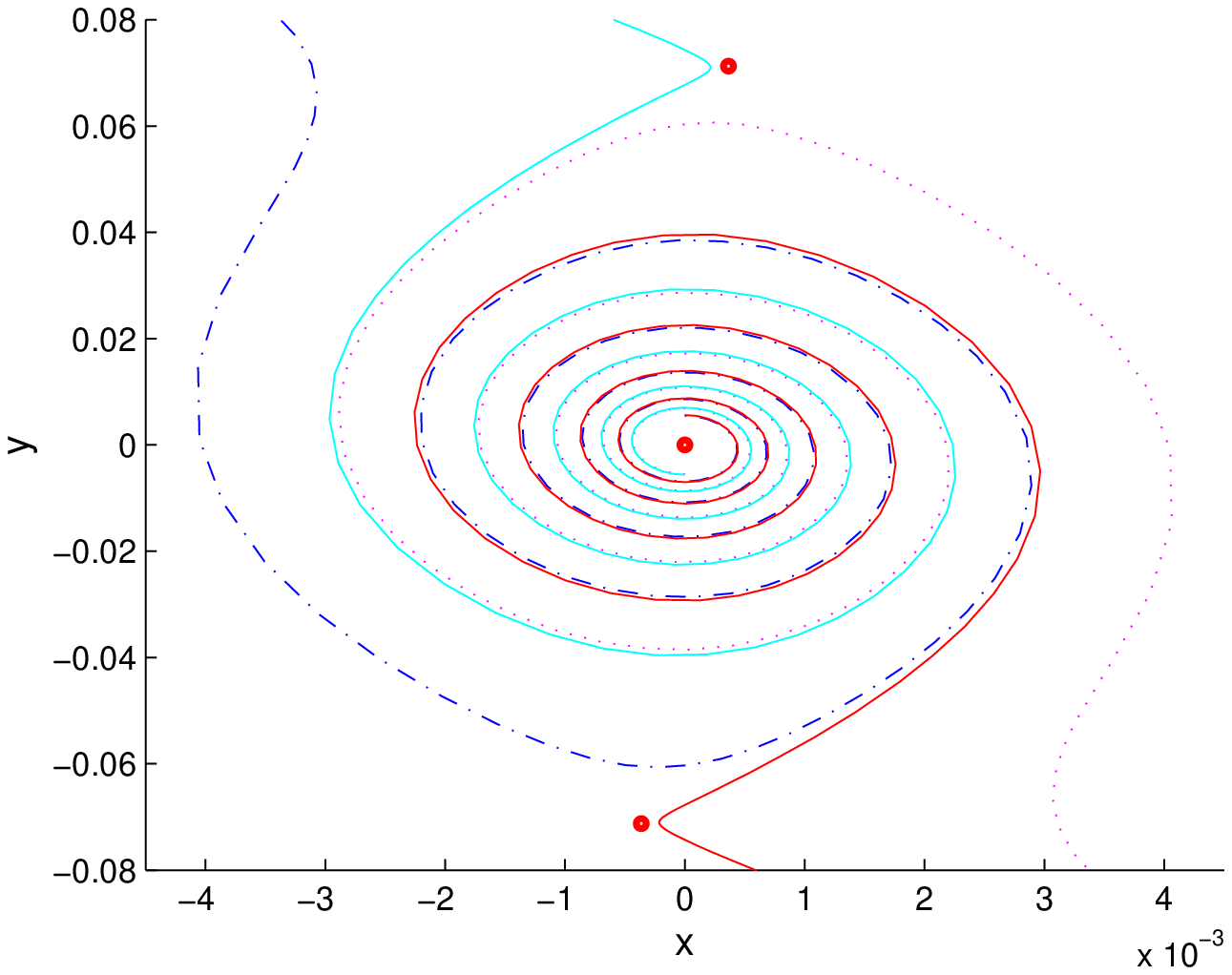}}
\subfigure[An unstable limit cycle is surrounded by two saddle points.\label{Fig8(b)}]
{\includegraphics[width=.34\columnwidth,height=.2\columnwidth]{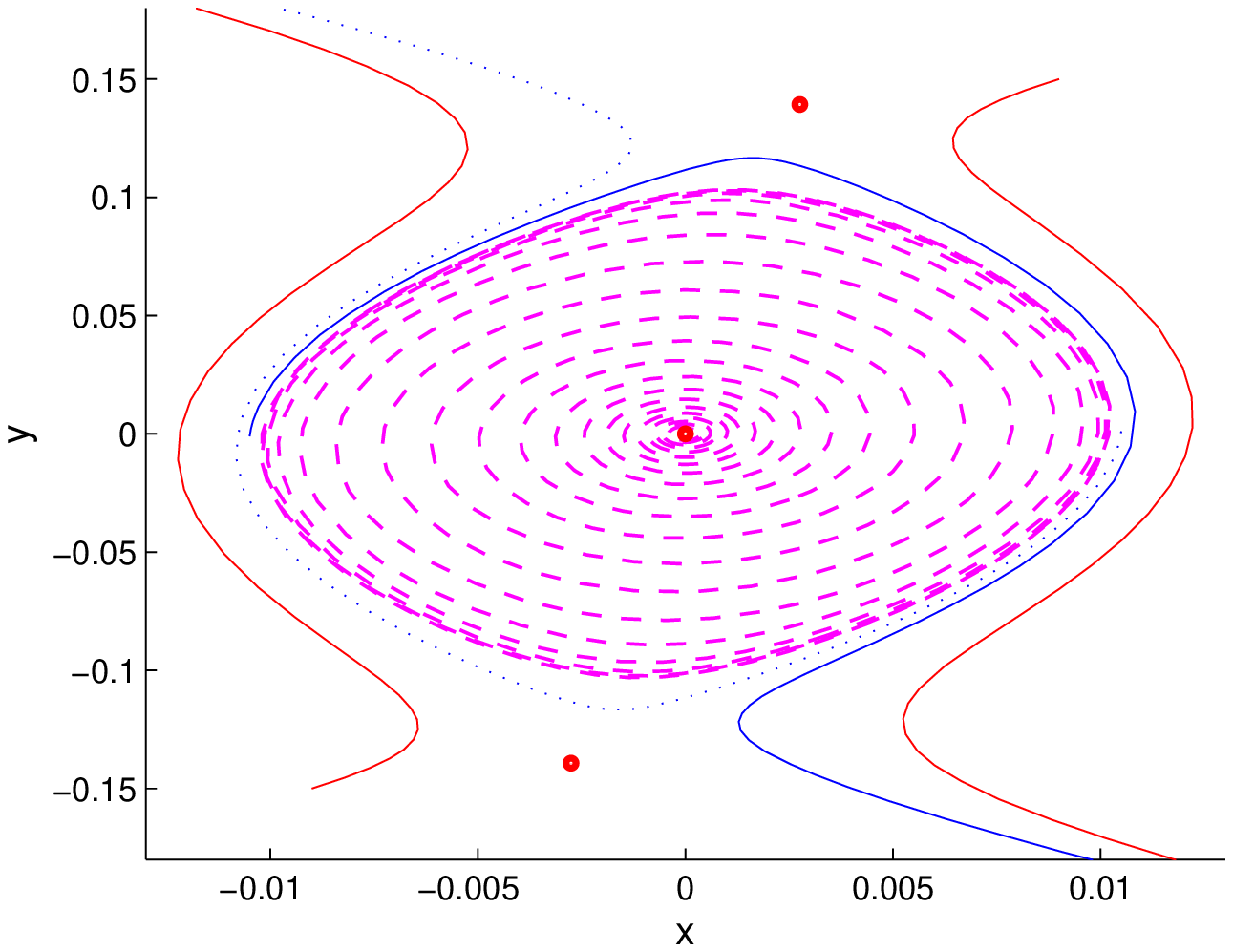}}
\subfigure[The limit cycle in Figure \ref{Fig8(b)} is broken via a heteroclinic bifurcation.]
{\includegraphics[width=.31\columnwidth,height=.2\columnwidth]{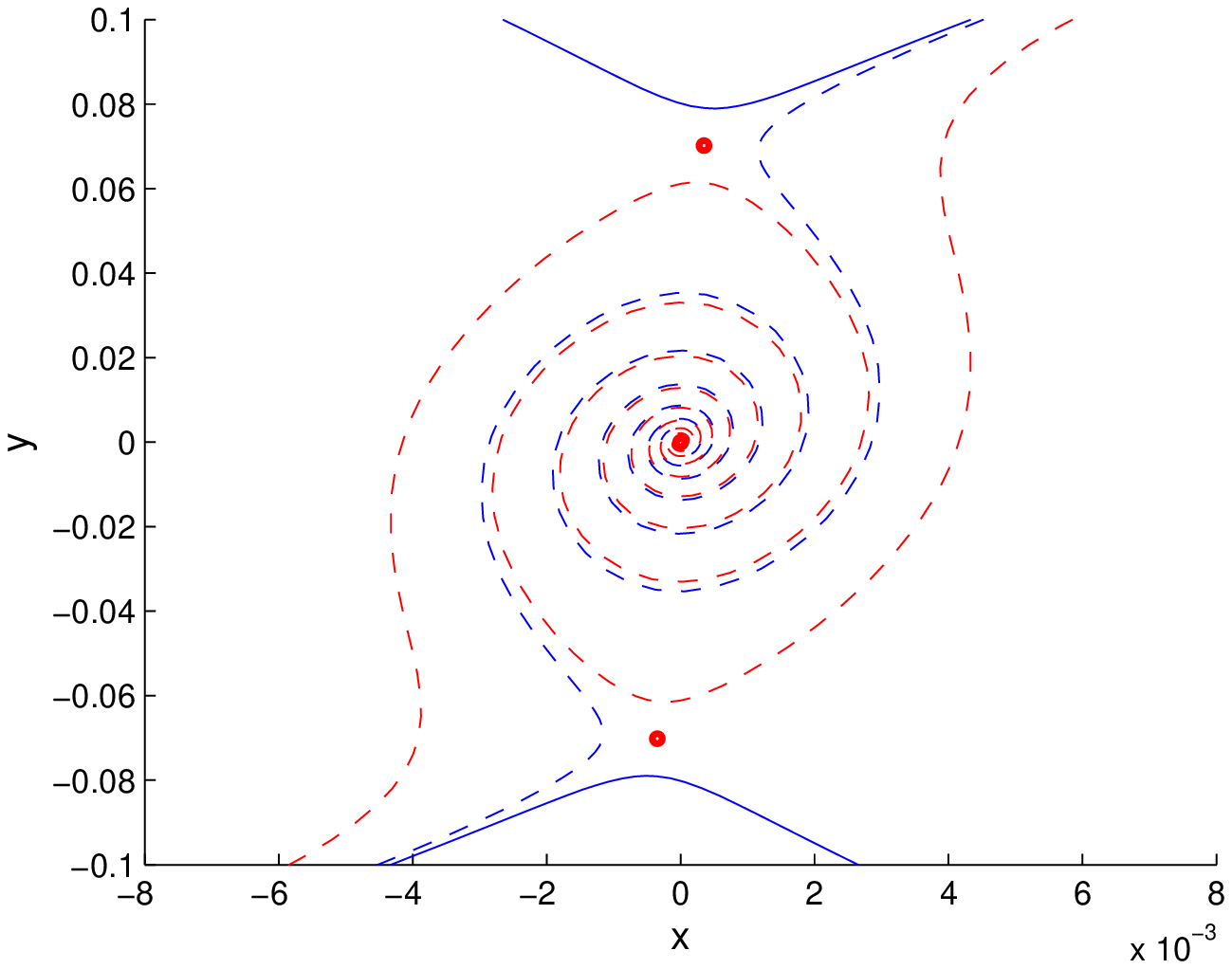}}
\subfigure[There are two saddles and a nodal sink at the origin.]
{\includegraphics[width=.32\columnwidth,height=.2\columnwidth]{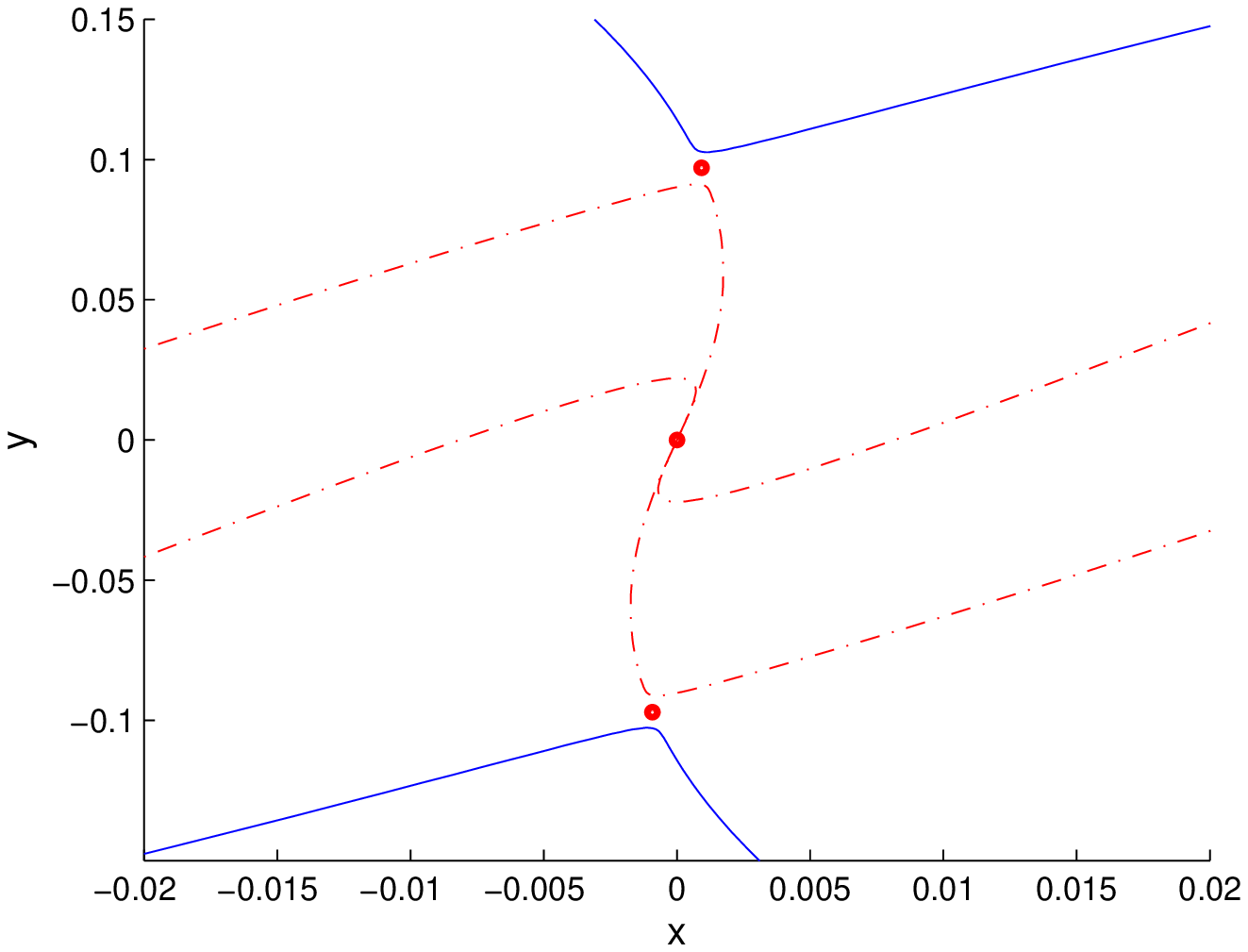}}
\subfigure[There is a saddle point at the origin.\label{Fig8(d)}]
{\includegraphics[width=.32\columnwidth,height=.2\columnwidth]{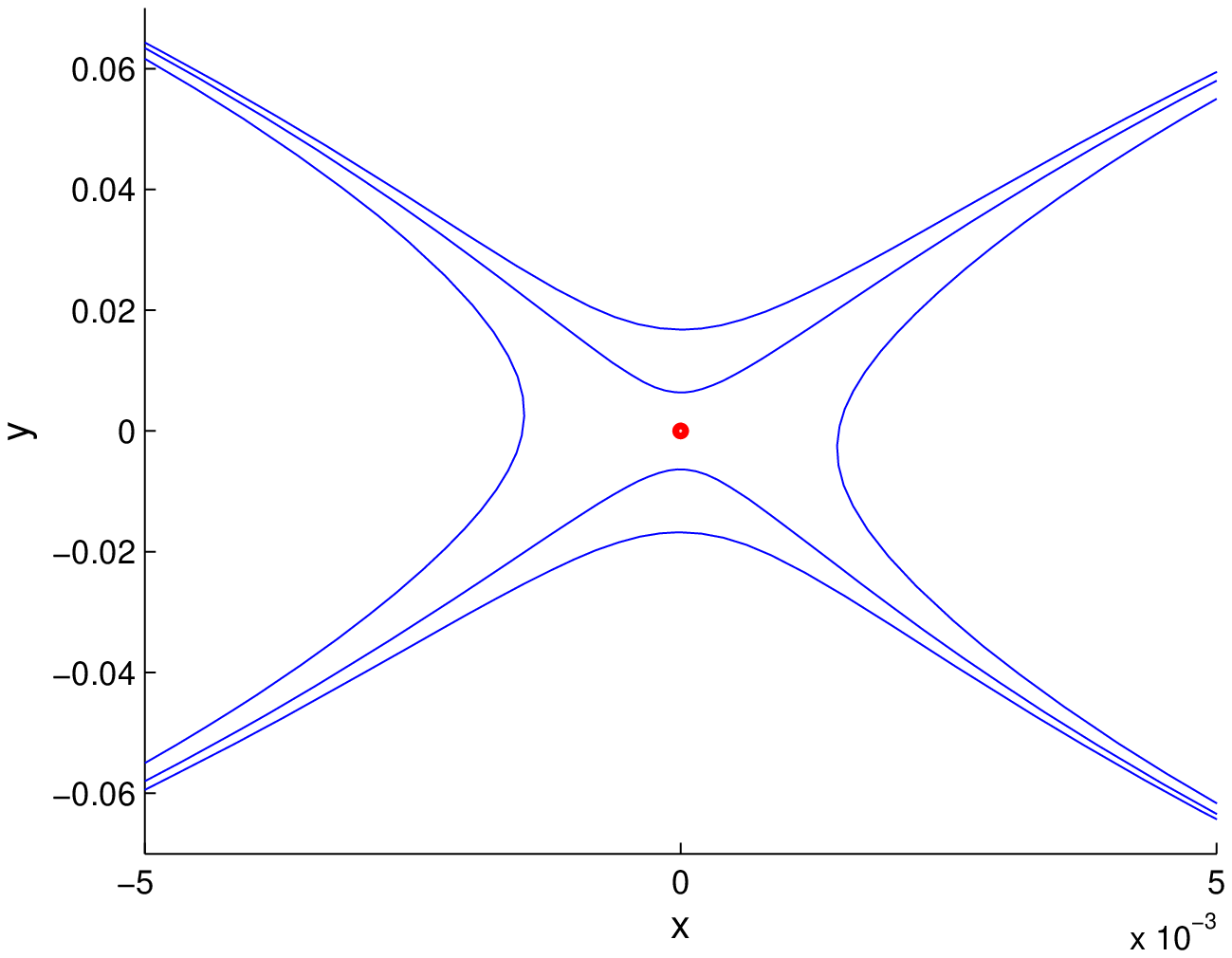}}
\subfigure[There are two saddle equilibria along with a nodal source at the origin.\label{Fig8(g)}]
{\includegraphics[width=.32\columnwidth,height=.2\columnwidth]{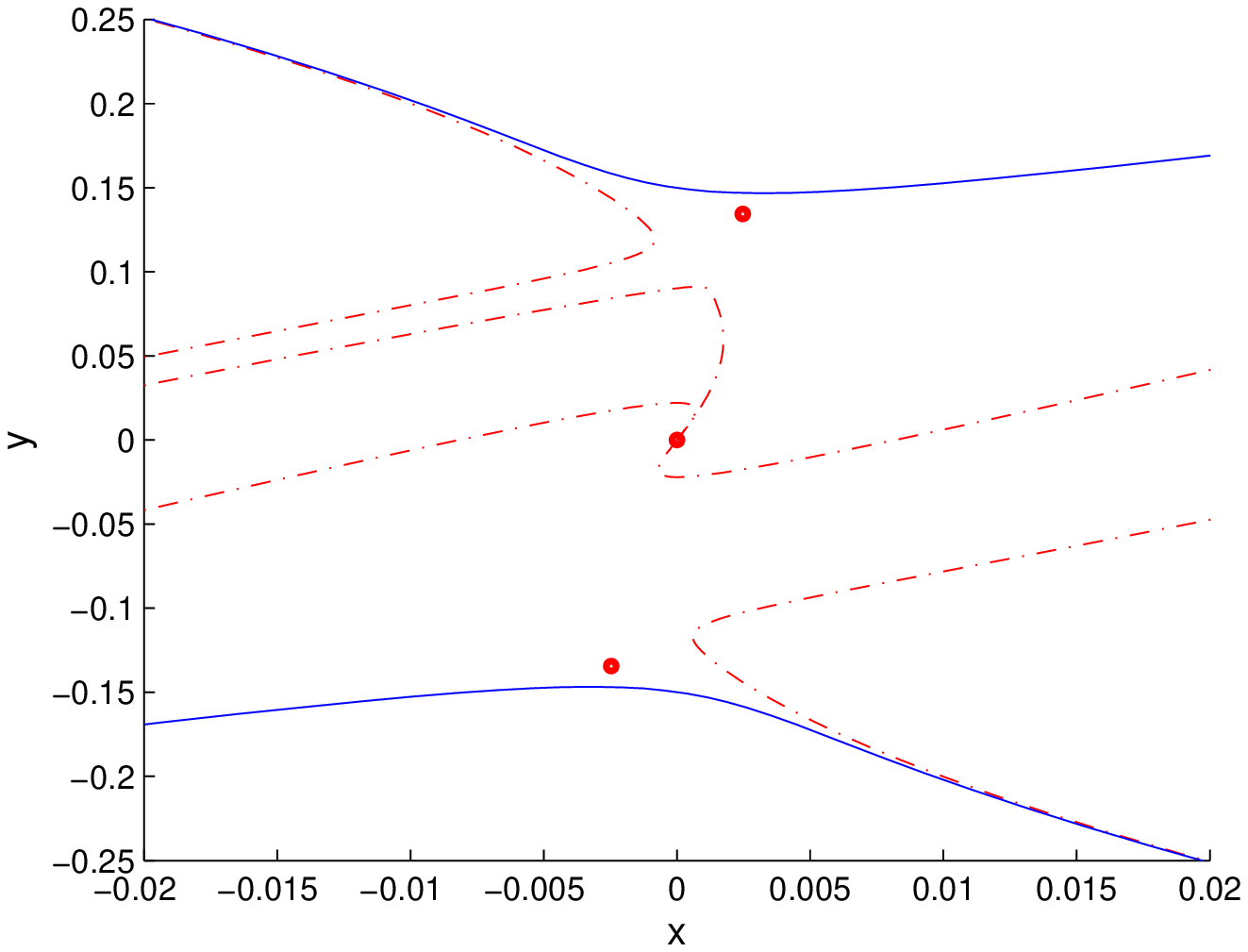}}
\caption{\(\mathbb{Z}_2\)-equivariant controlled numerical phase portraits for the plant \eqref{BfControlr2s2} when \(a_2=-1, b_2= 1\). Figures \ref{Fig8(a)}-\ref{Fig8(g)} are associated with the connected regions assigned by (a)-(f) in Figure \ref{ABCD}(c), respectively.}
\end{center}
\end{figure}

The condition \(r=s=2\) and a simple {\sc Maple} programming imply that the parameters \((\mu_1, \mu_3)\) and \((\mu_3, \mu_4)\) can play the role of the distinguished parameters (asymptotic unfolding parameters), \ie \(\nu_2\) and \(\nu_3\) are diffeomorphic polynomials in terms of \(\mu_1\) and \(\mu_3\); see \cite{GazorSadri}. We choose \(d_2:=1,\) \(d_3:=d_6:=b,\) and  set \(\mu_2:= \mu_4:= 0\). Then the estimated bifurcation control varieties with symbolic coefficients (\(b, d_1, d_4, d_7, d_8\)) of the controlled system \eqref{BfControlr2s2} are derived as
\begin{eqnarray}\nonumber
T_{P}&=&\{ (\mu_1, \mu_3)\,|\,\mu_3=0\},\qquad\qquad
T_{H}=\{ (\mu_1, \mu_3)\,|\,\mu_1=0\},\\\label{Transr2s2Z21}
T_{F}&=&\left\{ (\mu_1, \mu_3)\,|\, 64 \mu_3+24{\mu_1}^2+8 b{\mu_1}^3+b^2{\mu_1}^4=0\right\},\\\nonumber
T_{H_{\pm}}&=&\left\{ (\mu_1, \mu_3)\,|\, 12 b \mu_1-48 b^2\mu_3-8b\big(9d_1-4bd_4-4bd_7+3d_8\big){\mu_3}^2 +9b{\mu_1}^2-12bd_7\mu_1\mu_3=0\right\},\\\nonumber
{T_{F}}_{\pm}&=&\Big\{(\mu_1, \mu_3)\,| 12 b\mu_3+\left(18d_1+6d_8\mp 8d_4b\mp 8d_7b \right) {\mu_3}^2-3 b{\mu_1}^2+3d_7\mu_1\mu_3=0\Big\},
\end{eqnarray}
\begin{eqnarray}\nonumber
&T_{SC}= \Big\{(\mu_1, \mu_3)\,|&
-120b\mu_1+\big( \left( 136bd_7+40bd_{{4}}-90d_1-30d_8 \right)\mu_1
+384{b}^2 \big) \mu_3\\\nonumber
&& +\big( 192bd_8+576bd_1-256{b}^2 \left( d_4+d_7 \right)\big) {\mu_3}^2-3\left( 5d_7+32{b}^2 \right){\mu_1}^2
=0\Big\},\\\label{Transr2s2Z22}
&T_{SNLC}=\Big\{(\mu_1, \mu_3)\,|&
-3000 b \mu_1+ \big((1000b d_4+3256 b d_7-2250d_1-750 d_8) \mu_1+9024{b}
^2 \big) \mu_3\\\nonumber
&&+\big(13536 b d_1+4512 b d_8-6016 b^2\left(d_4+d_7\right)\big)
{\mu_3}^2- \left(375d_7+2256 b^2\right) {\mu_1}^2=0\Big\}.
\end{eqnarray}

Choosing \(b=\pm 1,\) we obtain highly accurate estimated numerical transition varieties in terms of \(\mu_1\) and \(\mu_3.\) These are depicted in Figures \ref{Fig6(a)}-\ref{Fig6(d)}. For numerical simulations, we choose a pair of parameters \((\mu_1, \mu_3)\) from each connected region labeled (a)-(i) in Figure \ref{Fig6(a)} as
\begin{eqnarray}\label{mu1mu3Values}
&&(0.05, 0.02), (-0.05, 0.02), (-0.1, 0.00143), (-0.15, -0.00353), (-0.1, -0.02)\\\nonumber
&& \qquad (-0.1, -0.027), (-0.1, -0.0292), (-0.1, -0.0345), (0.1, 0.00143),
\end{eqnarray} when \(a_2= b_2=1.\) The associated controlled numerical phase portraits of the plant \eqref{BfControlr2s2} are plotted in Figures \ref{Fig7(a)}-\ref{Fig7(i)}, respectively. The constants are chosen as \(d_2:=-1, b:=1\) for the numerical simulations of the case \(a_2=-1, b_2=1.\) The bifurcation varieties \(T_P\) and \(T_H\) follow those of the equations \eqref{Transr2s2Z21} while
\bas
&T_{HtC}= \Big\{(\mu_1, \mu_3)\,|&
120b\mu_1+\big( \left( 90d_1+30d_8-40d_4 b-16d_7b \right)\mu_1+96{b}^2 \big) \mu_3\\\nonumber
&&+\big(144d_1b+48 d_8b-64b^2 \left(d_{{4}}+d_7\right)\big){\mu_3}^2-3\left( 8{b}^2-5d_7 \right){\mu_1}^2=0\Big\}.
\eas The input controller parameters \((\mu_1, \mu_3)\) from regions (a)-(f) in \ref{Fig6(c)} are taken as
\bes
(-0.01, 0.005), (-0.01, 0.015), (-0.02, 0.0015), (-0.3, 0.0122), (-0.01, -0.005), (0.3, 0.0122).
\ees The numerical phase portraits of either of these controlled plants are given in Figures \ref{Fig8(a)}-\ref{Fig8(g)}, respectively.

\subsubsection{\(\mathbb{Z}_2\)-symmetry breaking bifurcation control }

\begin{figure}
\begin{center}
\subfigure[\(a_2=b_2=1, \mu_5:=0.3\)\label{Fig9(a)} ]
{\includegraphics[width=.28\columnwidth,height=.2\columnwidth]{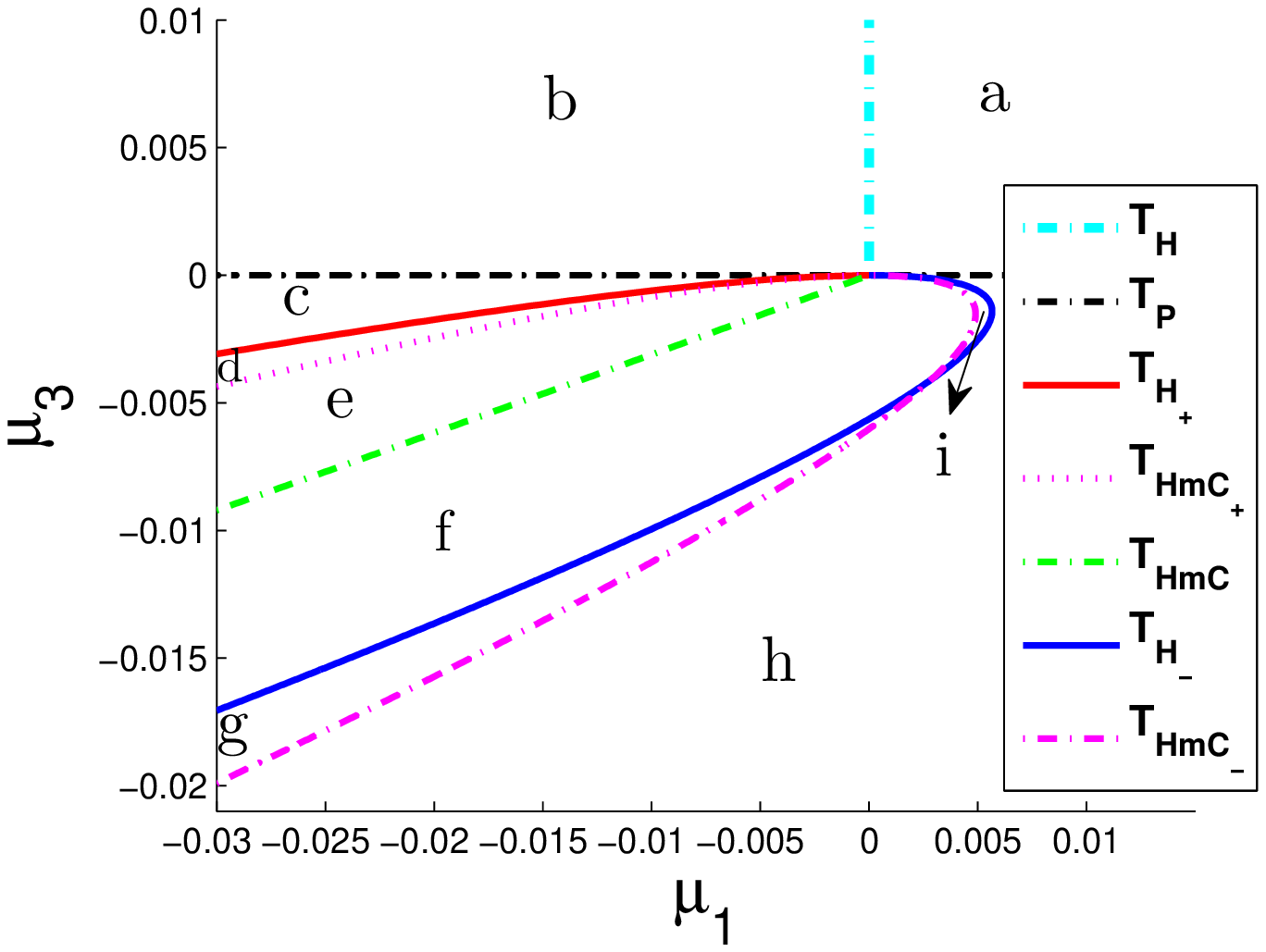}}
\subfigure[\(a_2= b_2=1, \mu_5:=-0.3\) \label{Fig9(c)}]
{\includegraphics[width=.30\columnwidth,height=.2\columnwidth]{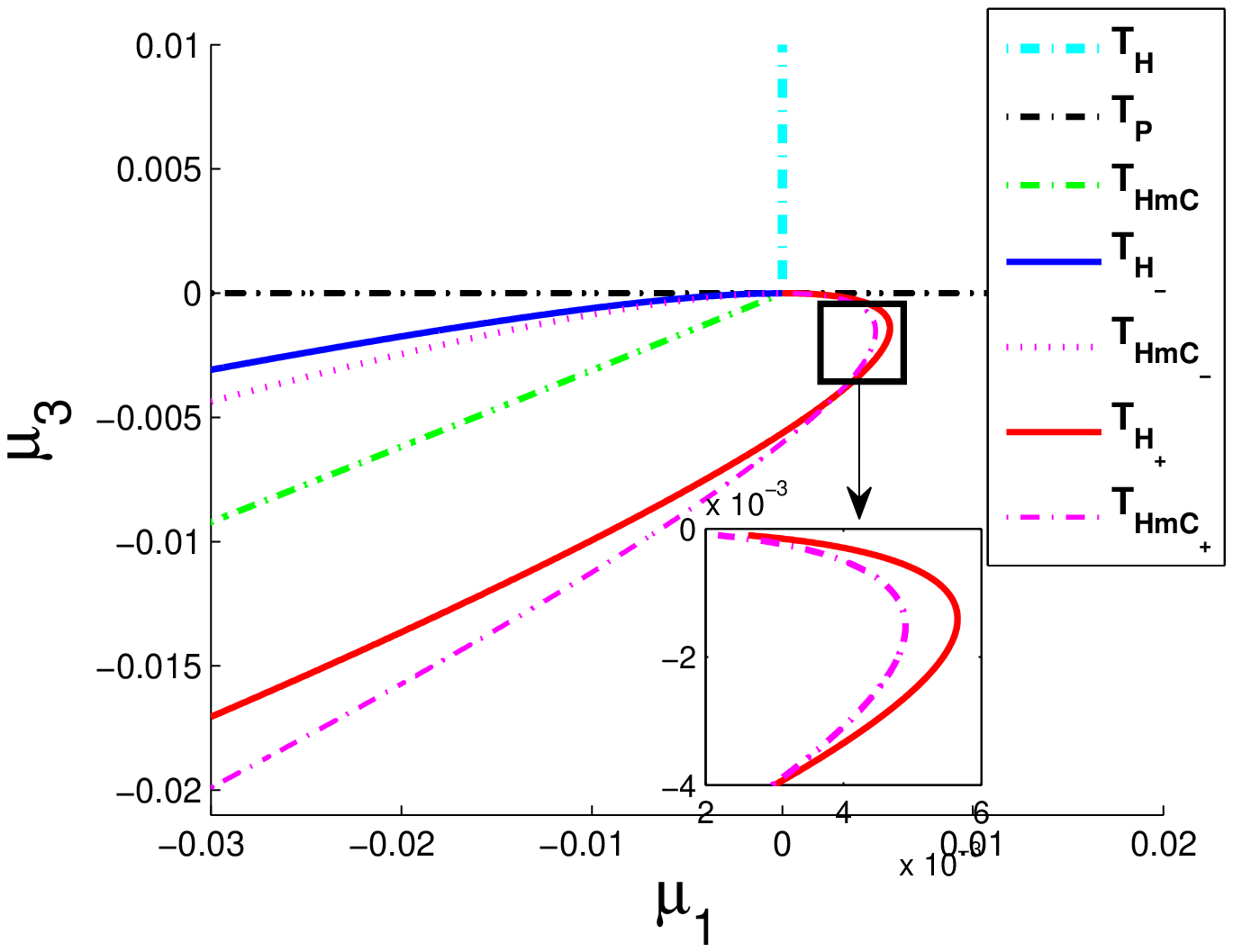}}
\subfigure[\(a_2= b_2=-1, \mu_5:=\pm 0.3\)\label{Fig9(b)}]
{\includegraphics[width=.19\columnwidth,height=.2\columnwidth]{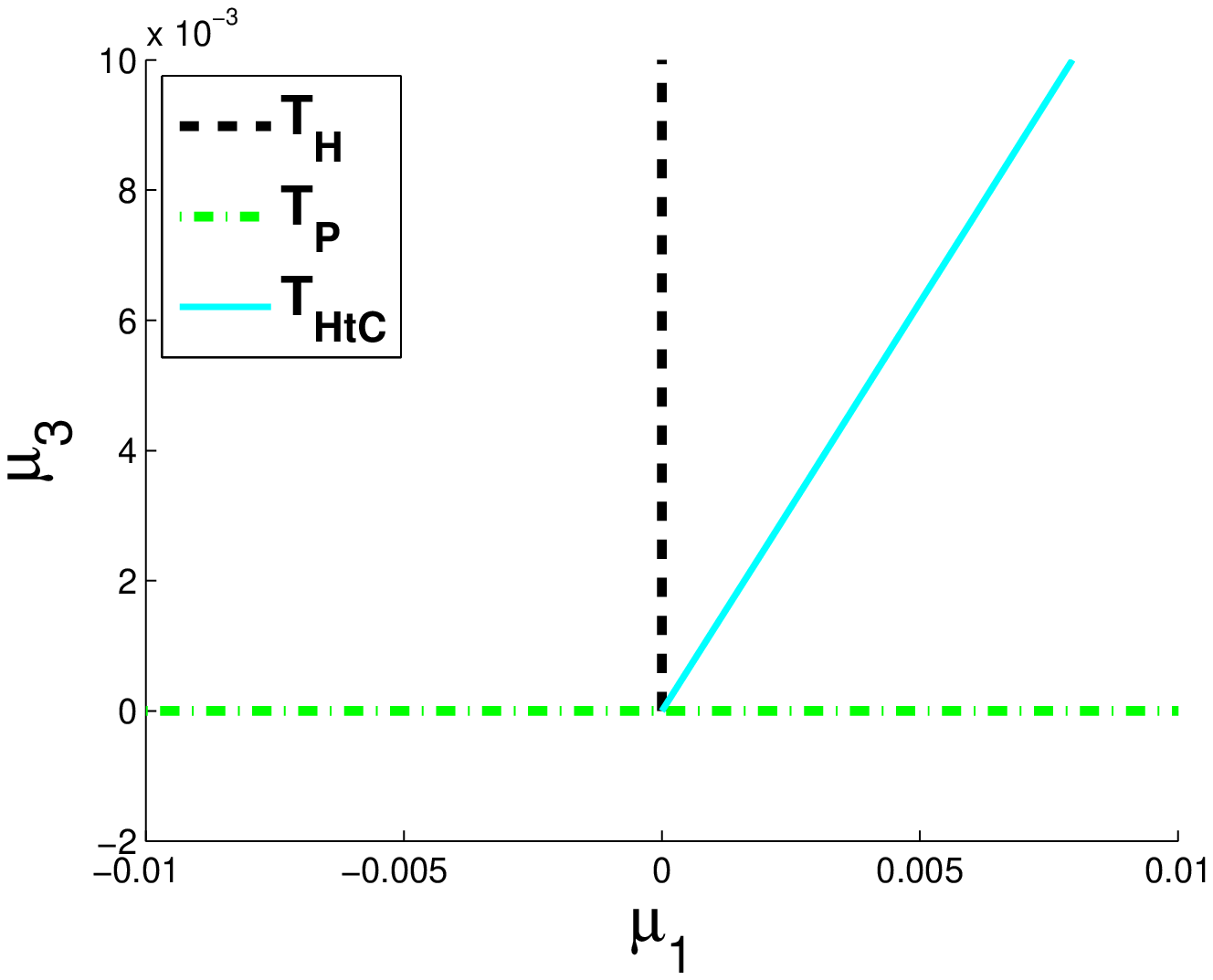}}
\subfigure[\(b_2=- a_2=1, \mu_5:=\pm 0.3\)\label{Fig9(d)}]
{\includegraphics[width=.19\columnwidth,height=.2\columnwidth]{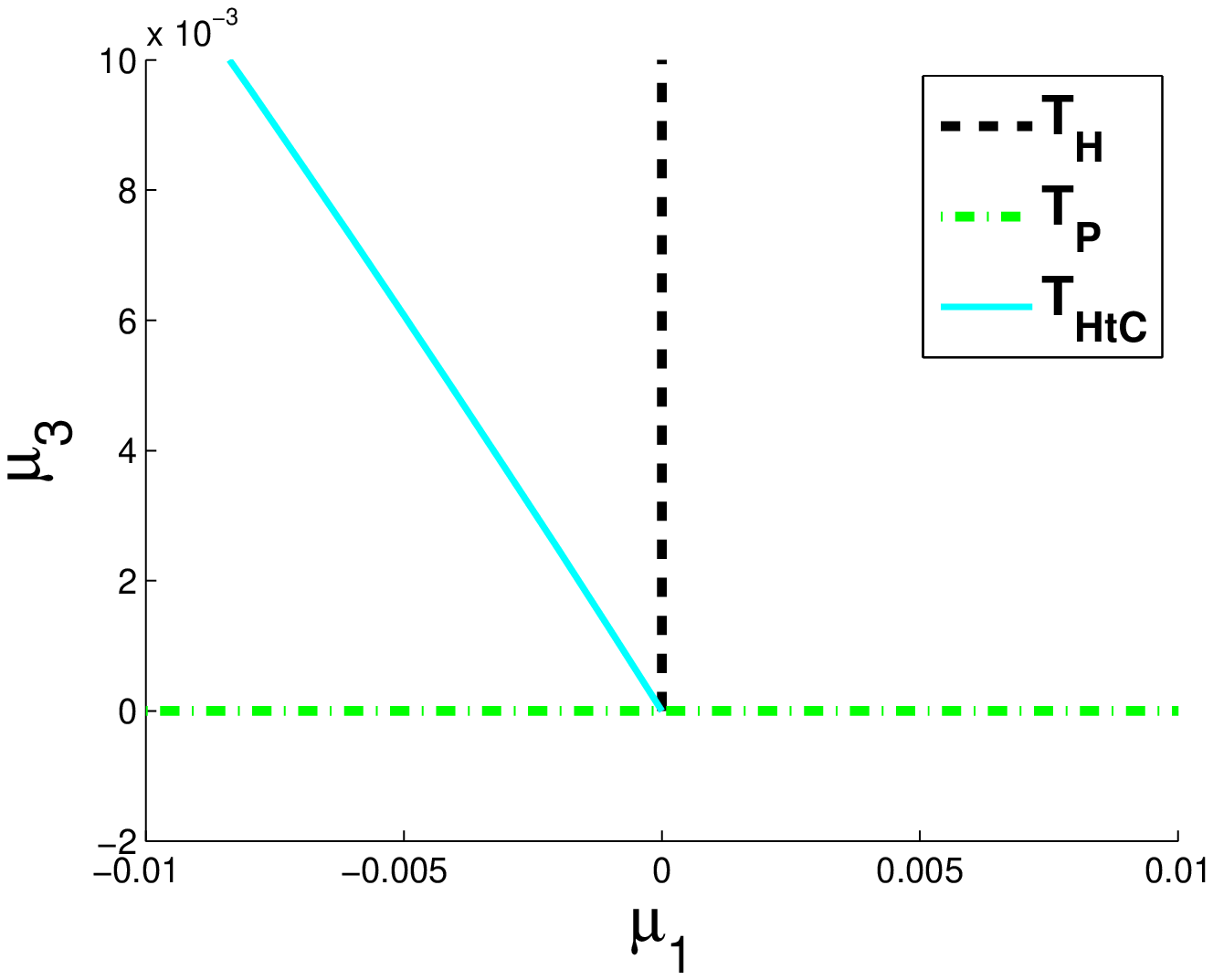}}
\caption{Numerical \(\mathbb{Z}_2\)-symmetry breaking bifurcation controller varieties for the plant \eqref{BfControlr2s2} when \(\mu_i:= 0\) for \(i\neq 1, 3, 5\) and \(r=s=2.\)} \label{Fig9}
\end{center}
\end{figure}

In this subsection we consider one parameter symmetry breaking bifurcation control. Either the parameters \((\mu_1, \mu_3, \mu_5)\) and \((\mu_3, \mu_4, \mu_5)\) play the roles of asymptotic unfolding (\ie \(\nu_2, \nu_3, \nu_4\)). Hence for consistency with subsection \ref{subsecZ2} and briefness in the derived formulas, we choose \(\mu_i:=0\) for any \(i\neq 1, 3, 5,\) \(d_1:=d_2:=d_3:= d_4:=d_5:=d_6:=d_7:=1.\)
Using our {\sc Maple} program, the three-jet truncated parametric normal form of the system \eqref{BfControlr2s2} is given by
\be\label{why}
\dot{x}= \nu_2  y+\nu_3 x+ y^3+\nu_4 x y+x y^2, \qquad \dot{y}= -x+\nu_3 y+\nu_4 y^2+y^3,
\ee where
\begin{eqnarray*}
&\nu_2=\mu_3+\frac{2{\mu_3}^2+6 d_8{\mu_3}^2-3{\mu_1}^2+3\mu_1\mu_3}{12},\quad
\nu_3=\frac{\mu_1}{2}+ \frac{6d_8 \mu_3+3\mu_1+2\mu_3}{48}\mu_1, \quad
\nu_4= \frac{\mu_5}{3}+ \frac{68 \mu_1+99\mu_3}{270}\mu_5.&
\end{eqnarray*}

\begin{figure}
\begin{center}
\subfigure[The equilibrium is a spiral source.\label{Fig10(a)}]
{\includegraphics[width=.18\columnwidth,height=.18\columnwidth]{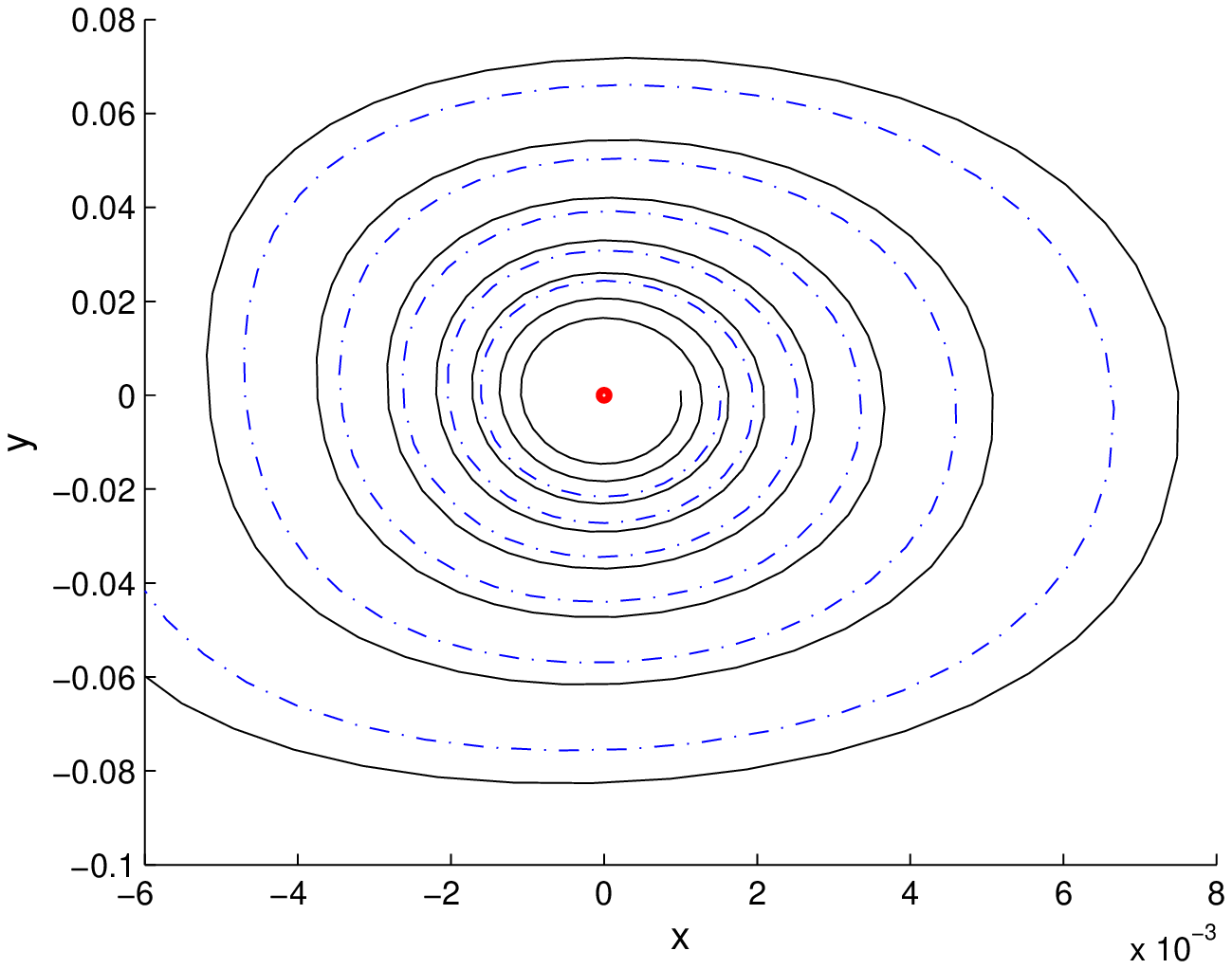}}
\subfigure[An unstable limit cycle encircles a spiral sink. \label{Fig10(b)}]
{\includegraphics[width=.2\columnwidth,height=.18\columnwidth]{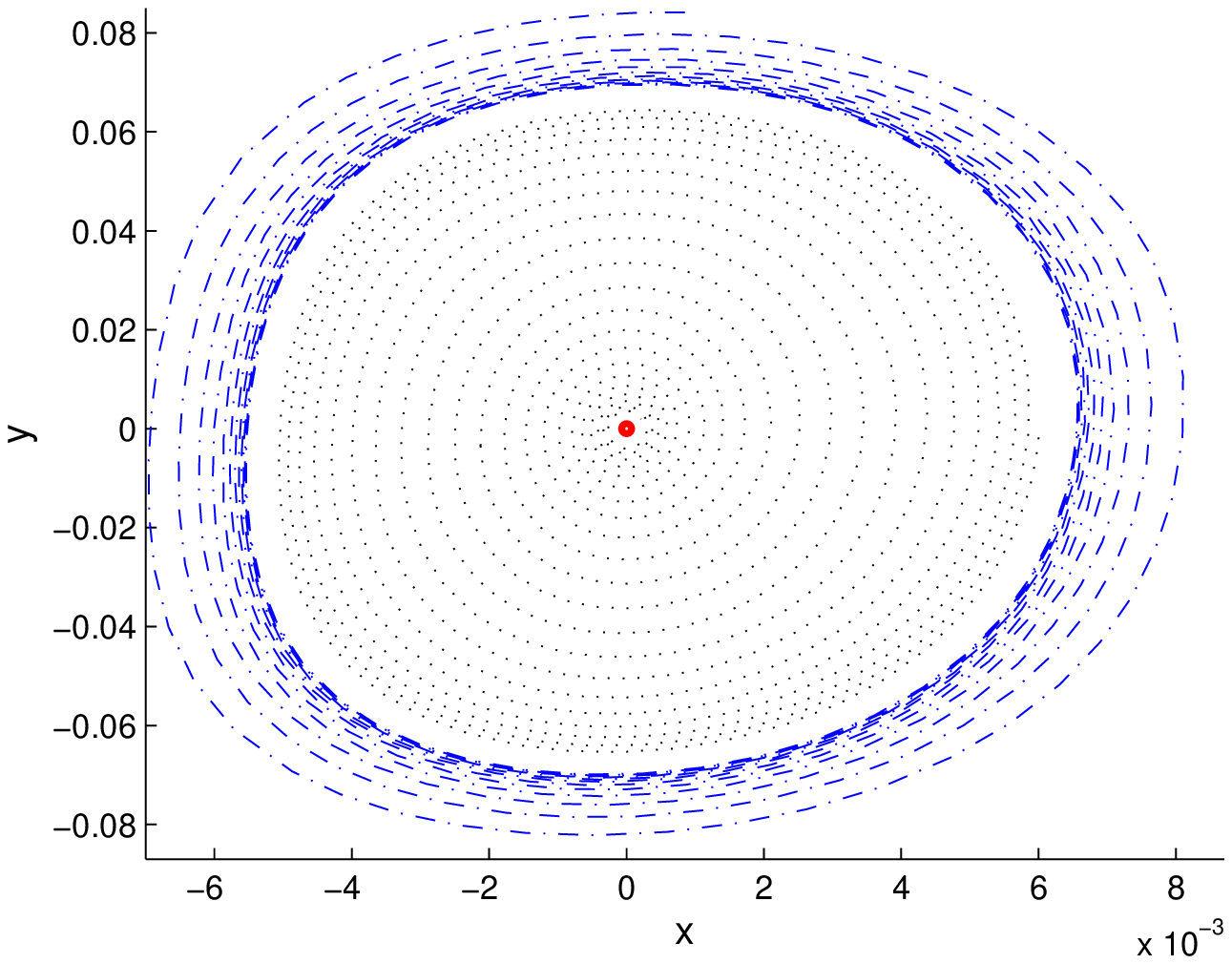}}
\subfigure[There are an unstable limit cycle, two spiral sinks, and a saddle.]
{\includegraphics[width=.31\columnwidth,height=.18\columnwidth]{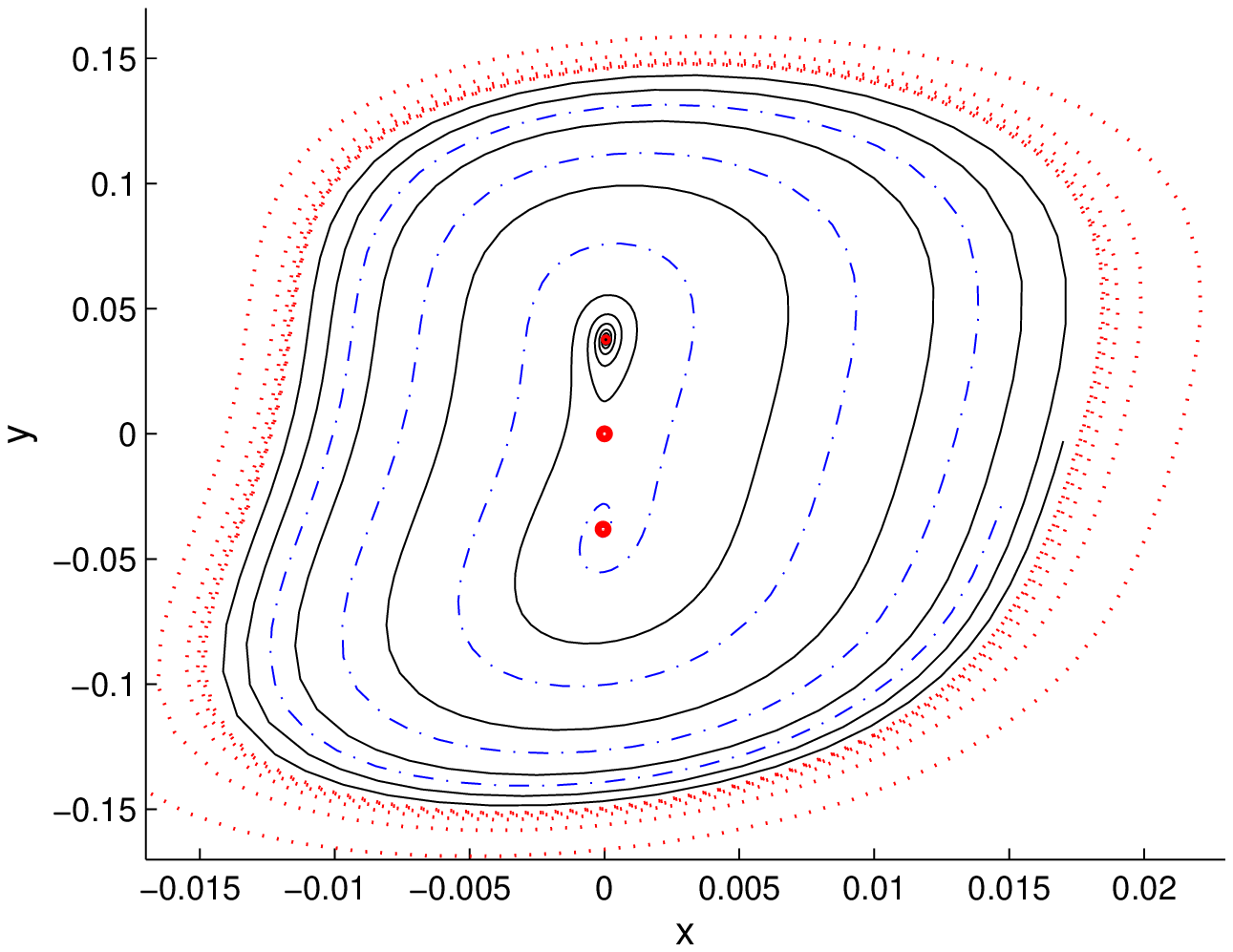}}
\subfigure[There are two spiral sinks, a saddle and two unstable limit cycles.]
{\includegraphics[width=.29\columnwidth,height=.18\columnwidth]{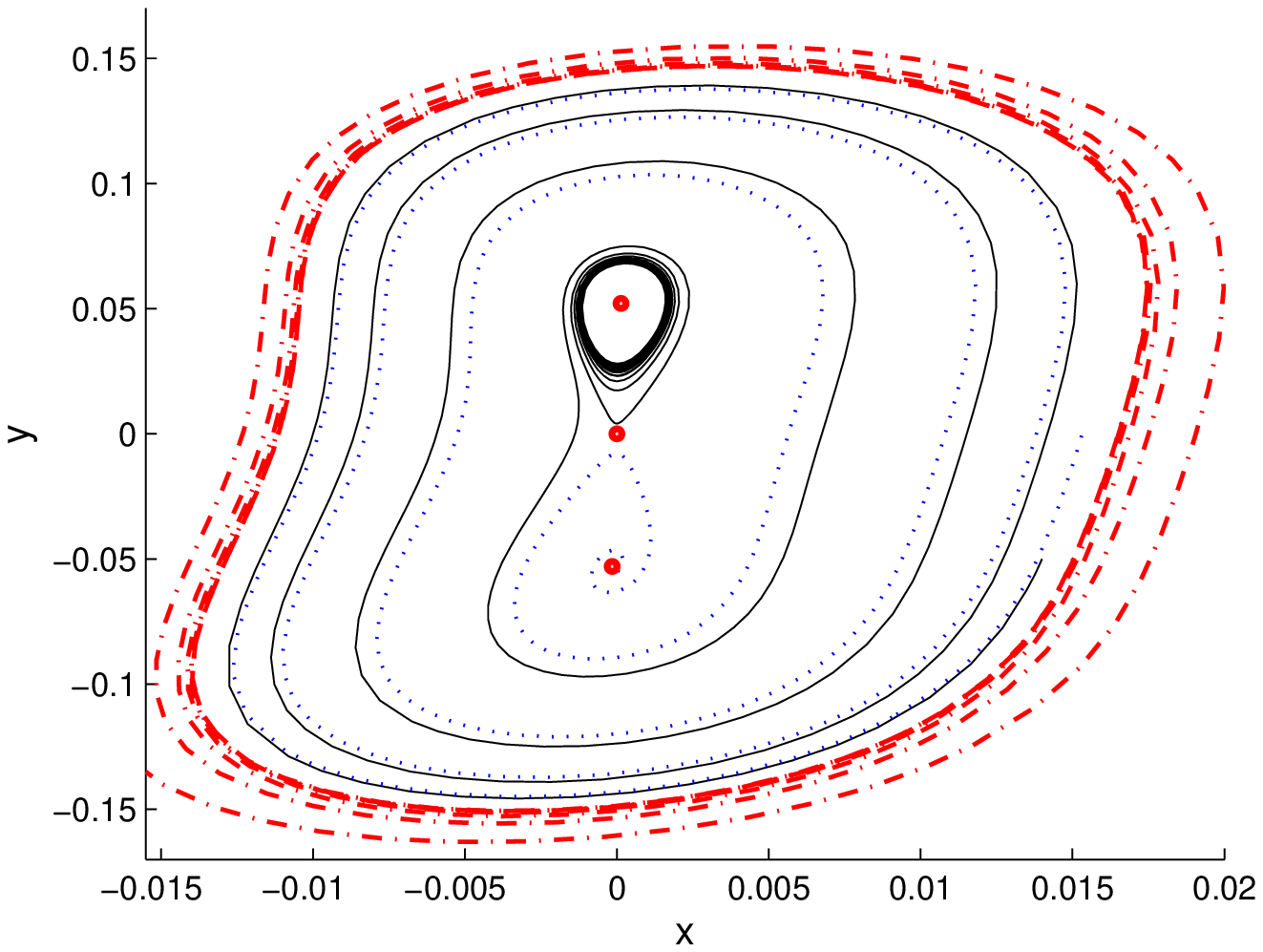}}
\subfigure[A saddle, a spiral source and a sink are inside an unstable limit cycle.]
{\includegraphics[width=.28\columnwidth,height=.18\columnwidth]{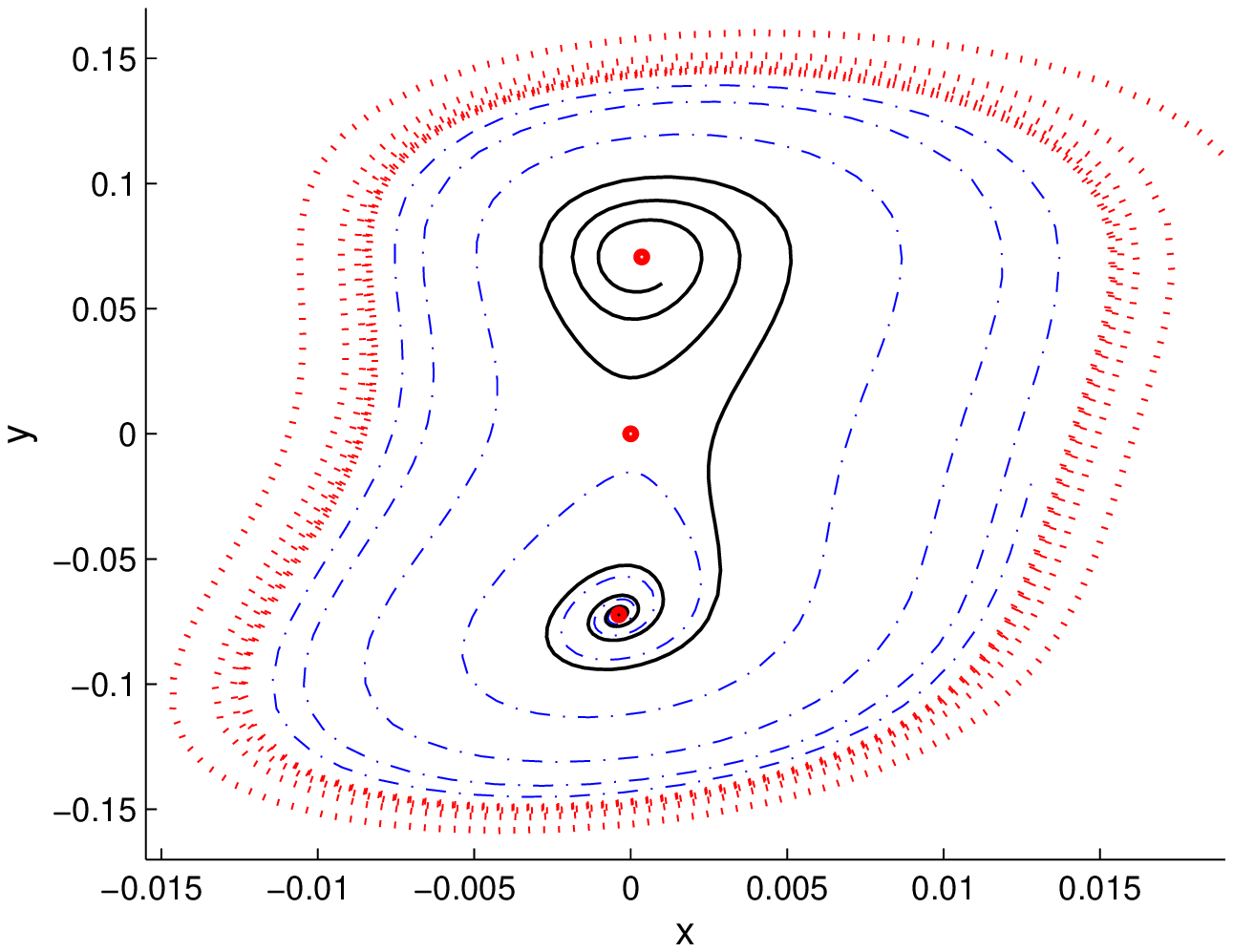}}
\subfigure[There are a saddle, a source and a sink.]
{\includegraphics[width=.18\columnwidth,height=.18\columnwidth]{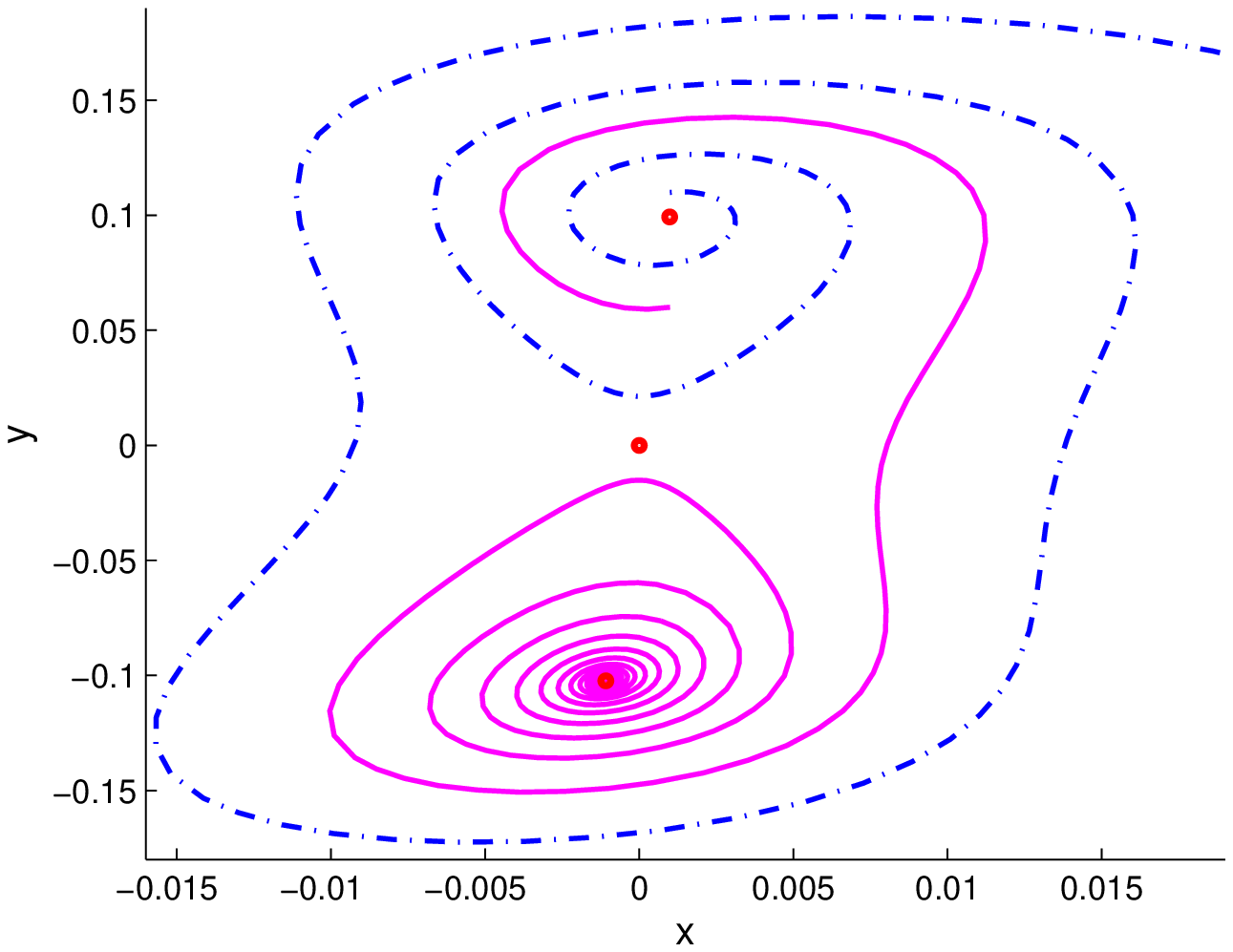}}
\subfigure[There are two spiral sources, a saddle and a stable limit cycle.]
{\includegraphics[width=.27\columnwidth,height=.18\columnwidth]{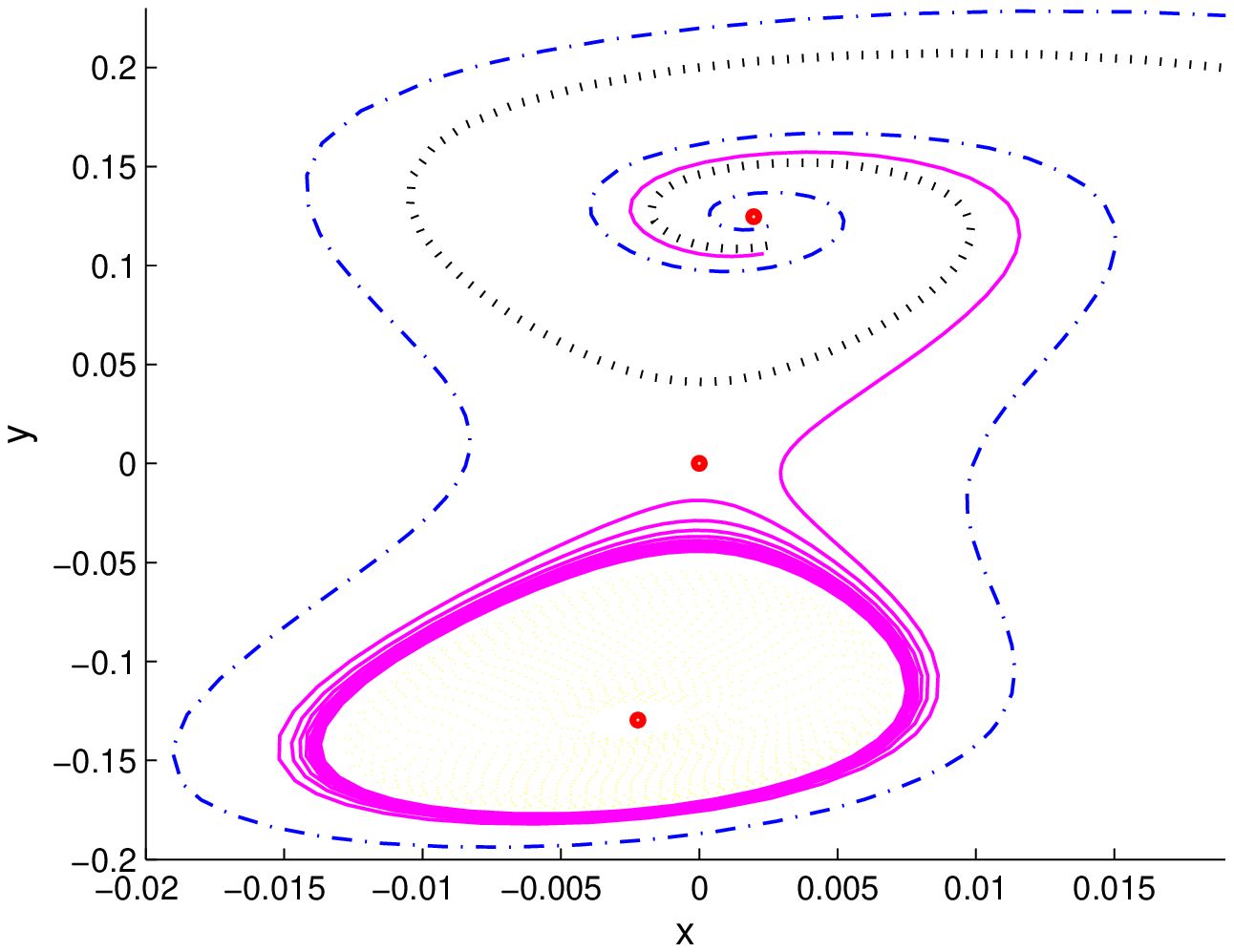}}
\subfigure[There are two spiral sources and a saddle.\label{Fig10(h)}]
{\includegraphics[width=.2\columnwidth,height=.18\columnwidth]{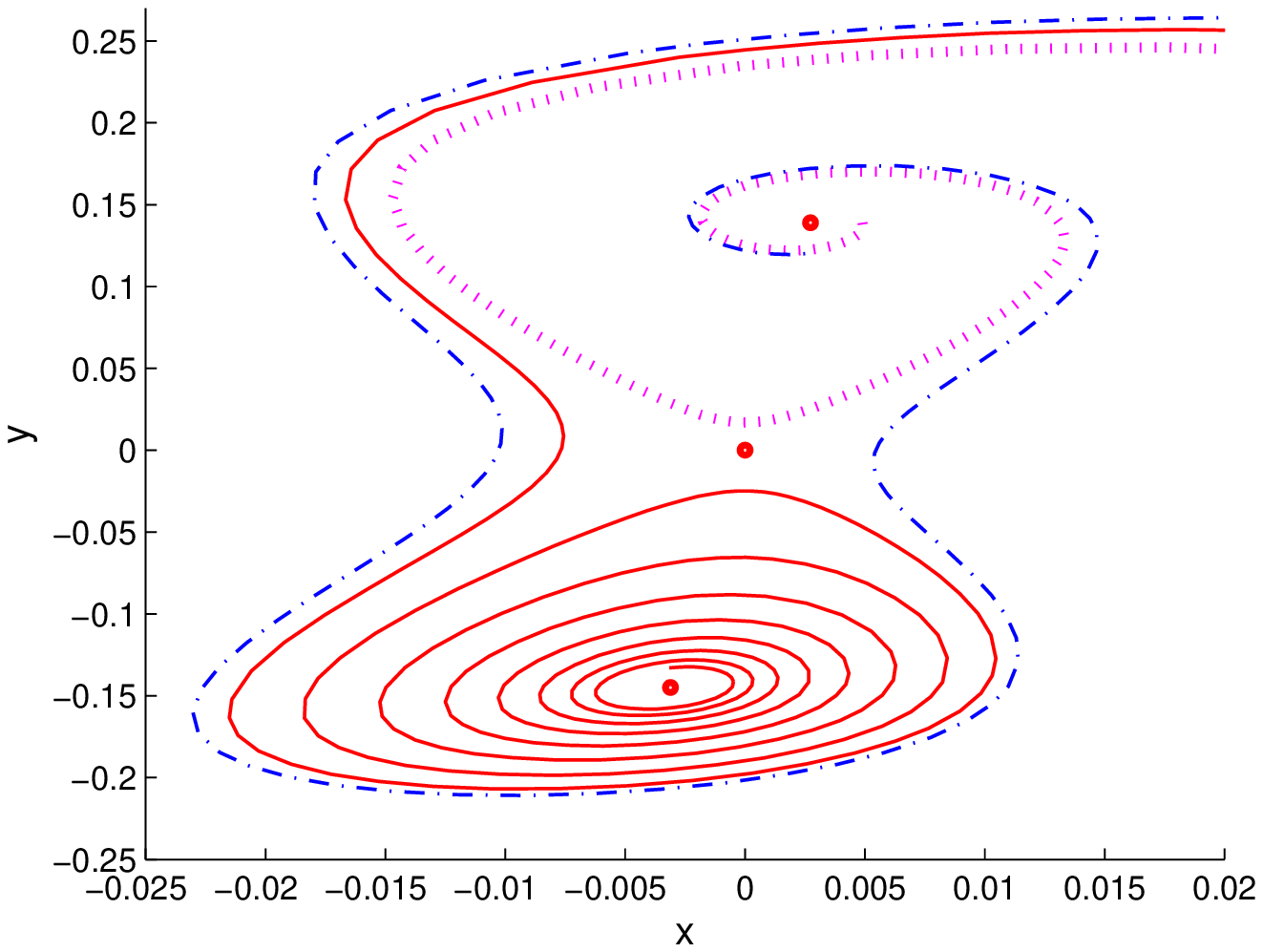}}
\caption{The numerical controlled phase portraits of the plant \eqref{BfControlr2s2}. Figures \ref{Fig10(a)}-\ref{Fig10(h)} are associated with input parameters chosen from regions labeled (a)-(h) in Figure \ref{Fig9(a)}. }
\end{center}
\end{figure}

\noindent Hence the bifurcation controller transition sets of the system \eqref{BfControlr2s2} are derived as
\bas
T_{H}&=& \Big\{(\mu_1, \mu_3, \mu_5)\,|\, \mu_1=0\Big\},\qquad\qquad T_{P}= \Big\{(\mu_1, \mu_3, \mu_5)\,|\, \mu_3=0\Big\},\\
T_{HmC}&=&\Big\{(\mu_1, \mu_3, \mu_5)\,|\, \frac{1}{2} \mu_1- \Big(\frac{8}{5}+{\frac {43}{120}} \mu_1-\frac{1}{8} \mu_1 d_8\Big) \mu_3-\Big( {\frac {4}{15}}+\frac{4}{5} d_8 \Big){\mu_3}^{2}+{\frac {37}{80}}{\mu_1}^{2}=0\Big\},\\
T_{HmC_{\pm}}&=& \Big\{(\mu_1, \mu_3, \mu_5)\,|\, \frac{1}{2}\mu_1\mp\bigg({\frac {17\sqrt {2}\pi}{800}} \mu_1+{\frac {9\sqrt {2}\pi}{320}}\bigg) \sqrt {-\mu_3}-\frac{8}{5}\mu_3\mp{\frac {99\sqrt {2}\pi}{3200}} \mu_3\sqrt {-\mu_3}=0\Big\},
\eas
while the controller curve \({T_{H_{\pm}}}\) in the input space \((\mu_1, \mu_3, \mu_5)\) follow
\bas
&\bigg(\big(\frac{173923}{129600} -\frac{751}{1500} d_8\big) {\mu_1}^2+\big( {\frac {197}{800}}d_8+{\frac {112303}{36000}} \big) \mu_1 -{\frac {99}{25}} \bigg) \mu_3-\Big( \big( {\frac {290969}{270000}}+{\frac {108791}{36000}}d_8 \big)
 \mu_1-{\frac {143}{250}}-{\frac {99}{50}}d_8
\Big) {\mu_3}^2&
\\&\mp\bigg( \Big( {\frac {262636021}{1152000000}}-{\frac {16287331}{76800000}}
d_8-{\frac {597}{256000}}{d_8}^{2} \Big) {\mu_1}^{2}+\Big( {\frac {13997}{40000}}+{\frac {597}{8000}}d_8\Big) \mu_1
-{\frac {9331477}{14400000}}-{\frac {219371}{960000}}d_8 \bigg)(-\mu_3)^{\frac{3}{2}}
&\\ &\pm\bigg({\frac {597}{2000}}-{\frac {69461}{240000}}\mu_1+\big({\frac {597}{32000}}d_8-{\frac {30172573}{57600000}}\big) {\mu_1}^2\bigg)\sqrt {-\mu_3}+{\frac {2411}{24000}}{\mu_1}^2+{\frac {197}{200}}\mu_1=0.\qquad \qquad\qquad\qquad &
\eas
With reassignments of \(d_1:= \pm1\) and \(d_3:= d_6:=-1,\) we obtain the cases \((a_2=\pm1, b_2=-1)\) and the estimated heteroclinic variety
\bas
T_{HtC}&=&\left\{ (\mu_1, \mu_3, \mu_5)\,\big|\, \frac{1}{2}\mu_1\pm\left(\frac{2}{5}+\frac{45\pm52}{120}\mu_1+\frac{1}{8}d_8 \mu_1\right) \mu_3\mp\frac{3}{80}{\mu_1}^2+\left(\frac{9\pm8}{15}+\frac{1}{5}d_8\right) {\mu_3}^2=0\right\}.\eas

In order to illustrate the numerical \(\mathbb{Z}_2\)-breaking controller bifurcation varieties we choose \(d_8:=1,\) the controller input \(\mu_5:=\pm0.3,\) and obtain Figures \ref{Fig9(a)}-\ref{Fig9(d)}. These numerical transition varieties are highly accurate over the plotted intervals. Next for each input pair (\(\mu_1, \mu_3\)) of values
\bas
&(0.005, 0.005), (-0.005, 0.005), (-0.025, -0.0014), (-0.025, -0.0027), &\\\nonumber
& (-0.025, -0.005), (-0.025, -0.01), (-0.025, -0.016), (-0.025, -0.02)&
\eas chosen from each region labeled (a)-(h) in Figure \ref{Fig9(a)}, we depict the associated numerical controlled phase portraits in Figures \ref{Fig10(a)}-\ref{Fig10(h)}, respectively. The parameters chosen from the area surrounding the region (i) in Figure \ref{Fig9(a)} violate the assumption \eqref{Restrct} and thus, we skip their phase portrait. The qualitative dynamics (except \(\mathbb{Z}_2\)-symmetry) associated with the regions in Figure \ref{Fig9(b)} are similar to those in Figures \ref{Fig8(a)}-\ref{Fig8(g)}.

\section{Applications of bifurcation control }\label{sec8}

This section is devoted to illustrate the applicability of our results in engineering control problems. We first demonstrate two important engineering applications of our bifurcation control analysis; they are tracking and regulating controller designs. Then in subsection \ref{ShipCourse}, we apply our results to two nonlinear ship course models for a controller design in a ship steering system.

\subsection{Regulating and tracking controller designs }\label{BifCont}

\begin{figure}
\begin{center}
\subfigure
{\includegraphics[width=.4\columnwidth,height=.20\columnwidth]{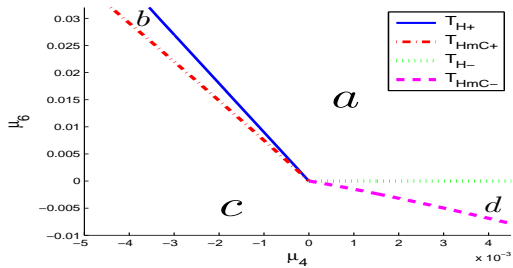}}
\caption{Numerical transition sets for the controller inputs \((\mu_4, \mu_6)\) from the controlled system \eqref{BfControlr1s1}, \eqref{Uncontrollable}. A regulating approach for linearly uncontrollable equilibria using bifurcation control analysis. }\label{section72}
\end{center}
\end{figure}

This section is devoted to explain bifurcation control's application in regularizing an equilibrium and in solving an engineering tracking problem. Our proposed approach is applicable to the cases uncontrollable by the classical input-state feedback linearization and back-stepping methods in nonlinear control theory; see \cite{Slotine}. Kang {\em et. al.,} \cite{KangIEEE} studied a nonlinear system with an uncontrollable linearization at the origin. They proved that a generic system of this type has a nearby controllable equilibrium. Then, they suggested a regularization approach to stabilize an uncontrollable equilibrium by moving it into a nearby linearly controllable equilibrium. The latter enables {\it gain scheduling} possibilities and works fine for many engineering applications; see \cite{KangIEEE}.
We recall that a nonlinear system in the vicinity of an equilibrium is called {\it linearly controllable} when its linearization satisfies Kalman's controllability condition. Note that any possible change in the linear controllability of an equilibrium is called a control bifurcation; see \cite{KangIEEE}. Our bifurcation control approach readily contributes into a systematic and symbolic regularization of this type, when the system falls within the case \(r=s=1.\) In order to illustrate this, consider a (linearly uncontrollable) plant that is given by equation \eqref{BfControlr1s1},
\be\label{Uncontrollable} u_1:=0\quad \hbox{ and } \quad  u_2:=v +\mu_4+\mu_6y,\ee
where the constants follow the equation \eqref{di_s}, except \(d_1:= d_5:=-1\). Then equation \eqref{r1s1NF} is a time-reversed truncated parametric normal form for this system, where
\bes
a_1=b_1=1, \quad b_3=\frac{-69}{350}, \quad \nu_1=- \frac{9}{4}{\mu_4}^2, \qquad \nu_2=- \frac{1}{4}(9\mu_4+2\mu_6).
\ees The transition varieties are given by
\bas
&T_{H\pm}=\left\{(\mu_4, \mu_6)\,|\, \mu_6=-\frac{9}{2}\mu_4\pm\frac{9}{2}\mu_4-{\frac {27}{32}}{\mu_4}^2\pm{\frac{222507}{8960}}{\mu_4}^3\right\} \qquad \hbox{ and } \qquad &\\
&T_{HmC\pm}= \left\{(\mu_4, \mu_6)\,|\, \mu_6=-\frac{9}{2}\mu_4\pm\frac{15\mu_4}{14} \sqrt{9- \frac{619\sqrt{3}}{28}\sqrt{|\mu_4|}\pm \frac{2041875}{50176}\mu_4}\right\}.&
\eas These are plotted in figure \ref{section72}. Since \(\nu_1<0\) for any \(\mu_4\neq0,\) each nonzero input choice for \((\mu_4, \mu_6),\) that is taken from either of the regions \((a),\) \((b),\) \((c),\) and \((d)\) in figure \ref{section72}, gives rise to two linearly controllable local equilibria of a saddle and one of either a source or a sink type.

\begin{figure}
\begin{center}
\subfigure[\label{phaseportrraitr2s2b}]
{\includegraphics[width=.3\columnwidth,height=.2\columnwidth]{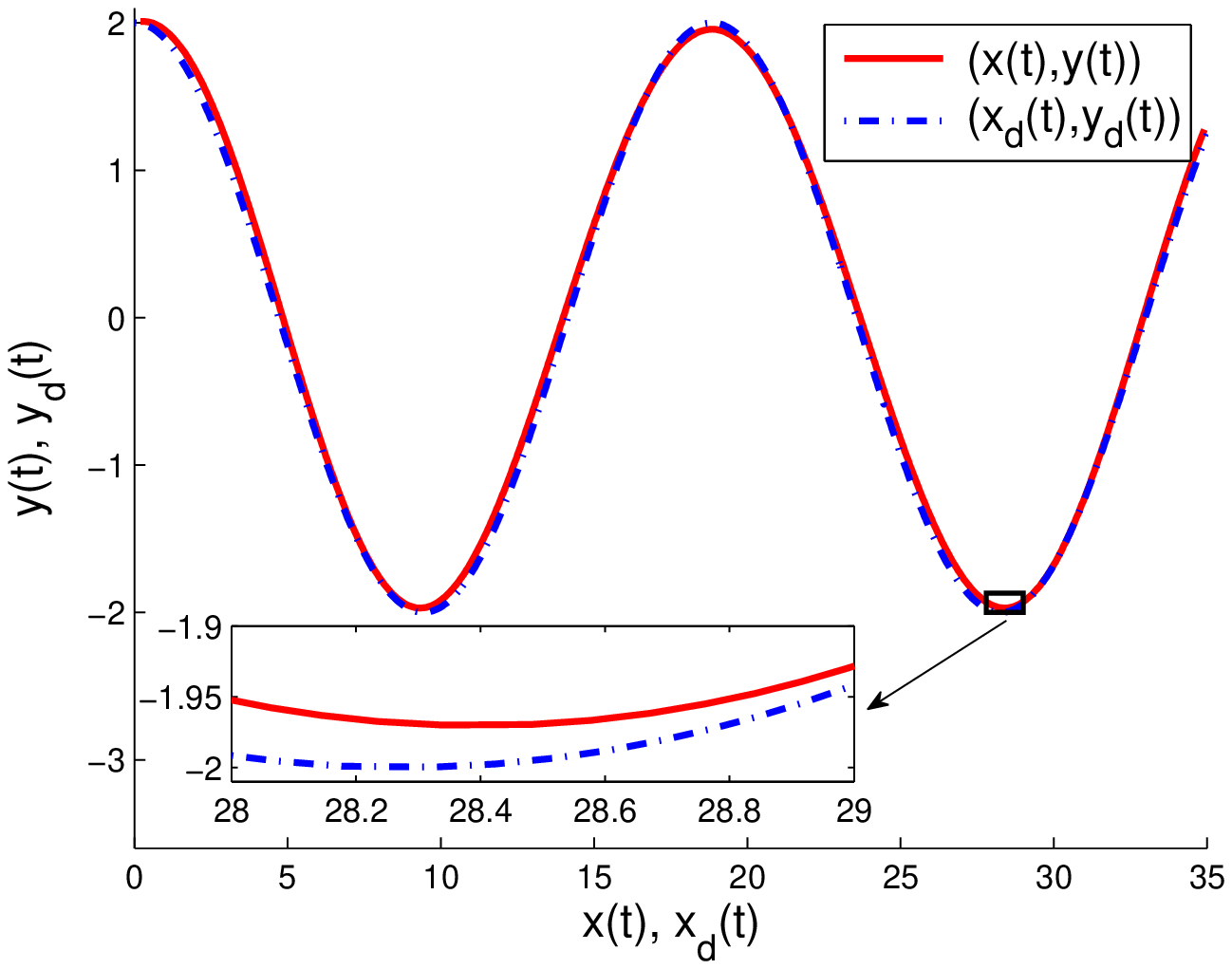}}
\subfigure[\label{r2s2Fig(b)}]
{\includegraphics[width=.3\columnwidth,height=.2\columnwidth]{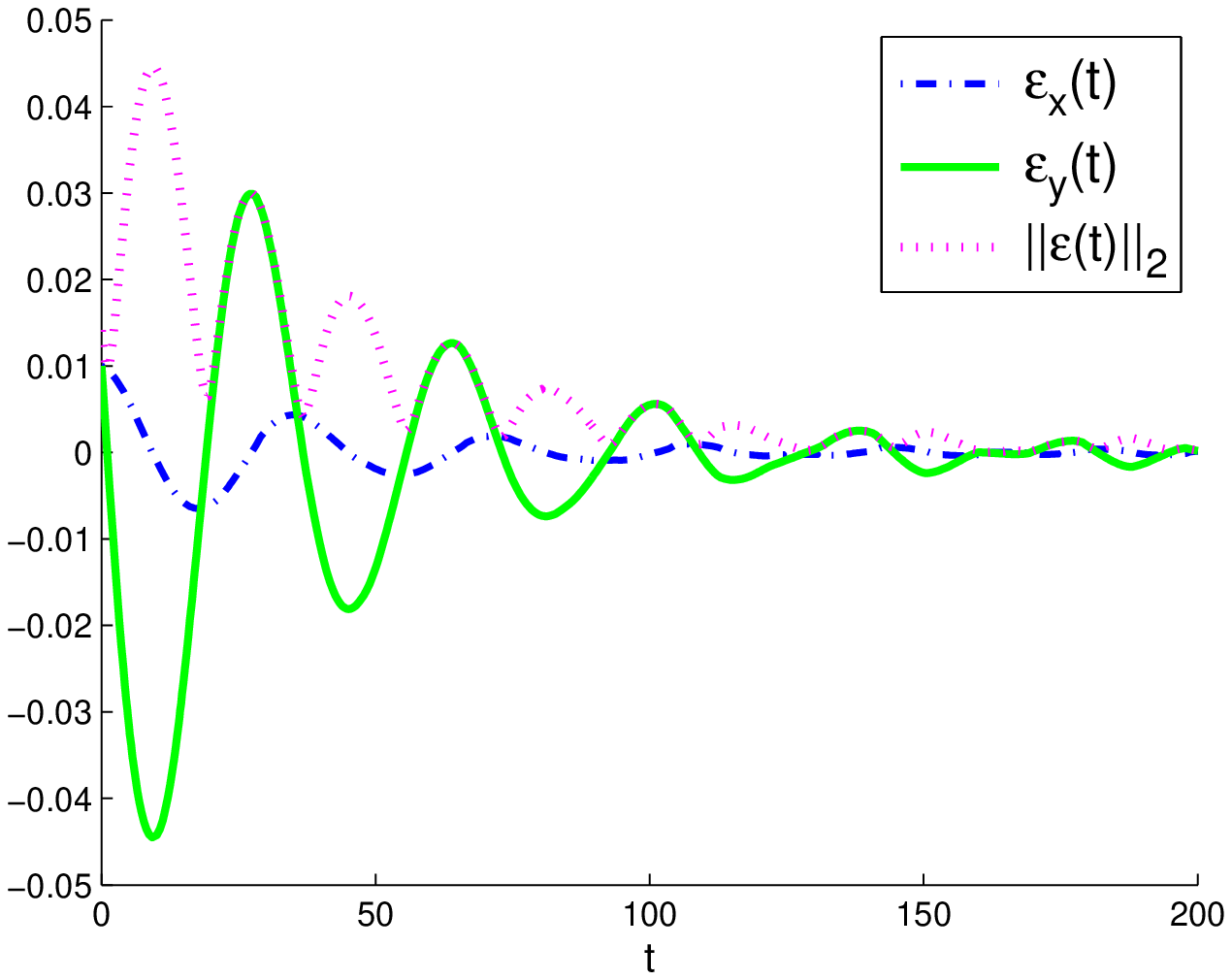}}
\subfigure[A time-inverse orbit. \label{LimitCycleB}]
{\includegraphics[width=.3\columnwidth,height=.2\columnwidth]{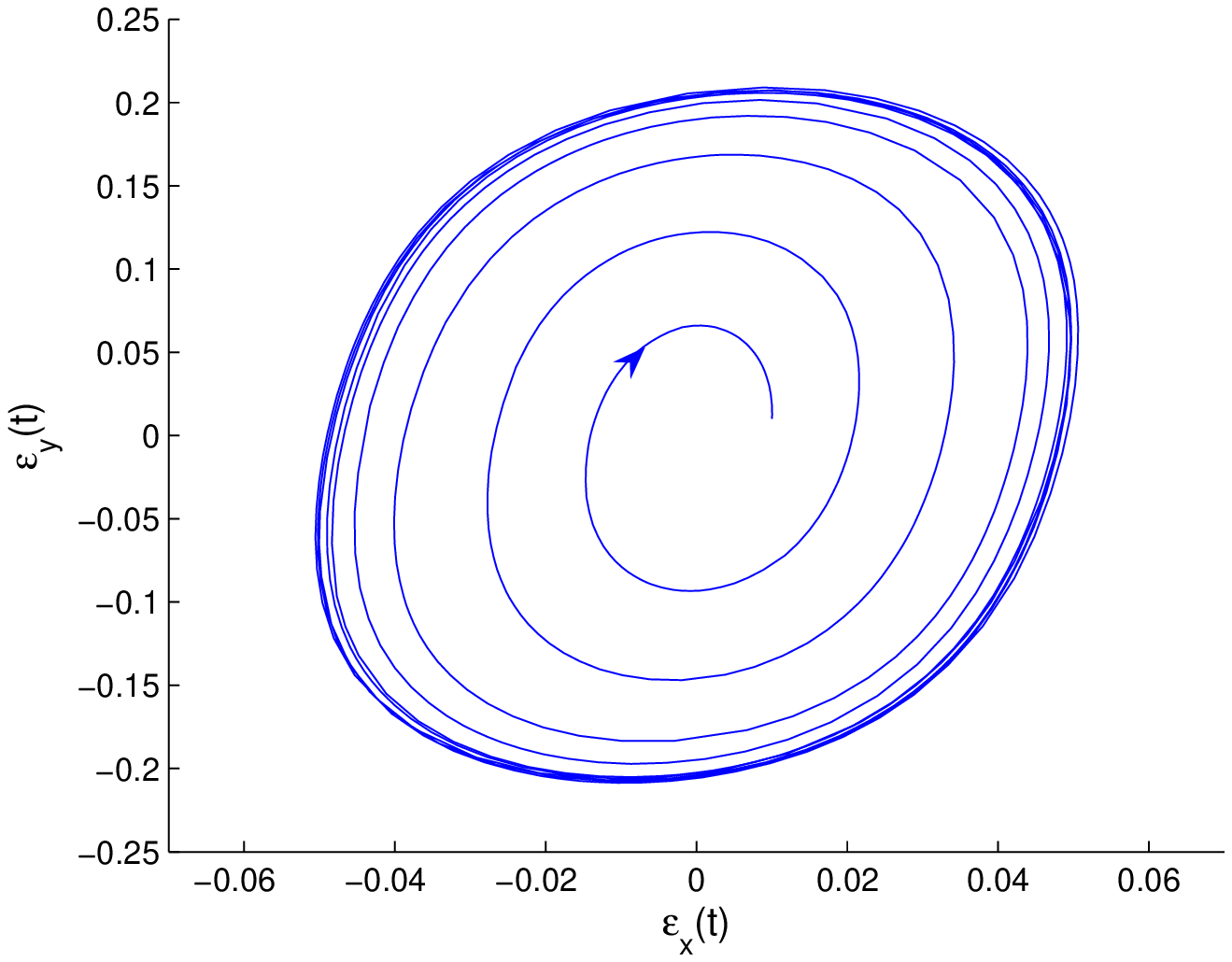}}
\subfigure[\label{phaseportrraitr2s2c}]
{\includegraphics[width=.3\columnwidth,height=.2\columnwidth]{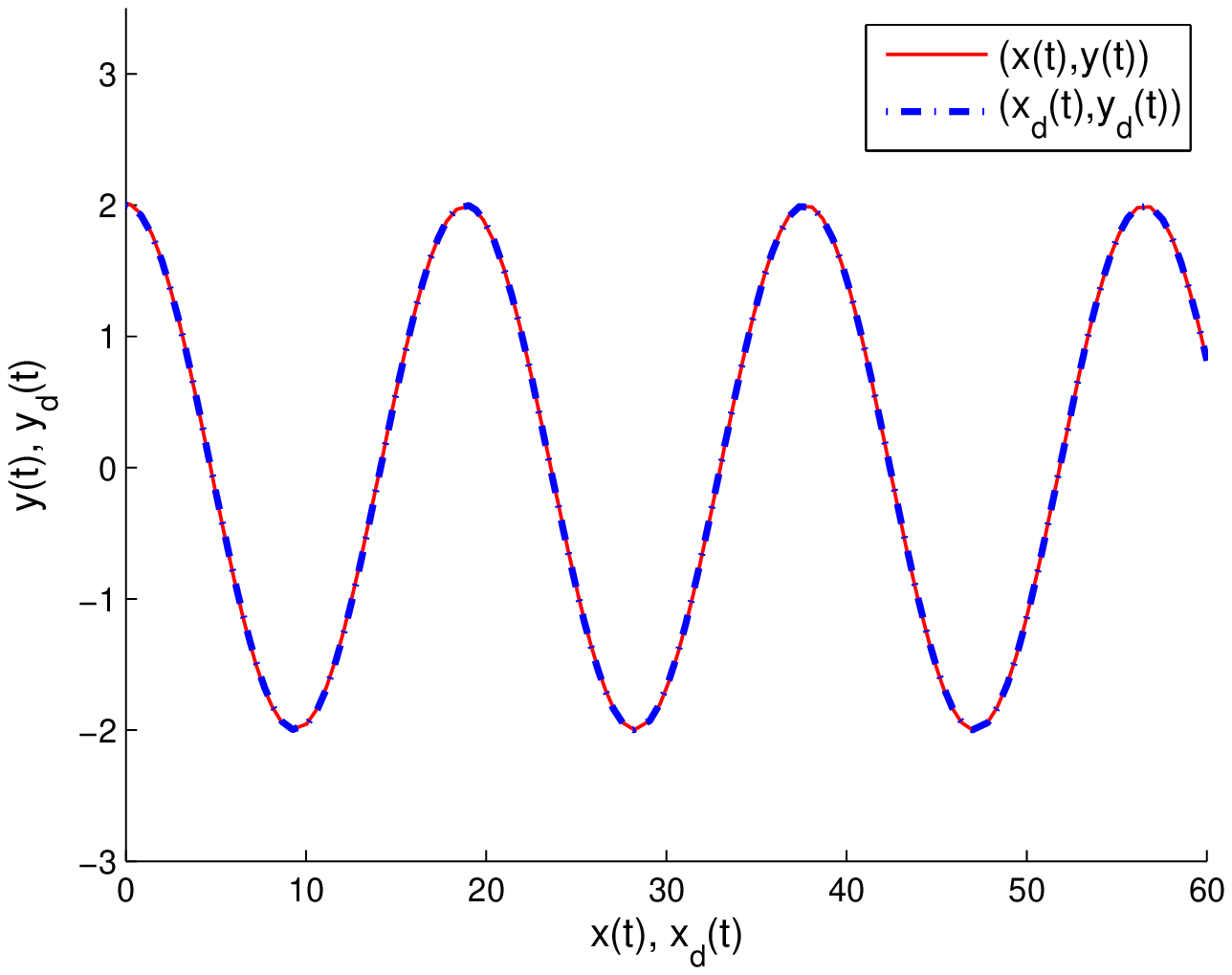}}
\subfigure[\label{r2s2Fig(c)}]
{\includegraphics[width=.3\columnwidth,height=.2\columnwidth]{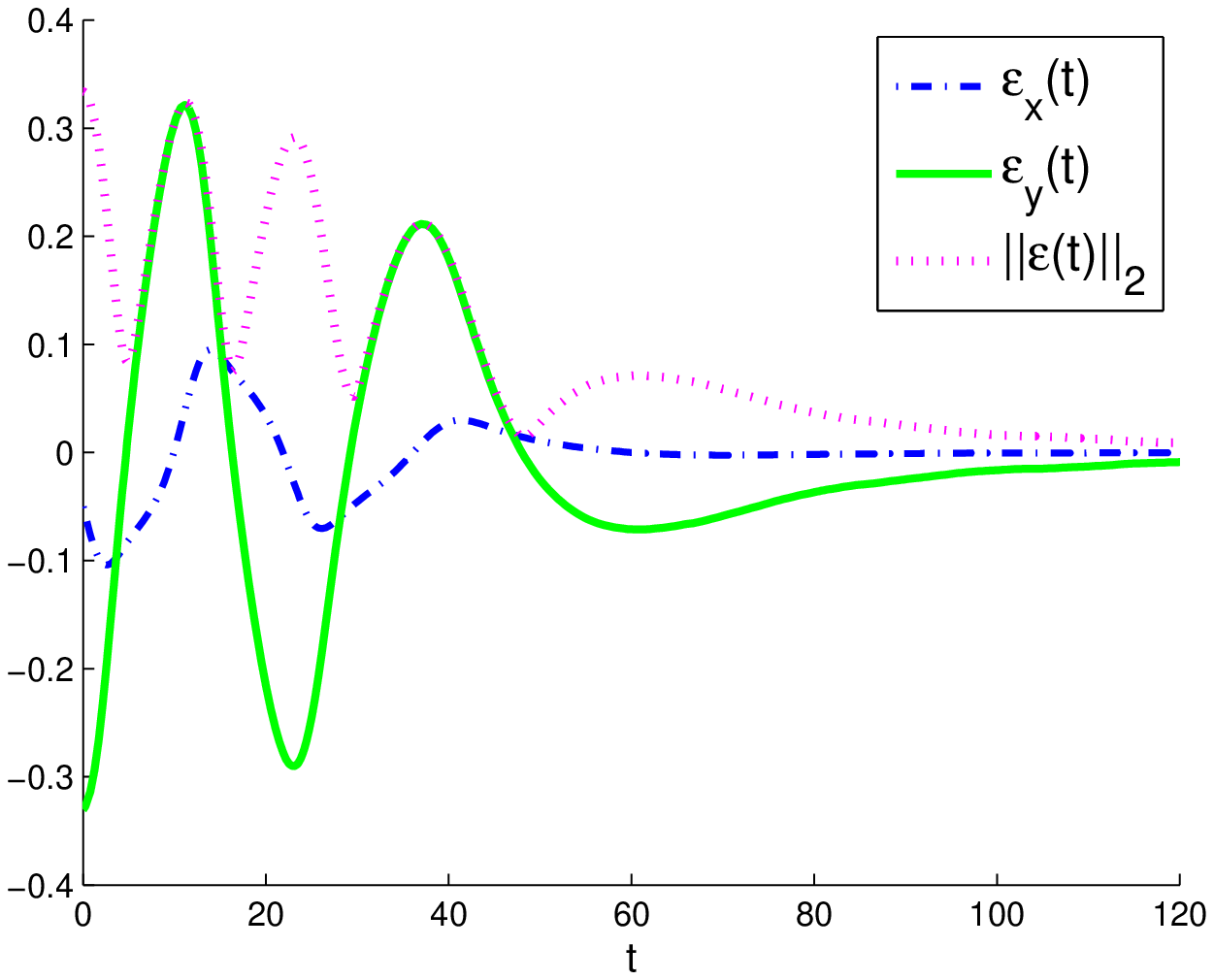}}
\subfigure[A time-inverse orbit.\label{LimitCycleC}]
{\includegraphics[width=.3\columnwidth,height=.2\columnwidth]{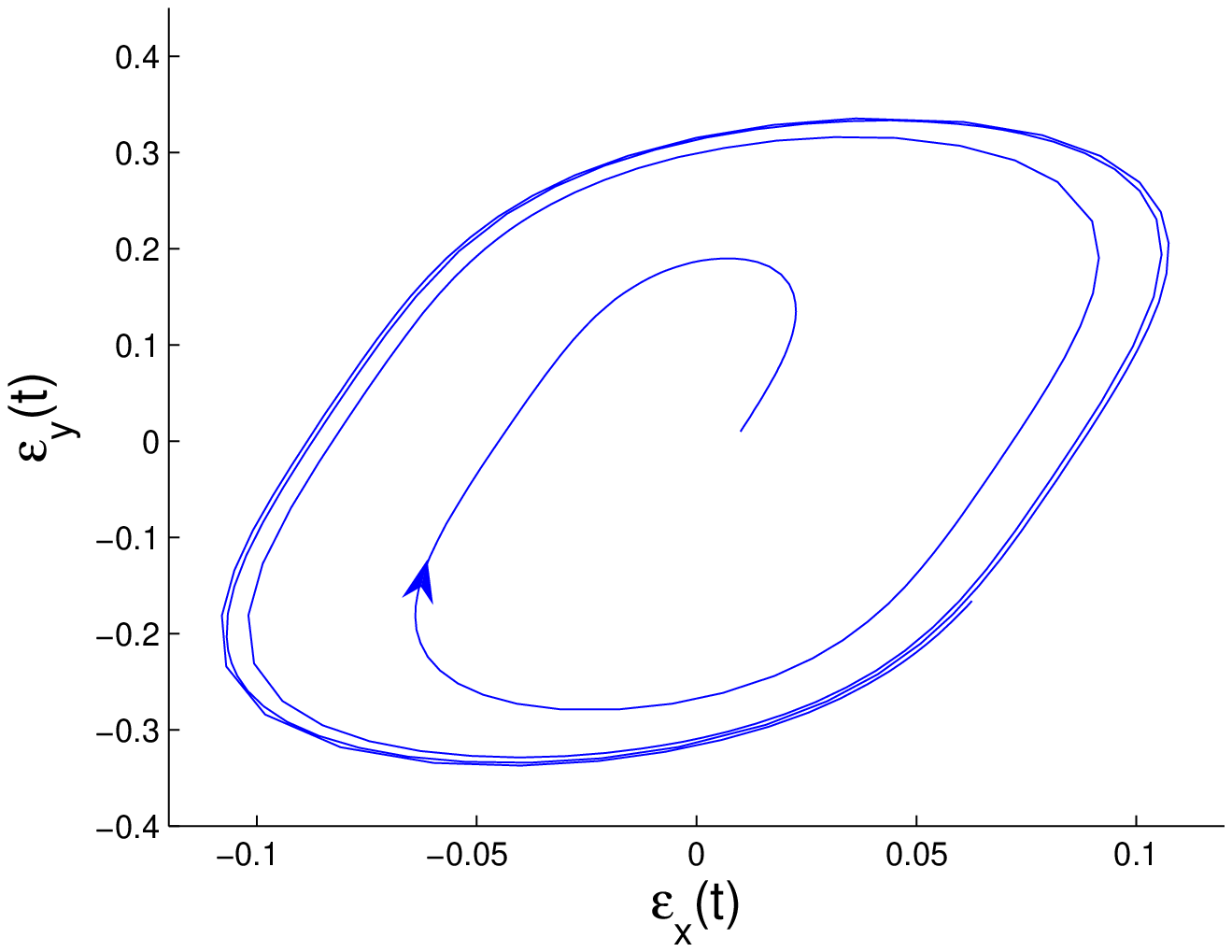}}
\caption{These figures illustrate the error trajectories associated with the tracking problem that is given by equations \eqref{BfControlr2s2},
\eqref{Trackr2s2}, and \eqref{DesiredSol}. Here the tracking problem is solved by a bifurcation control analysis. Figures \ref{phaseportrraitr2s2b}-\ref{r2s2Fig(b)}-\ref{LimitCycleB} and \ref{phaseportrraitr2s2c}-\ref{r2s2Fig(c)}-\ref{LimitCycleC} are associated with parameters taken from regions (b) and (c) in figure \ref{Fig6(a)}, respectively.}
\end{center}
\end{figure}

A pole placement method can be used for a singular system with a linearly controllable equilibrium in order to stabilize the system in the vicinity of the equilibrium. However, a crude use of pole placement method without controlling nonlinear terms is not suitable for many nonlinear systems; \eg see the nonlinear back stepping and the gain scheduling literatures, and \cite{Kang92Extnd}. Indeed, poles close to the imaginary axes induce undesired potential bifurcations while placing the poles far distanced from the imaginary axis substantially reduces the equilibrium's basin of attraction and thus, the system will be noise sensitive with amplified undershoots and overshoots. Furthermore, singularity of a system is not a {\it negative} characteristic for many engineering applications where (smooth) manoeuvrability and flexility are of the prime concerns. Hence, the singularity needs to be utilized rather than being destroyed by a crude use of the classical pole placement method. Our approach is useful and applicable for both linearly controllable and linearly uncontrollable systems. We may somehow say that bifurcation controllers for a regulating design make a careful choice for pole {\it placement} by nonlinear controllers. Indeed, poles of a stable bifurcated equilibrium in a Bogdanov-Takens singular system usually have small negative real parts and thus in our opinion, the equilibrium's basin of attraction is as large as practically possible. Hence, it is not much sensitive to noises. This is also because nonlinear bifurcation controllers (through universal asymptotic unfolding) control the effects of nonlinear terms, uncertainties and imperfections in the nonlinear modeling and measurements. Further, it has a sufficient attraction rate for many engineering applications with a {\it smooth manoeuvring} possibilities; \eg see subsection \ref{ShipCourse}.

The bifurcation control approach of the degenerate case provides various regularization choices and works well to solve the tracking engineering problems. These are possible through the list of qualitative dynamics obtained by the bifurcation control analysis. In order to illustrate this, we consider a tracking problem for a system of the case \(r=s=2\), that is to design a controller so that the solution of the controlled system follows a predefined desired solution, say \((x_d(t), y_d(t)).\) Consider the \(\mathbb{Z}_2\)-equivariant plant \eqref{BfControlr2s2} while the controllers \(u_1\) and \(u_{2}\) are defined below. Next, we define the error functions \(\varepsilon_x(t), \varepsilon_y(t)\) and \(\varepsilon(t)\) by
\bes \varepsilon(t):=(\varepsilon_x(t), \varepsilon_y(t)):= \big(x(t)- x_d(t), y(t)- y_d(t)\big)\ees and assume that \(||(\varepsilon_x(t), \varepsilon_y(t))||_2\) remains sufficiently small. Let
\be\label{Trackr2s2} u_{1}:= u_{11}+ u_{12} \quad \hbox{ and } \quad u_{12}:=\mu_1 x+\mu_3y.\ee
Now we introduce the controllers \(u_{11}\) and \(u_2\) so that the errors' dynamics follow a time-invariant differential system whose parametric normal form is given by the parametric system \eqref{why}. Indeed, by substituting \((\varepsilon_x+x_d, \varepsilon_y+y_d)\) for \((x, y)\) we derive
\ba\nonumber
u_{11}&=&{\dot{x_d}(t)}-d_{2}{y_d(t)}^{3}+(3d_{2}y+d_{{3}}x-d_{{3}}{x_d(t)}) {y_d(t)}^{2}-(2d_{4}xy+\mu_{1}+3d_{1}{x}^{2}+d_{3}{y}^{2}){x_d(t)}-d_{{1}}{x_d(t)}^{3}\\\nonumber&&
-\big(d_{{4}}{x_d(t)}^{2}- (2d_{{4}}x+2d_{{3}}y){x_d(t)}+d_{{4}}{x}^{2}+\mu_{{3}}+3d_{{2}}{y}^{2}+2d_{{3}}xy\big)y_d(t)+(3d_{{1}}x+d_{4}y) {x_d(t)}^{2},\\\nonumber
u_2&=&\dot{y_d}(t)-d_{{6}}{y_d(t)}^{3}+ \big(d_{{7}}x+3d_{{6}}y-d_{{7}}{x_d(t)}\big) {y_d(t)}^{2}+ (1-d_{{7}}{y}^{2}-2d_{8}xy-3d_{5}{x}^{2}) {x_d(t)}-d_{5}{x_d(t)}^{3}\\\label{DesiredController}&&
- \big(d_{{8}}{x_d(t)}^{2}-(2d_{{8}}x+2d_{{7}}y ) {x_d(t)}
+d_{{8}}{x}^{2}+2d_{{7}}xy+3d_{{6}}{y}^{2}\big)y_d(t)
+ (3d_{5}x+d_{{8}}y ) {x_d(t)}^{2}.
\ea Then the transition varieties \eqref{Transr2s2Z21}-\eqref{Transr2s2Z22} hold. For briefness, we take the same values for \(d_i\) as those of the case \(a_2=b_2=1\) in subsection \ref{subsecZ2}. Thereby, figure \ref{Fig6(a)} describes the transition sets for the errors' dynamics while their qualitative dynamics' list for different values of \((\mu_1, \mu_3),\) taken from \eqref{mu1mu3Values}, are depicted in figures \ref{Fig7(a)}-\ref{Fig7(i)}. Hence the parameter choices within the regions (b) and (c) in figure \ref{Fig6(a)} lead to a change of stability of the origin into an attracting node or a spiral sink; see figures \ref{Fig7(b)} and \ref{Fig7(c)}. In either of these cases, the limit cycle encircling the origin constitutes the origin's basin of attraction.

\begin{figure}
\begin{center}
\subfigure[Desired and tracking solutions. \label{Minus1r2s2PhaseA}]
{\includegraphics[width=.24\columnwidth,height=.18\columnwidth]{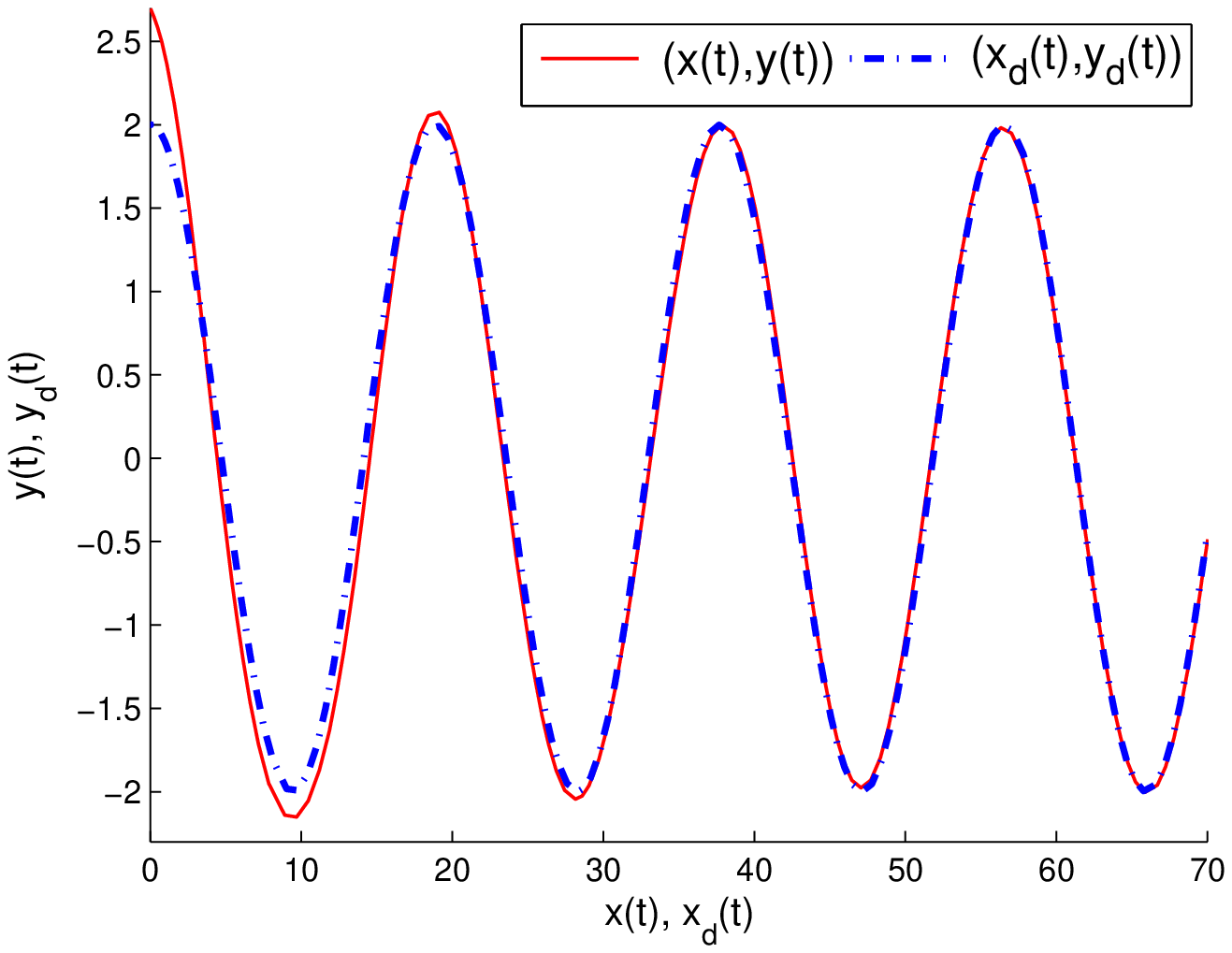}}
\subfigure[A stable limit cycle\label{Minus1LimitCycleB}]
{\includegraphics[width=.24\columnwidth,height=.18\columnwidth]{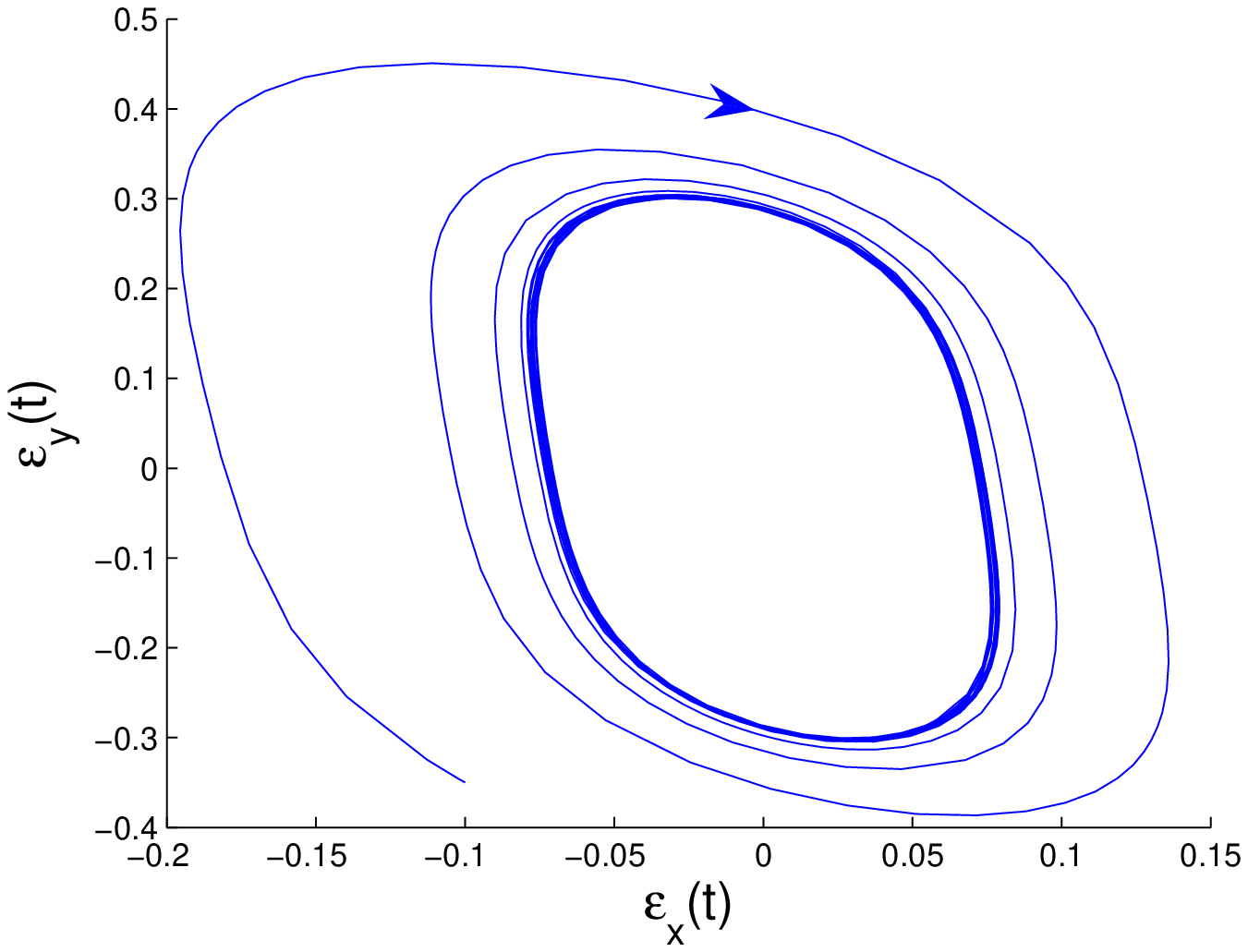}}
\subfigure[Tracking and desired orbits.\label{Minus1r2s2PhaseB}]
{\includegraphics[width=.24\columnwidth,height=.18\columnwidth]{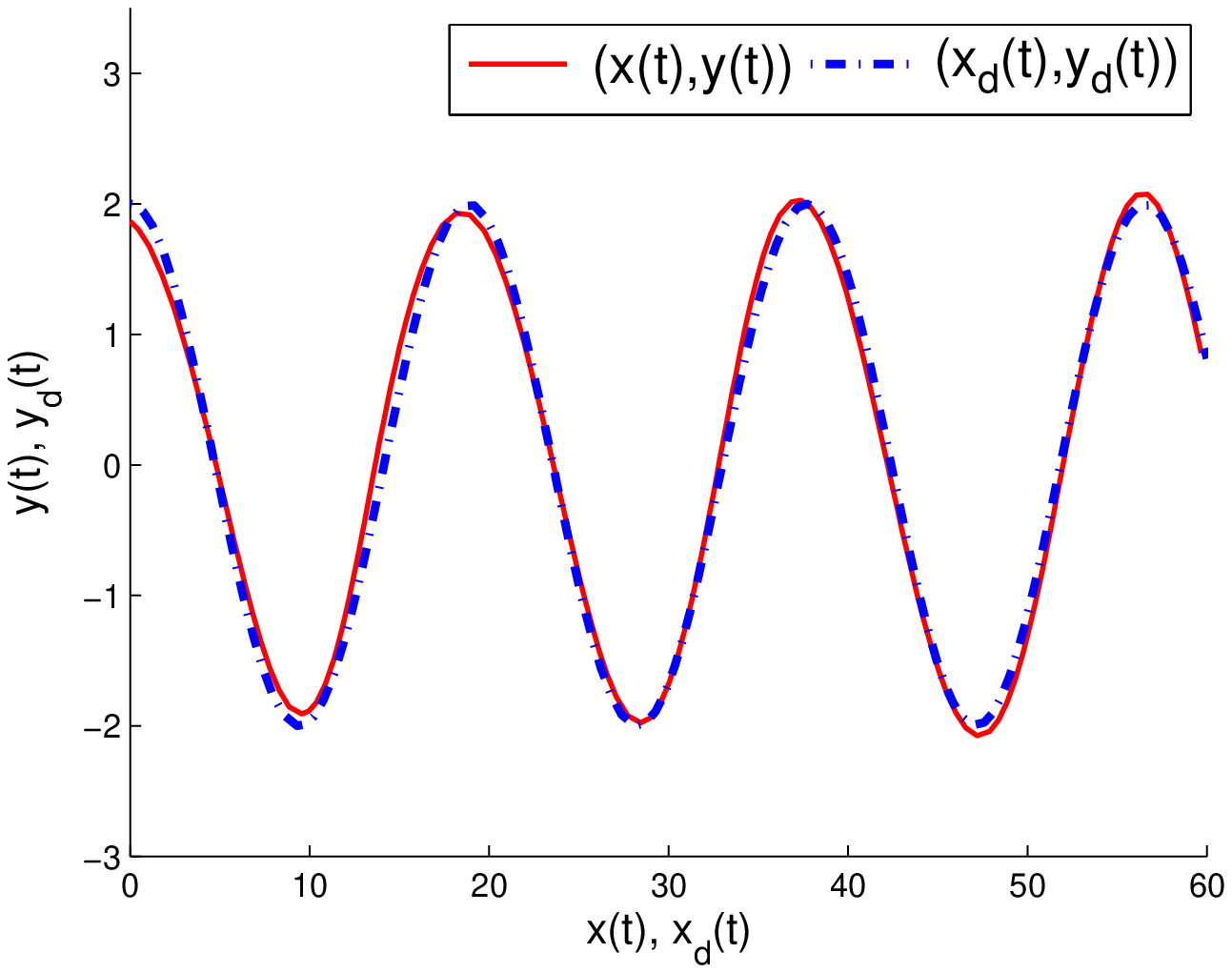}}
\subfigure[Oscillating error orbits. \label{Minus1r2s2TrajB}]
{\includegraphics[width=.24\columnwidth,height=.18\columnwidth]{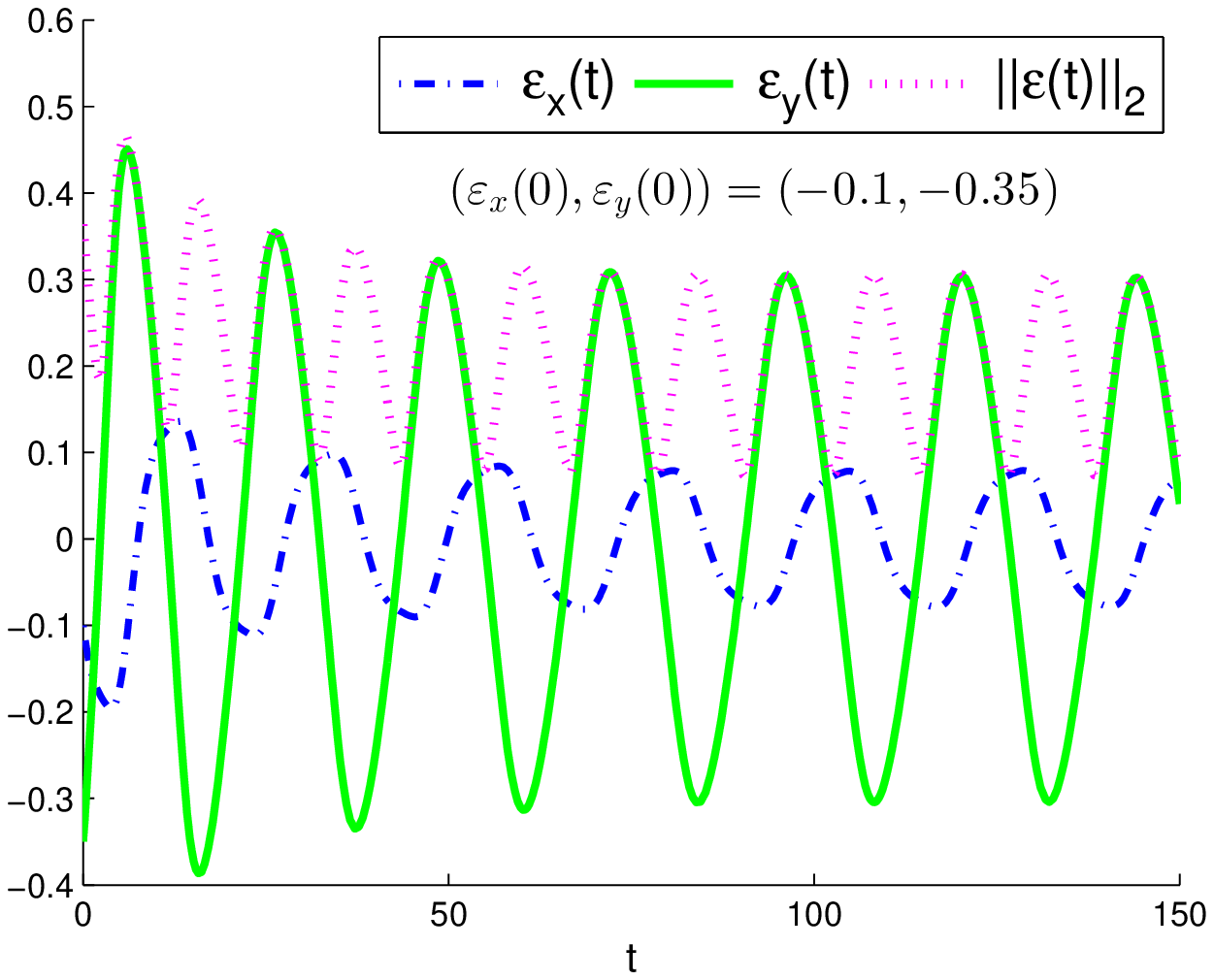}}
\subfigure[Error orbits. \label{Minus1r2s2TrajA}]
{\includegraphics[width=.24\columnwidth,height=.18\columnwidth]{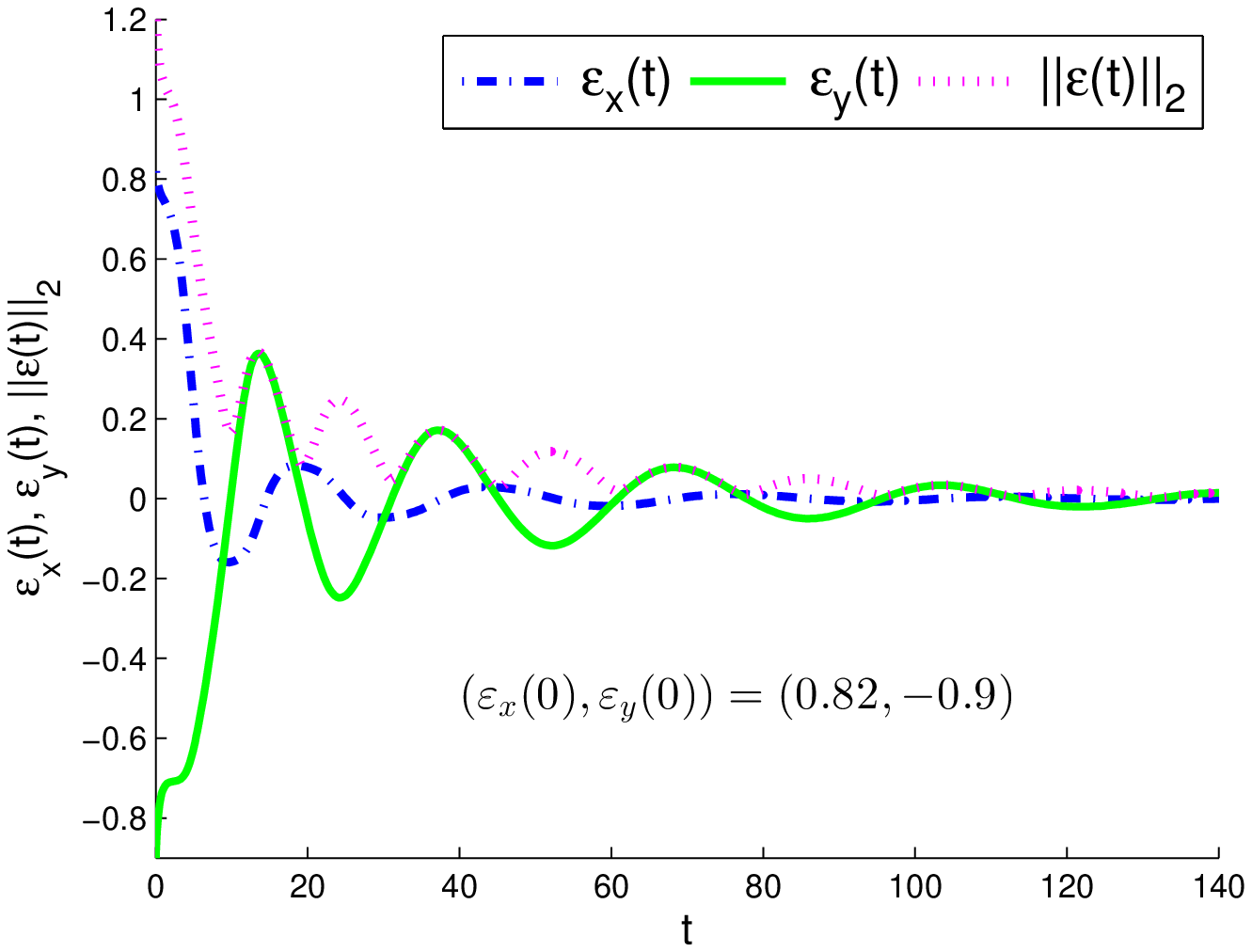}}
\subfigure[A stable limit cycle\label{Minus1LimitCycleC}]
{\includegraphics[width=.24\columnwidth,height=.18\columnwidth]{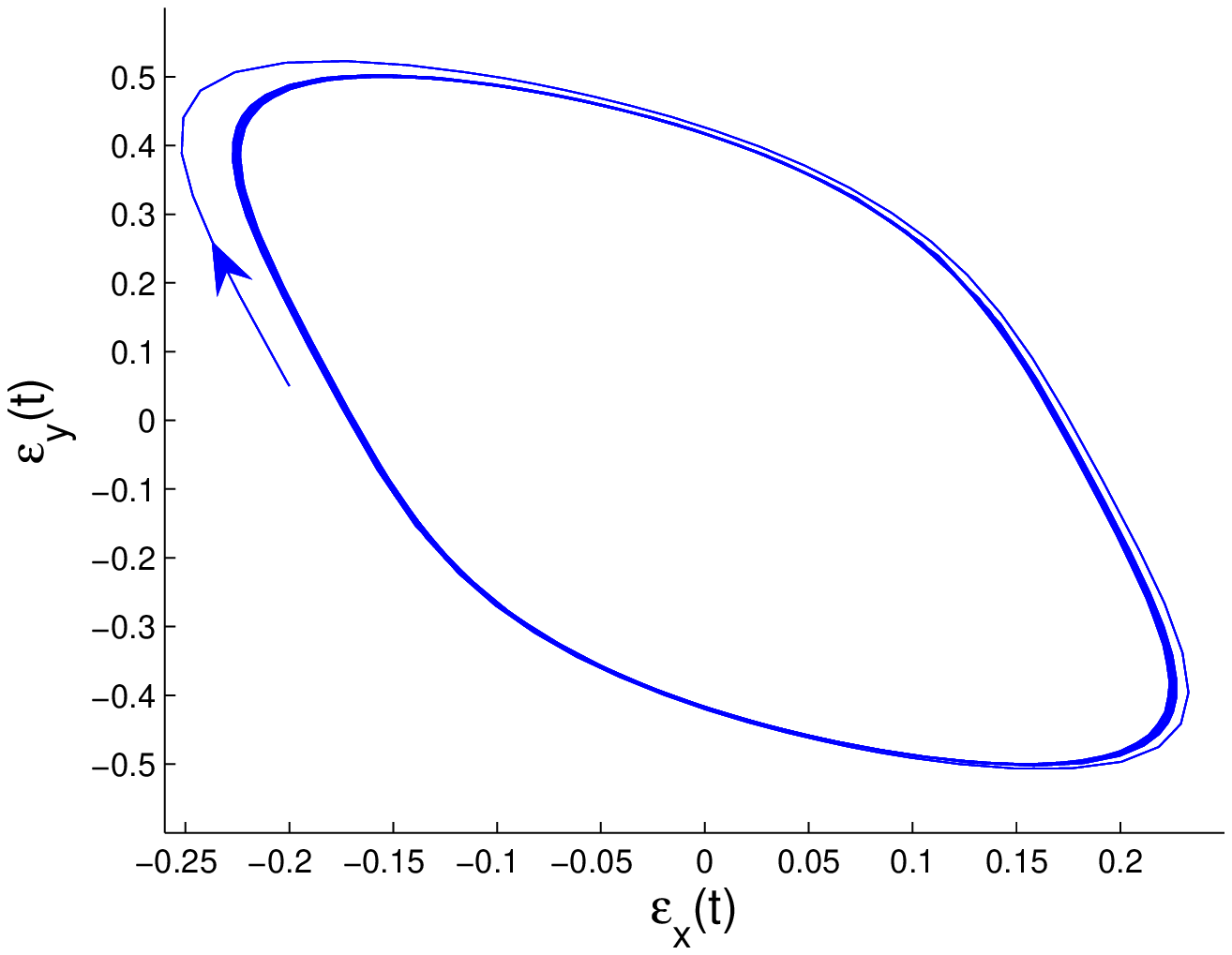}}
\subfigure[Desired and tracking trajectories. \label{Minus1r2s2PhaseC}]
{\includegraphics[width=.24\columnwidth,height=.18\columnwidth]{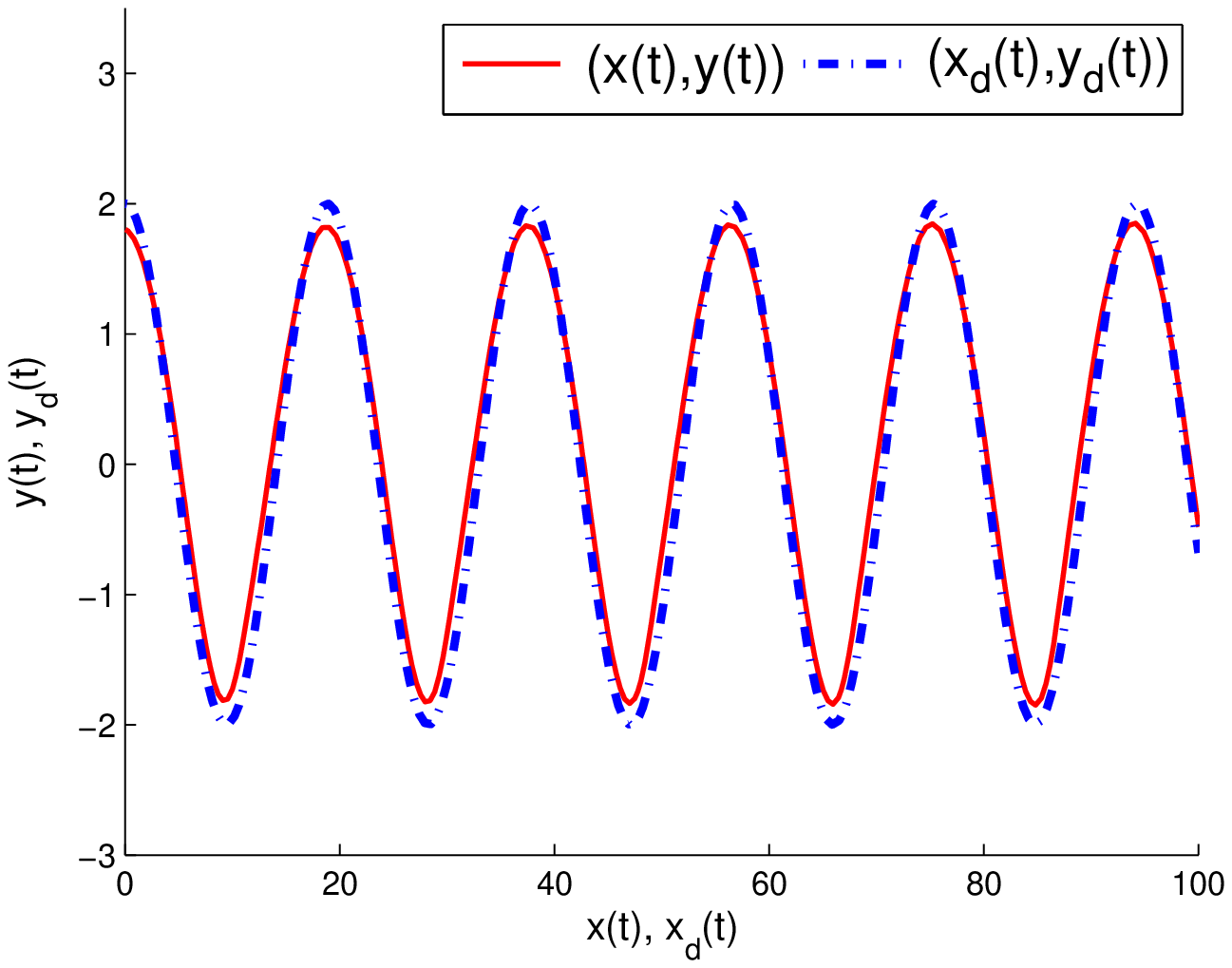}}
\subfigure[Oscillating error trajectories. \label{Minus1r2s2TrajC}]
{\includegraphics[width=.24\columnwidth,height=.18\columnwidth]{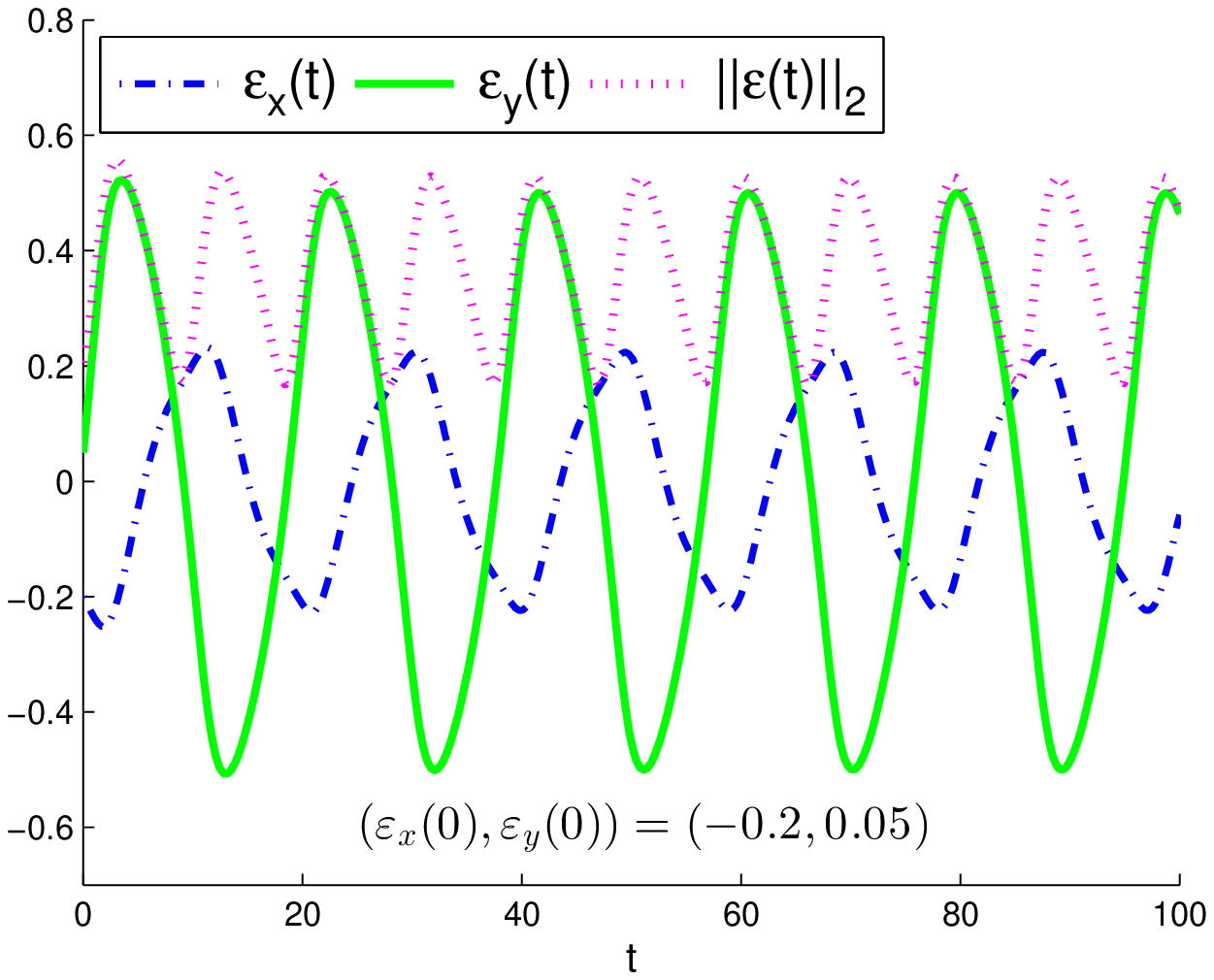}}
\caption{Bifurcation controllers are used to derive the numerical error and solution trajectories associated with the tracking problem given by \eqref{BfControlr2s2}, \eqref{Trackr2s2}, and \eqref{DesiredSol}. Figures (\ref{Minus1r2s2PhaseA}, \ref{Minus1r2s2TrajA}), (\ref{Minus1LimitCycleB}--\ref{Minus1r2s2TrajB}), and (\ref{Minus1LimitCycleC}--\ref{Minus1r2s2TrajC}) correspond to parameters taken from regions (\(a\)), (\(b\)) and (\(c\)) in figure \ref{Fig6(b)}.}
\end{center}
\end{figure}

For the numerical simulation, we further take
\be\label{DesiredSol} x_d(t):= 2\cos(\frac{t}{3}), \qquad y_d(t):= t, \ee
the parameters \((\mu_1, \mu_3):=(-0.05,0.03)\) and \((\mu_1, \mu_3):=(-0.15,0.003)\) from regions (b) and (c) in figure \ref{Fig6(a)} and the initial values \((x(0), y(0))= (2.09, 0.2)\) and \((x(0), y(0))= (1.95, -0.33),\) \ie
\bes(\varepsilon_x(0), \varepsilon_y(0))= (0.09, 0.2)\qquad \hbox{ and } \qquad (\varepsilon_x(0), \varepsilon_y(0))=(-0.05, -0.33).\ees

Figures \ref{phaseportrraitr2s2b} and \ref{phaseportrraitr2s2c} depict the phase portraits associated with \((x(t), y(t))\) versus the desired solution, while figures \ref{r2s2Fig(b)} and \ref{r2s2Fig(c)} illustrate the trajectories of the error functions \(\varepsilon_x(t)\) and \(\varepsilon_y(t)\) along with the error magnitude \(||(\varepsilon_x(t), \varepsilon_y(t))||_2\), respectively. The undershoot and overshoot associated with the error functions \(\varepsilon_x(t)\) and \(\varepsilon_y(t)\) are illustrated in figures \ref{LimitCycleB} and \ref{LimitCycleC}; these are almost \((-0.1, 0.32)\) and \((-0.05, -0.25)\), respectively. The overshoot (undershoot) here refers to the positive (negative) maximum value that a trajectory exceeds its final steady-state value.

The numerically computed basins of attraction associated with \(\varepsilon_x(t)\) and \(\varepsilon_y(t)\) in figures \ref{r2s2Fig(b)} and \ref{r2s2Fig(c)} are ellipse like shapes with the minor and major axes of at least \((0.1, 0.33)\) and \((-0.08, 0.33)\), respectively; see figures \ref{LimitCycleB} and \ref{LimitCycleC}. These also provide the basin of attraction for the desired trajectories, \ie \(x_d(t)\) and \(y_d(t)\). Since the undershoot and overshoot values remain within the basin of attraction for the origin, the normal form and its bifurcation analysis are valid. In other words, the initial values \(x(0)\) and \(y(0)\) should be sufficiently close to the initial desired values \(x_d(0)\) and \(y_d(0)\) so that the (first) overshoot and undershoot of the trajectories remain within the basin of attraction. This constructs a permissable set for the initial values for either of the cases. In our numerical example, the permissable sets include at least the rectangles \([-0.06, 0.09]\times [-0.1, 0.2]\) and \([-0.05, 0.05]\times[-0.21, 0.21],\) respectively. These restrictions for the initial values seem reasonable for many practical engineering problems.

The case \(a_2=-b_2=1\) provides a feasible approach to regularize the origin with a larger basin of attraction than the case \(a_2= b_2=1\). Consider a reassignment of \(d_3:=d_6:=-1\) and keep the remaining constants as stated above in the controlled plant \eqref{BfControlr2s2}, \eqref{Trackr2s2} and \eqref{DesiredController} to obtain the case \(a_2=- b_2=1\). Now figure \ref{Fig6(b)} describes the transition set. Next we observe that the local dynamics associated with the region (\(a\)) in figure \ref{Fig6(b)} only includes an attracting equilibrium at the origin. The numerical simulation in {\sc MATLAB}, using \((\mu_1, \mu_3):=(-0.05,0.03)\) and the initial errors \((\varepsilon_x(0), \varepsilon_y(0))= (0.82, -0.9),\) provides the figures \ref{Minus1r2s2PhaseA} and \ref{Minus1r2s2TrajA}. Parameter choices from the regions (\(b\)) and (\(c\)) in \ref{Fig6(b)} give rise to an oscillatory behavior for the asymptotic dynamics of the error functions. For instance the parameters
\bes (\mu_1, \mu_3):=(0.05,0.01)\qquad \hbox{ and } \qquad (\mu_1, \mu_3):=(0.15,0.003)\ees
from regions (\(b\)) and (\(c\)) in figure \ref{Fig6(b)} and assuming the initial errors
\bes ((\varepsilon_x(0), \varepsilon(0)):= (-0.1 ,-0.35)\qquad \hbox{ and } \qquad (\varepsilon_x(0), \varepsilon(0)):= (-0.2, 0.05)\ees give rise to the figures \ref{Minus1LimitCycleB}-\ref{Minus1r2s2PhaseB} and \ref{Minus1LimitCycleC}-\ref{Minus1r2s2PhaseC}, respectively.

\subsection{Controller design for ship course tracking problem} \label{ShipCourse}

In this section we apply our bifurcation controller design approach to a ship course tracking control problem. Ship course control problem is a typical problem in control engineering and problems of this type also occur in {\it aircraft navigation}, {\it vehicle steering control} and {\it motion planning} in robotics; also see \cite{ChenBifuControl,ChenBifControl2000} for a brief review on some other
engineering applications of bifurcation control.

The ship's helmsman determines the desired ship course (yaw angle) and it is translated into the ship's wheel. The ship's wheel is related to the rudder angle. However, the dynamics between the rudder angle and the ship's yaw angle is nonlinear. Therefore, there is a need for a feedback controller design so that the helmsman's command through the wheel would reflect to an appropriate time-varying rudder angle. Then, the controlled rudder angle enforces the ship's yaw angle to closely track the desired ship course. In this direction we consider a nonlinear dynamics model between the rudder and yaw angles, that is given by
\be\label{BechSmith}
T_1T_2 \,\ddot{r}(t)+ (T_1+T_2) \,\dot{r}(t)+ K H(r(t))= K\delta(t)+ KT_3 \,\dot{\delta}(t), \qquad r:= \dot{\psi}(t),
\ee where \(T_1,\) \(T_2,\) \(T_3,\) \(K\) are constants related to the mass, speed, and hydrodynamic coefficients, while \(\psi,\) \(r,\) \(\delta\) stand for the yaw angle, yawing rate, and rudder angle, respectively. This is the well-known model of Nomoto with a nonlinear {\it manoeuvreing characteristic} function \(H(r)\) due to Bech and Wagner-Smith; see \cite{ShipFossen,ShipBackStep,ShipTomera,BechWagnerSmith}. The function \(H\) is estimated by
\bes
H(r)= c_0+ c_1r+ c_2r^2+ c_3r^3,
\ees through Kempf's zigzag maneuver test; see \cite{ShipBackStep,ShipTomera,BechWagnerSmith}. 
The goal is to control the rudder angle so that the actual yaw angle closely follows the desired yaw angle, say \(\psi_d(t).\) This is achieved through our bifurcation control results and approach described in subsection \ref{BifCont}. Thus, let \(x_1:= \psi-\psi_d,\) \(x_2:= \dot{\psi}_d-r,\) \(x_3:= \ddot{\psi}_d-\dot{r},\) where \(\psi_d(t)\) represents the desired function for the yaw angle.
Further we take
\be\label{ShipCont1}
\delta:= c_0- c_1x_2+ v_1+\frac{v_2}{K T_3}+\int_0^t \left(\frac{T_1T_2}{KT_3}\psi^{(3)}_d(\tau)+ \frac{T_1+T_2}{KT_3}\ddot{\psi}_d(\tau)\right)\exp\Big(\frac{\tau-t}{T_3}\Big)d\tau,
\ee where
\bes
\dot{v}_2=-\frac{v_2}{T_3}\exp\big(\frac{-t}{ T_3}\big)+ K(c_1\dot{\psi}_d+c_2{\dot{\psi}_d}^2-2c_2{\dot{\psi}_d}x_2+c_3{\dot{\psi}_d}^3-3c_3x_2{\dot{\psi}_d}^2+3c_3{x_2}^2{\dot{\psi}_d}).
\ees Then, the model in \((x_1, x_2, x_3)\)-coordinates reads
\be\label{ShipEq1}
\dot{x}_1= -x_2, \qquad \dot{x}_2= x_3, \qquad \dot{x}_3= -\frac{Kv_1}{T_1T_2}- \frac{KT_3\dot{v}_1}{T_1T_2}+\left(\frac{KT_3}{T_1T_2}c_1-\frac{T_1+T_2}{T_1T_2}\right)x_3+ \frac{K}{T_1T_2}(c_2 {x_2}^2- c_3 {x_2}^3).
\ee
\begin{figure}
\begin{center}
\subfigure[Pitchfork and Hopf varieties.\label{Fig13(a)}]
{\includegraphics[width=.29\columnwidth,height=.22\columnwidth]{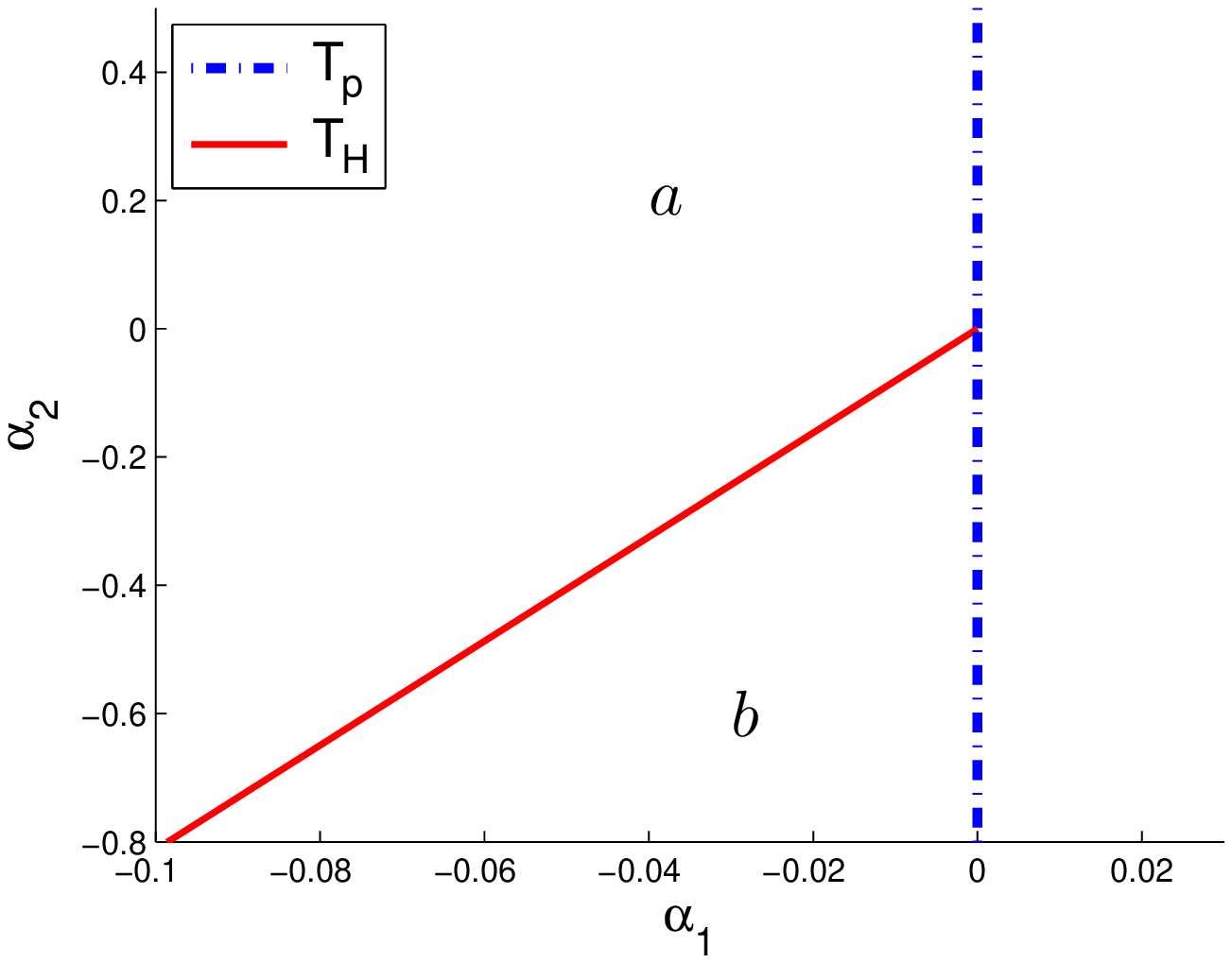}}
\subfigure[The desired and highly accurate tracking yaw solutions.\label{Fig13(b)}]
{\includegraphics[width=.35\columnwidth,height=.22\columnwidth]{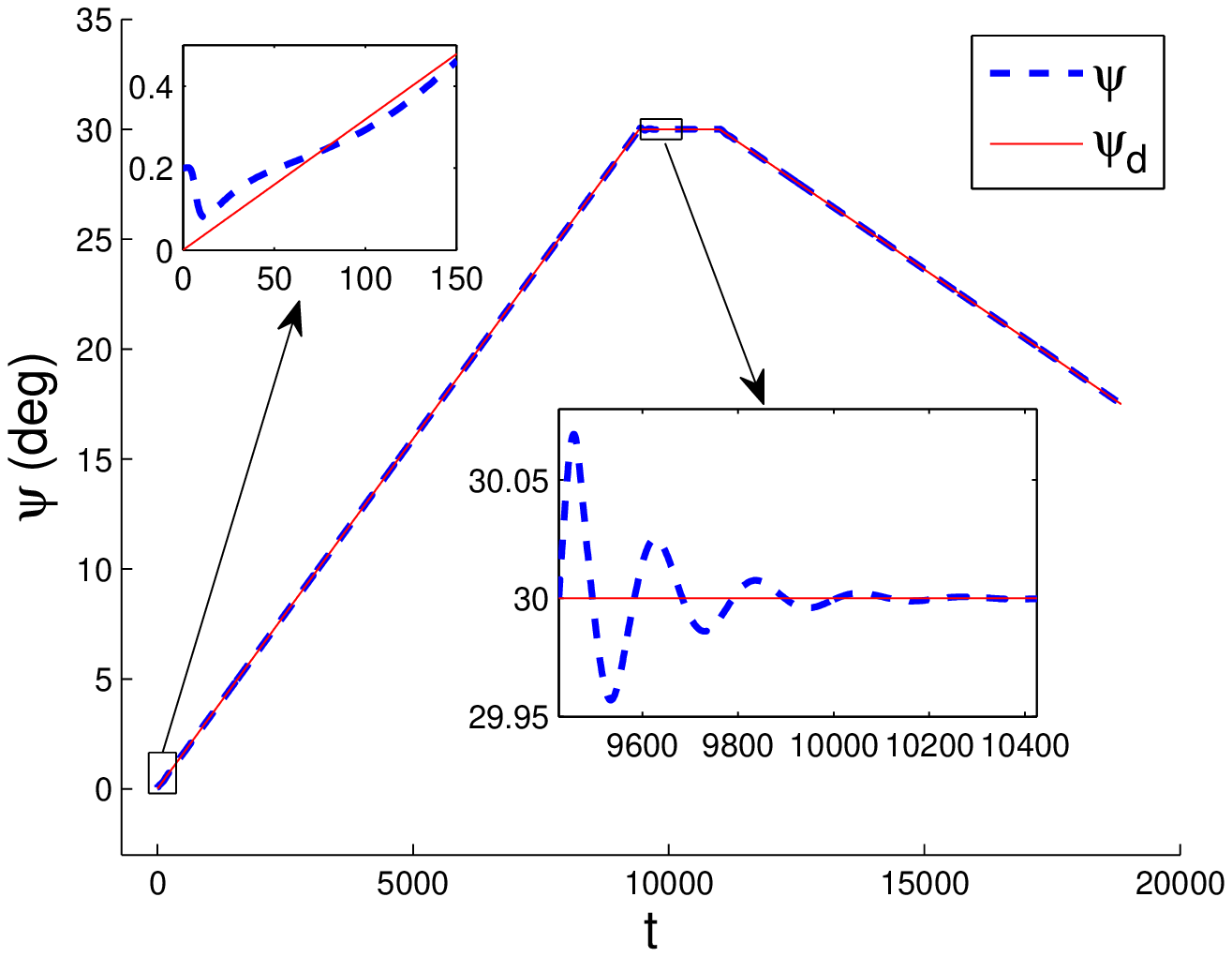}}
\subfigure[Tracking errors in yaw angle. \label{Fig13(c)}]
{\includegraphics[width=.29\columnwidth,height=.22\columnwidth]{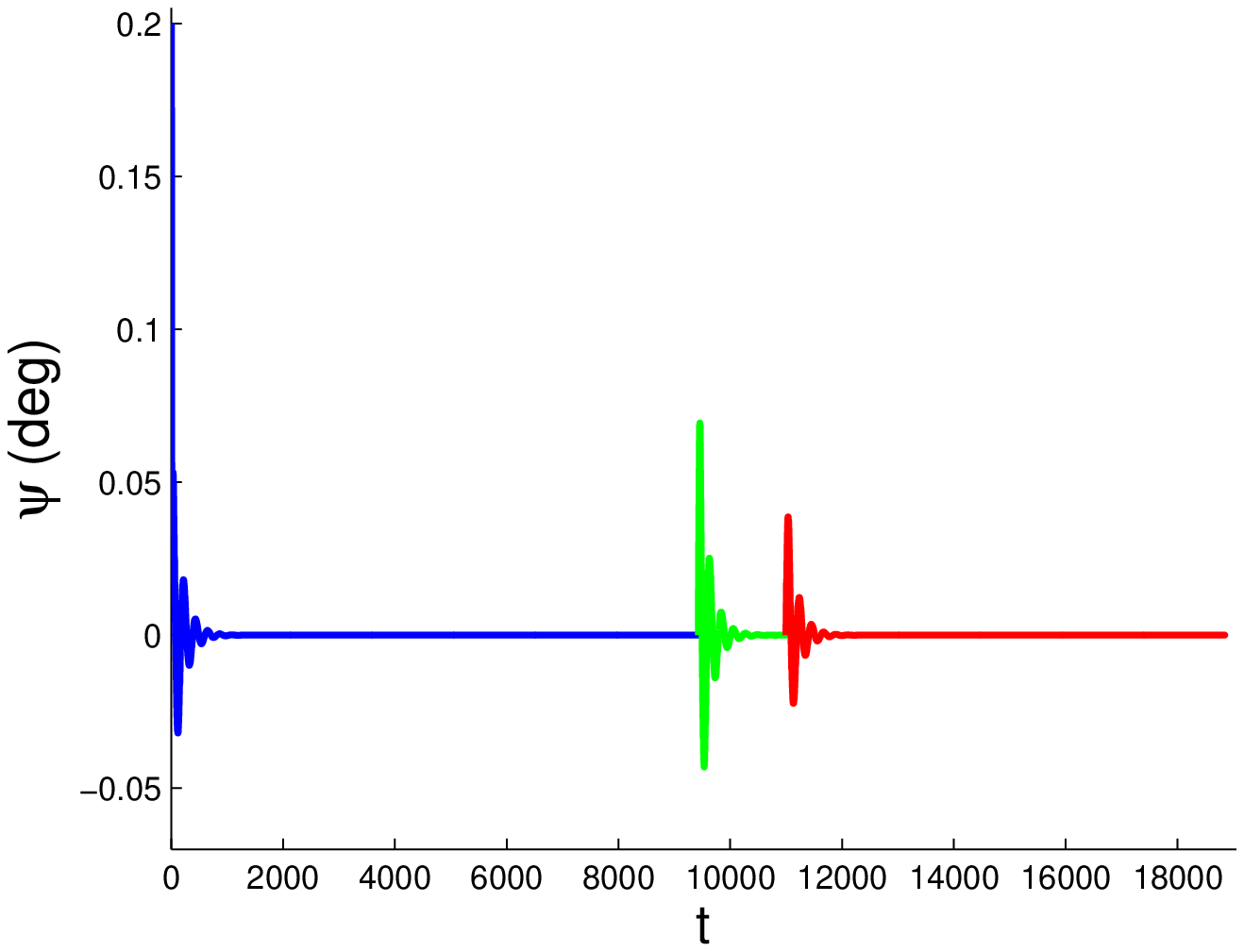}}
\caption{The ship course trajectory associated with controlled system \eqref{BechSmith}--\eqref{ShipCont1} tracks the desired time-dependent yaw angle \eqref{DesiredYaw}. }
\label{BifContr1s1}
\end{center}
\end{figure}
Denote \(c:=T_1+T_2-c_1KT_3\) and let
\bes
v_1:= \alpha_1x_1+\alpha_2x_2+ c_2{x_2}^2 -\frac{c^3}{K{(T_1T_2)}^2}{x_1}^3+{\frac {c^2(3T_1T_2-3{T_3}c-4c)}{K{(T_1T_2)}^2}}{x_1}^2{x_2}.
\ees The system \eqref{ShipEq1} has a generalized cusp case Bogdanov-Takens singularity at the origin. Hence, the cubic-estimation of its governing equations on the center manifold is given by the system \eqref{BfControlr2s2}-\eqref{u1u2}, where \(d_5= d_6= d_7= d_8=0, d_2= 1, d_3=-4 , u_2=0,\) and
\bas
&d_1= -\frac{K c_3 (T_1 T_2)^2}{c^3}+ \frac{T_1T_2T_3(-6T_3c+8c+6T_1 T_2)}{c^2}-4\frac{ {(T_1T_2)}^2}{ c^2 }-2{\frac {{(T_1 T_2)}^3}{ c ^3}}, d_4= -\frac{T_3(6{T_3}c-8c-6T_1T_2)  }{c}-8\frac{T_1T_2}{c }-3\frac{  {(T_1T_2)}^2 }{c^2}.&
\eas These give rise to the case \(r=s=2,\) equations \eqref{why}, \(a_2= - b_2= 1,\) \(\nu_3=\frac{K (-\alpha_2 c+ \alpha_1 cT_3 -\alpha_1 T_1 T_2)}{2 c^2},\) and
\bas
&\nu_2=-\frac{K \alpha_1}{c}+
\frac{3}{2} \frac{K^2 {(T_1 T_2)}^2 {\alpha_1}^2 (K c_3 +2 T_1 T_2)}{c^5}-\frac{1}{4}\frac{K^2 {(T_1 T_2)}^2 {\alpha_1}^2 (36T_3-27)}{c^4}
+\frac{K^2 T_1 T_2 \alpha_1 \left(6T_3 \alpha_1 (18 T_3-29) -18 \alpha_2 +37 \alpha_1\right)}{12 c^3}&\\&
+\frac{K^2 \left({\alpha_1}^2 (45 {T_3}^2 -16 -37 T_3) +9 \alpha_1 \alpha_2 (2T_3 -3)-3 {\alpha_2}^2 \right)}{12 c^2}.\qquad\qquad\qquad\qquad\qquad\qquad\qquad\qquad\qquad\quad&
\eas We skip a complete bifurcation control analysis for briefness. For the feedback tracking controller design, we use the bifurcation varieties associated with 
pitchfork and Hopf. 
In an experimental case study \cite[Table II]{ShipTomera}, the estimations for the constants in an unstable ship course are given as follows
\be\label{Cis} c_0:=-1.2947, \quad c_1:= -7.3227, \quad c_2:=2.8658, \quad \hbox{ and } \quad c_3:=9.7678.\ee
Given the permissible values in \cite[Table I]{ShipTomera}, we further take
\be\label{Tis}T_1:= T_2:= T_3:= 10,\quad \hbox{ and } \quad K:=0.5\ee for numerical simulations. Hence
\begin{eqnarray}
T_P&=& \{(\alpha_1, \alpha_2)\,|\, \alpha_2=-2.3943 \alpha_1\pm1.2502\sqrt{-12.8360 {\alpha_1}^2-1.4128 \alpha_1}\},
\end{eqnarray}
and \(T_{H}\) follows
\bas
&\alpha_2=6.8759+2.35045{ \alpha_1}-1.5571\sqrt {19.50-33.370 \alpha_1+6.9905 {\alpha_1}^2},&
\eas
 see figure \ref{Fig13(a)}. The dynamics associated with the region (a) in figure \ref{Fig13(a)} is similar to figure \ref{Fig10(a)}, except that the equilibrium is a stable spiral sink. We assume that a desired solution for the yaw angle is
given by \(\psi_d(t):= \frac{t}{100\pi}\) for \(0\leq t\leq 3000\pi,\)
\be\label{DesiredYaw}
\psi_d(t):= 30 \; \hbox{ for } \; 3000\pi<t\leq 3500\pi, \; \hbox{ and } \; \psi_d(t):= 30-\frac{t-3500\pi}{200\pi} \; \hbox{ for } \; 3500\pi<t\leq 5500\pi.
\ee We choose the parameters \((\alpha_1, \alpha_2):= (-0.1, 0.5)\) from region (a) and assume that the initial errors are given by
\be\label{InitError} (x_1(0), x_2(0), x_3(0)) := (0.2, \frac{1}{100\pi}, 0.001).\ee These give rise to a highly accurate ship's trajectory in tracking the (sea) route determined by \eqref{DesiredYaw} for the ship's course controlled plant \eqref{BechSmith}-\eqref{ShipCont1}; see figures \ref{Fig13(b)} and \ref{Fig13(c)}. These demonstrate that the controller \(\delta\) with parameters chosen from region (a) in figure \ref{Fig13(a)} provides a highly accurate controller for the ship course tracking problem \eqref{BechSmith}-\eqref{DesiredYaw}.

The basin of attraction for the desired solution includes the initial errors of at least up to the values in \eqref{InitError}. These initial errors can be interpreted as the {\it desired discontinuities} in the yaw, yaw rate and yaw acceleration. Possible desired discontinuities beyond the basin of attraction can be readily handled by interpolating a new smooth desired solution for the controller computation. In this direction, the electrical and mechanical filters are also the means to attenuate the possible values beyond the limitations.

\begin{rem}\label{Rem7.1} The feedback realizations of \(\delta\) and \(v_1\) in practice are feasible through a gyrocompass and a rate gyro. Given the different persistent qualitative dynamics list, small cognitive variation of controller parameters \((\alpha_1, \alpha_2)\) provides a flexible approach for a sudden and smooth change in the manoeuvring policy. An instance of this may include a small oscillating movement around a desired sea route; \eg a stable limit cycle around the origin causes such a tracking solution. A second example is to dodge an incoming object in order to scape from an immediate collision. Indeed, the parameter choices from the region \((c)\) in figure \ref{Fig14(a)} leads to a qualitative dynamics change and the origin becomes unstable. Hence the yaw angle is directed to a rapid divergence from the ship's route. The ship returns to its desired route by simple updates to the desired solution and controller inputs.
\end{rem}

\begin{figure}
\begin{center}
\subfigure[Transcritical, Hopf and estimated homoclinic bifurcation varieties.\label{Fig14(a)}]
{\includegraphics[width=.32\columnwidth,height=.2\columnwidth]{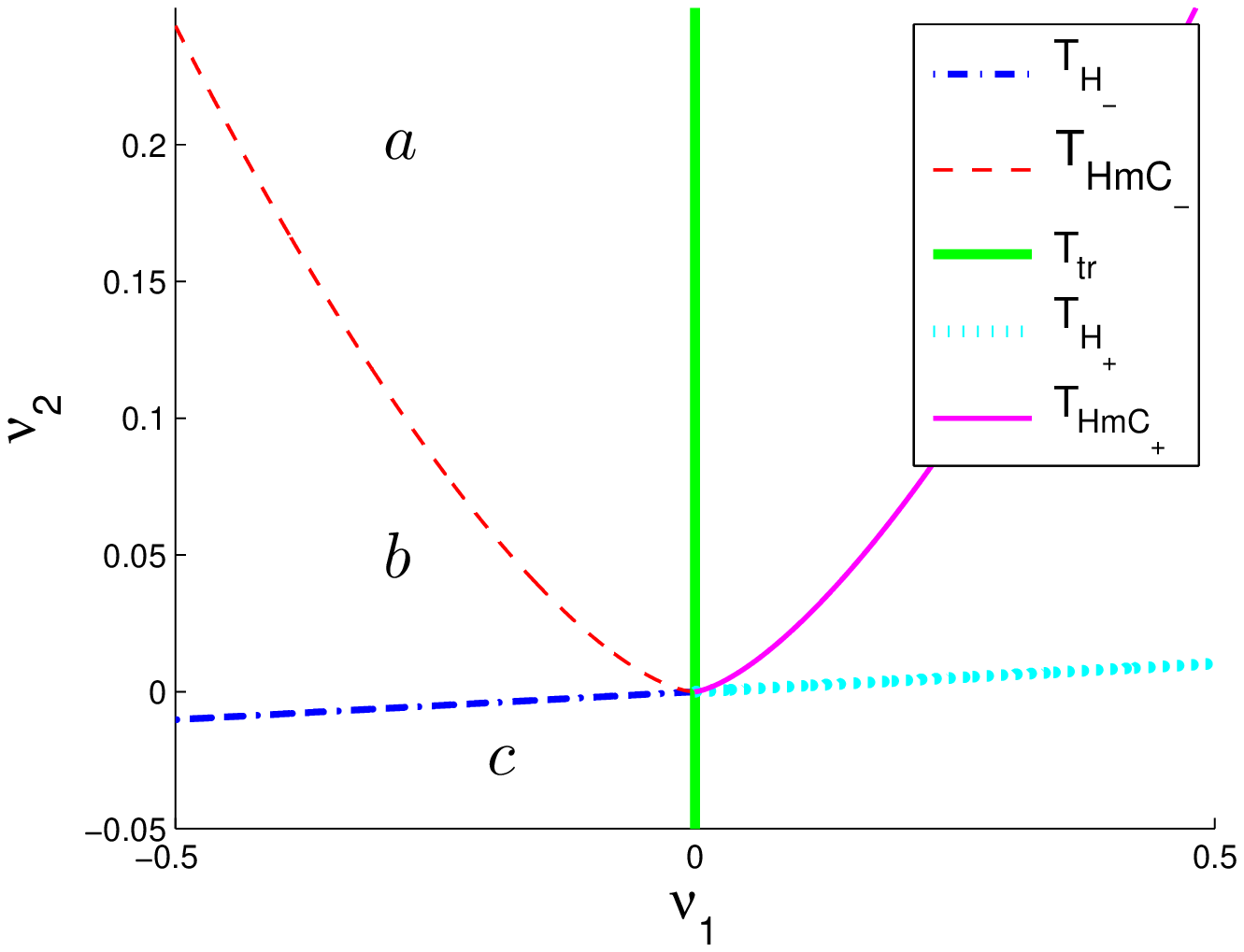}}
\subfigure[The desired and actual yaw angles for parameters from region \((a)\).\label{Fig14(b)}]
{\includegraphics[width=.32\columnwidth,height=.2\columnwidth]{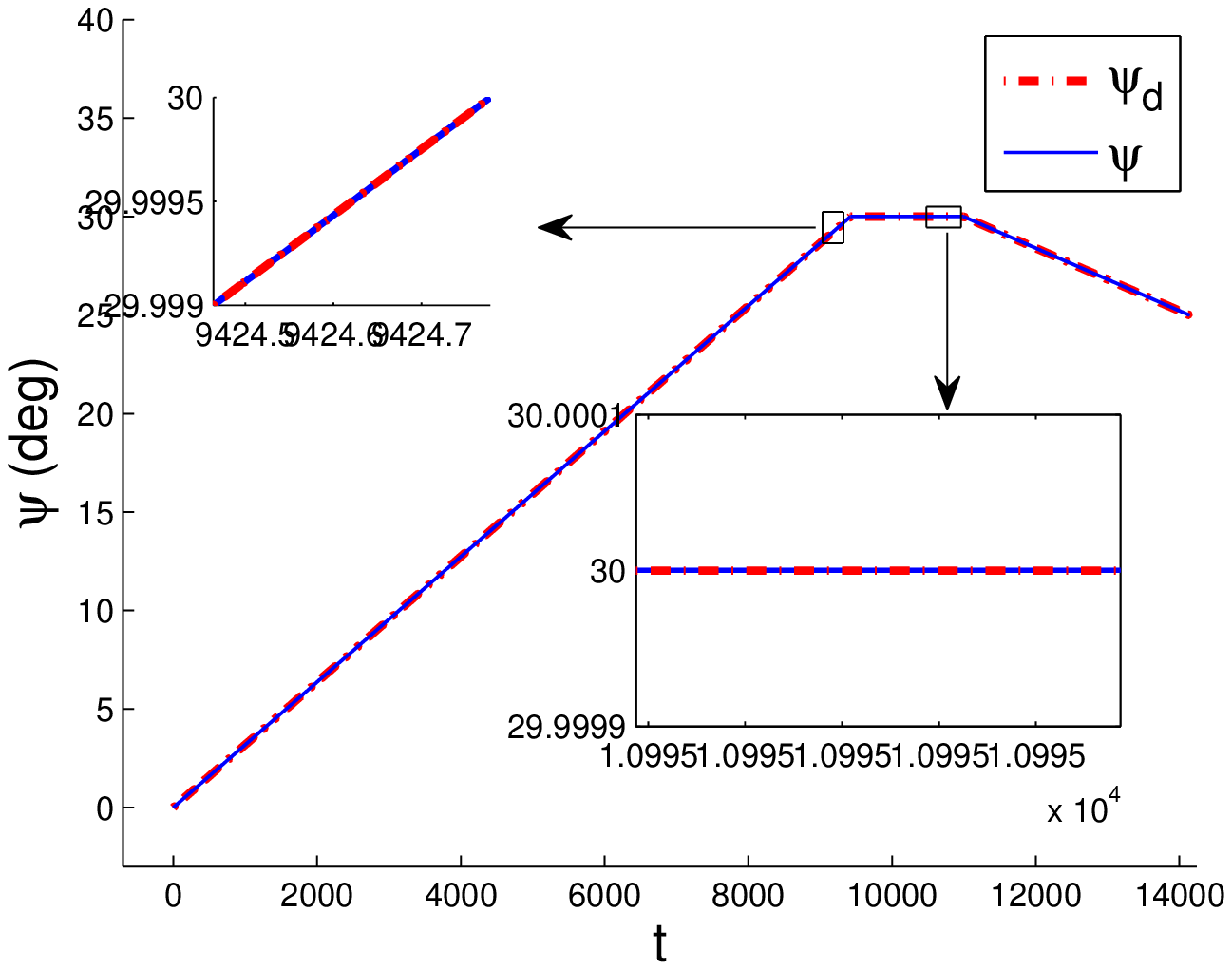}}
\subfigure[Tracking errors in yaw angle. \label{Fig14(c)}]
{\includegraphics[width=.32\columnwidth,height=.2\columnwidth]{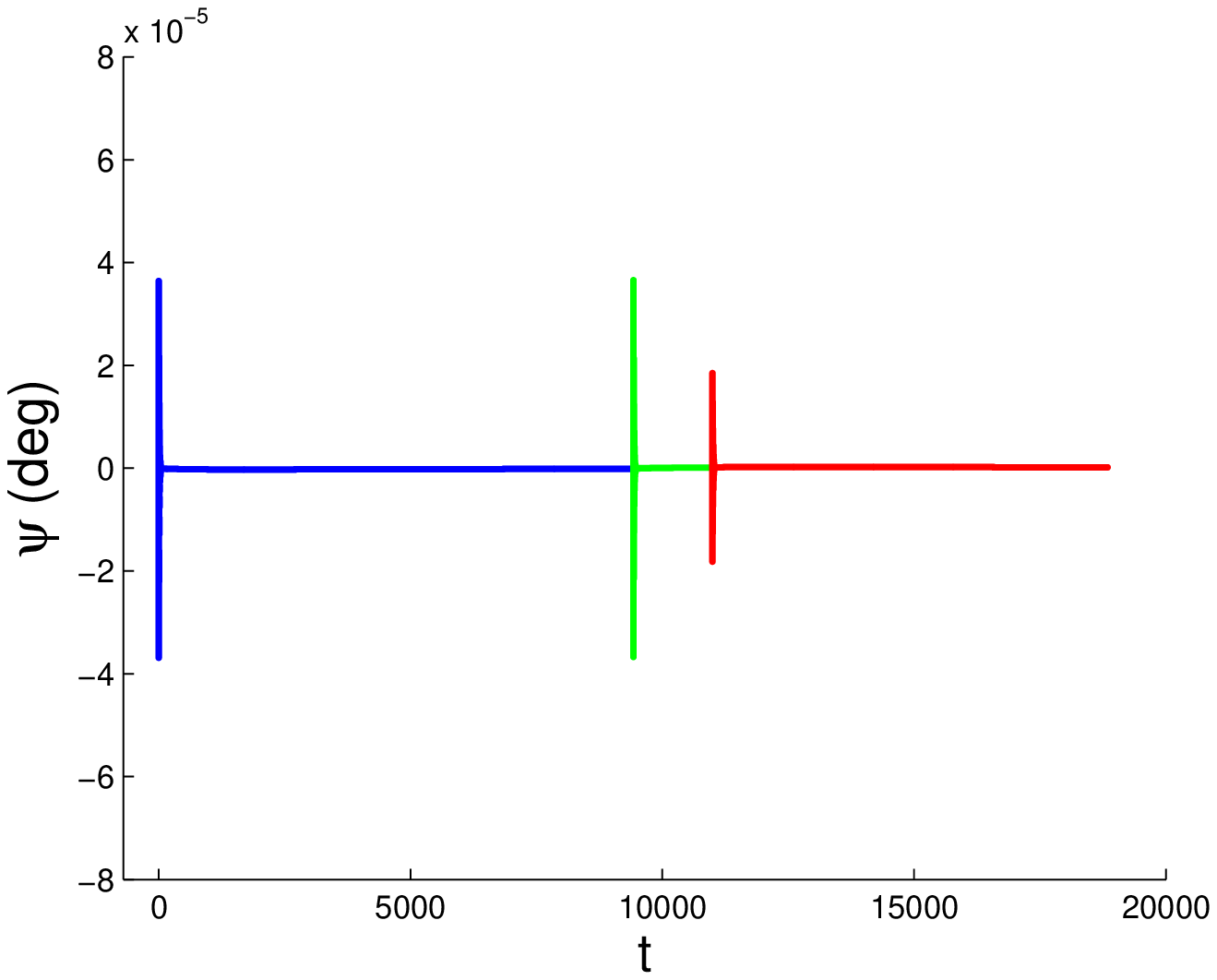}}
\subfigure[The desired and actual yaw angles with parameters taken from region \((b)\).\label{Fig14(d)}]
{\includegraphics[width=.38\columnwidth,height=.2\columnwidth]{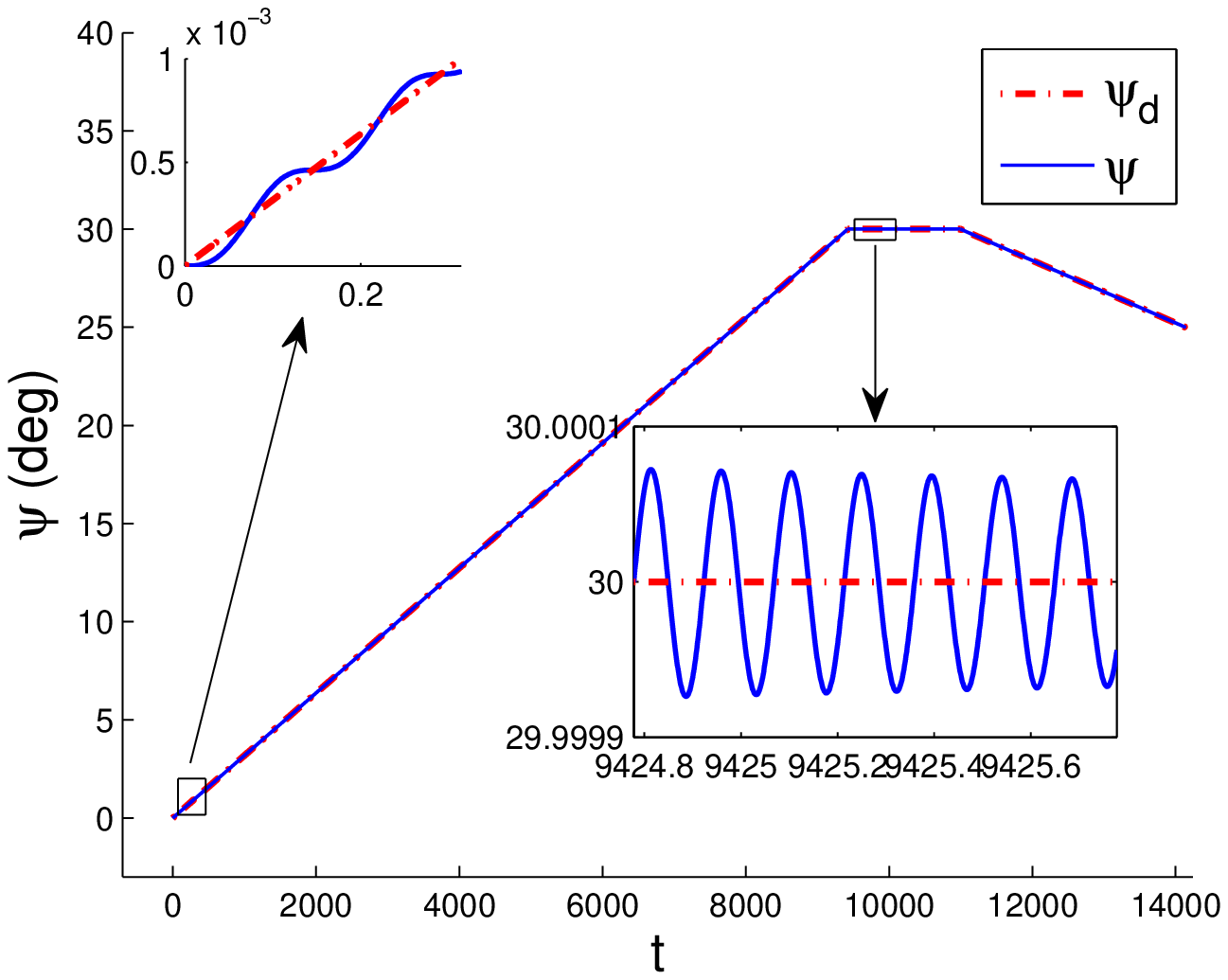}}
\subfigure[Tracking errors in yaw angle for parameters from region \((b)\). \label{Fig14(e)}]
{\includegraphics[width=.38\columnwidth,height=.2\columnwidth]{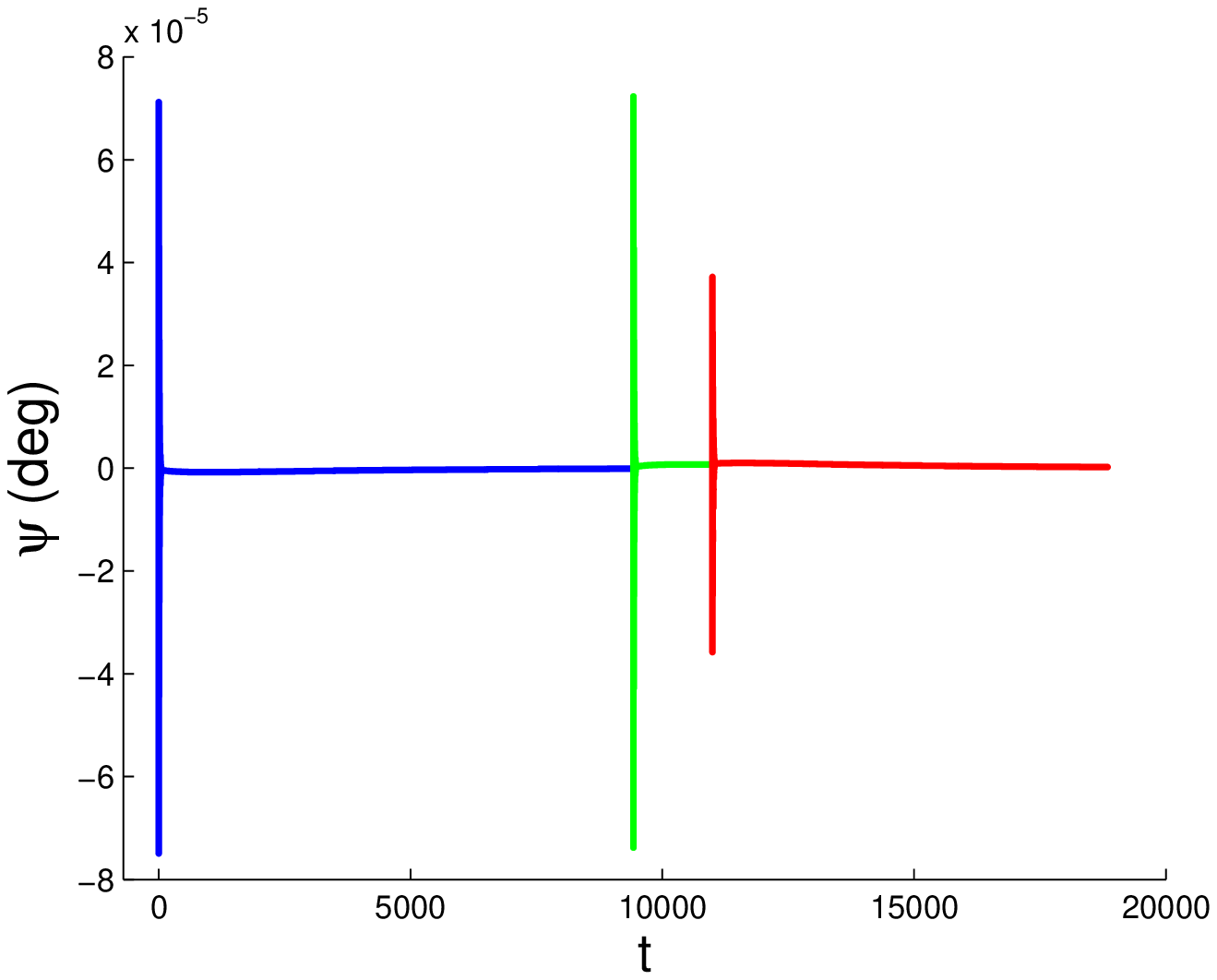}}
\caption{Transition sets, highly accurate tracking orbits and yaw error trajectories associated with the ship course controller in the ship steering system \eqref{BechSmith}--\eqref{ShipSteering}--\eqref{deltaiR1S1}.}
\label{BifContr1s1}
\end{center}
\end{figure}

A gearbox's dynamics adds an extra state dimension into the ship course problem. Indeed, we consider a steering gear dynamics that is modeled by
\be\label{ShipSteering}
\dot{\delta}(t)= -\frac{\delta(t)}{T_R}+ \frac{K_R}{T_R}\delta_i(t),
\ee
where \(\delta_i(t)\) stands for the feedback controller angle, \(K_R:=96\) and \(T_R:= 156(s)\); see \cite[page 278]{ShipBackStep}. Here the input controller signal \(\delta_i\) is fed into the gearbox and the gearbox's output angle determines the ship's rudder angle. Let the rudder angle \(\delta(t):= x_4-\frac{\varphi(t)}{KT_3}\) and \(x_4\) denote a new state variable.
Similar to the above approach, we further take
\be\label{deltaiR1S1}
\delta_i:= \frac{v_0+ v_1+T_Rv_2}{K_R}, \qquad v_0:=c_0+K_R \alpha_1x_1+({T_R}^2 K_R\alpha_2-c_1)x_2,
\ee where
\bas
&v_1:=\frac{C^2\,{x_1}^2+C\big(3C-2 \left({T_R}-T_3\right)\left(T_1+T_2-c_1K T_3\right)-2T_1T_2\big)\,x_1x_2}{K\left(T_1+T_2\right)}, \qquad\hbox{ for } \quad C:=T_1+T_2+c_1K(T_R-T_3),\; &\\
&v_2:=\frac{T_1T_2{\psi_d}^{3}+(T_1+T_2){\psi_d}^{2}}{T_3K} + \frac{\left( c_1+3 c_3 { x_2}^2-2 c_2 x_2 \right) \psi_d +c_2 { \dot{\psi_d}} ^2-3c_3 x_2{\dot{\psi_d}}^2+c_3{\dot{\psi_d}}^3}{T_3}-\frac{\varphi(t)\left(T_3-T_R\right)}{K{T_3}^2T_R}, \;\;&
\\
&\dot{\varphi}:=\frac{T_1 T_2}{K}{{\psi}_d}^{(3)} +c_3{\dot{\psi_d}}^3+\frac{(T_1+T_2)}{K}{{\psi}_d}^{(2)} +c_2{\dot{\psi_d}}^{2}-3c_3x_2\dot{\psi_d}^2-2c_2 x_2{\dot{\psi}_d}+c_1\dot{\psi}_d+3c_3{x_2}^2{\dot{\psi}_d}-\frac{\varphi(t)}{T_3}\exp({\frac{-t}{T_3}}).&
\eas
Thus,
\ba\label{ShipEq2}
&\dot{x}_1= -x_2, \qquad\qquad \dot{x}_2= x_3, \qquad\qquad \dot{x}_4= -\frac{x_4}{T_R}+\frac{1}{T_R}(v_0+ K_Rv_1), &\\\nonumber&
\dot{x}_3= \frac{K}{T_1T_2}\left(\dfrac{T_3}{T_R}-1\right)x_4- \frac{KT_3}{T_1T_2T_R}(v_0+ K_Rv_1)-\frac{T_1+T_2}{T_1T_2} x_3+\frac{K}{T_1T_2}(c_0-c_1 {x_2}+ c_2 {x_2}^2-c_3 {x_2}^3).&
\ea Then, a center manifold reduction around \((x_1, x_2, x_3, x_4)= (0, 0, 0, c_0)\) reduces the system into equations \eqref{BfControlr1s1}, where
\ba\label{CenterManfCoef}
&d_1=1, \quad d_2= 3-\frac{2T_1 T_2}{T_1+T_2}+2 (T_3-T_R),\quad d_4=T_R-T_3+\frac{T_1 T_2}{T_1+T_2},\quad d_5= \frac{T_1 T_2\left(T_R-T_3-3+\frac{T_1 T_2}{T1+T2}\right)}{K (T_1+T_2)},&
\ea while \(\mu_1= \mu_4= \mu_5=0,\)
\begin{eqnarray}\label{CenterManfPars}
&\mu_2=\frac {c_1(T_R-T_3)(T_R(T_1+T_2)+T_1T_2)}{C^2K^{-2}{K_R}^{-1}(T_1+T_2)}\alpha_1-\frac{K{K_R}{T_R}^2}{C}\alpha_2, \quad \mu_3=-\frac{KK_R}{C}\alpha_1,\hbox{ and } \mu_6=\frac{KK_R \left(T_3-T_R-\frac{T_1 T_2}{T_1+T_2}\right)}{C}\alpha_1.\qquad &
\end{eqnarray} Some negative powers appear in the denominators for briefness.
The system falls within \(r=s=1,\) and reads the normal form equation \eqref{r1s1NF} by \(a_1=b_1=1,\) \(\nu_1=-\frac{K^2{K_R}^2}{4{C^2}}{\alpha_1}^2,\)
\bas
&\nu_2=\left(\frac{3 K K_R}{4C}+\frac{(T_3-T_R)(T_1+T_2)-T_1T_2}{2(T_1+T_2)^2C{K_R}^{-1}{T_1}^{-1}{T_2}^{-1}} -\frac{c_1 (T_3-T_R) \left(T_1 T_2+T_R( T_1+T_2)\right)}{2 {K_R}^{-1} K^{-2} (T_1+T_2) C^2}\right) \alpha_1-\frac{T_1 T_2{K_R}}{2C{T_R}^{-2}}\alpha_2.&
\eas Since \(\nu_1<0\) for all nonzero values for \(\alpha_1,\) there are always two local equilibria bifurcating from the origin via a transcritical bifurcation at \(T_{tr}:=\{(\alpha_1, \alpha_2)\,|\, \alpha_1=0\}\).

An advantage of our controller parameter choice \((\alpha_1, \alpha_2),\) \ie \(\alpha_1 x_1+\alpha_2 x_2\) via \(v_1,\) at equation \eqref{ShipSteering} is that the origin is one of the two bifurcated local equilibria. When the origin is an asymptotically stable equilibrium, it provides an ideal choice for the controller parameters. Otherwise, we would be faced with a stable equilibrium that was in the vicinity of (but different from) the origin. As a consequence, we had to provide an adaptive controller design for a sufficiently accurate tracking solution in large time intervals.

Now our goal is to choose the controller parameters from a region in the parameter space so that they associate an asymptotically stable local dynamics for the origin. For numerical simulations and briefness, we take the numerical values in \eqref{Cis}-\eqref{Tis}, except \(T_3:= 500\).
Further,
the Hopf bifurcation varieties are given by
\bas
&T_{H+}=\big\{(\alpha_1, \alpha_2) |\,\alpha_1=\frac{25948349232}{533642261}\alpha_2\; \hbox{ for } \alpha_2>0\big\}\;\hbox{ and }\;
T_{H-}=\big\{(\alpha_1, \alpha_2) |\,\alpha_1=\frac{2162362436}{44203625}\alpha_2\; \hbox{ for } \alpha_2<0\big\}, \;&
\eas while the estimated \(T_{HmC_\pm}\) in equation \eqref{HomocV} follows
\bas
&100\alpha_2=2.050392014 \alpha_1\pm {\alpha_1}\sqrt{.00001938331498-0.6319781864{|\alpha_1|}^{\frac{1}{2}}+5151.291564|\alpha_1|}.&
\eas

The transition sets of Hopf, transcritical and homoclinic bifurcations are plotted in figure \ref{Fig14(a)}.
The unstable limit cycle corresponding to the region \((a)\) in figure \ref{Fig14(a)} constructs a restriction on the basin of attraction for the origin; that is, the errors corresponding to the initial values must be taken small enough to fall within the limit cycle. The limit cycle disappears through a homoclinic bifurcation, when we take our choices of parameters from region \((b).\) Then, the only remaining restriction for the basin of attraction is the normal form neighborhood validity.

We choose the initial values
\bes\big(\psi(0), \dot{\psi}(0), \ddot{\psi}(0), \delta(0), \varphi(0)\big):= (0, 0, 0, c_0, 0),\ees the parameters \((\alpha_1, \alpha_2):= (-0.3, 0.2)\) and \((\alpha_1, \alpha_2):=(-0.3, 0.05)\) from regions (a) and (b) in figure \ref{Fig14(a)}. Then, the actual controlled yaw angle versus the desired yaw are depicted in figure \ref{Fig14(b)} and \ref{Fig14(d)}, respectively. Figures \ref{Fig14(c)} and \ref{Fig14(e)} depict the tracking error. These confirm that the numerical tracking trajectories are highly accurate in tracking the desired solution.

%

\section{Conclusion}

In this paper we provide a truncated orbital and parametric normal form classification for the family of the generalized cusp case of Bogdanov-Takens singularity. Then, we consider two most generic cases of this family for their local bifurcation analysis. These are the most generic cases of Bogdanov-Takens singularity with and without a \(\mathbb{Z}_2\)-symmetry. These systems demonstrate varieties of bifurcations from primary to quinary bifurcations of various types. These include bifurcations of equilibria via saddle-node, transcritical and pitchfork. Bifurcations of multiple limit cycles are observed through Hopf, homoclinic, heteroclinic, saddle-node of limit cycles, and saddle-connection (double homoclinic) for the \(\mathbb{Z}_2\)-equivariant system. We further study one-parameter \(\mathbb{Z}_2\)-symmetry breaking bifurcations. The formulas for transition sets of the truncated parametric normal forms are derived in sufficiently high orders in order to enlarge their neighborhood validity for their applications in bifurcation control. These are successfully applied on a generic quadratic measurable plant with a possible multi-input linear controller and on a \(\mathbb{Z}_2\)-equivariant general measurable plant with possible multi-input linear (\(\mathbb{Z}_2\)-symmetry preserving) and quadratic (symmetry-breaking) controllers. We show that by feedback controller designs, we can precisely locate and accurately control all types of predicted bifurcations in our bifurcation analysis of the truncated universal asymptotic unfolding normal forms. We further use our parametric normal form analysis to demonstrate the applicability of our results for solving two most common control engineering problems: regulating and tracking problems. Given the smooth dependence on the controller parameters, our approach provides an efficient tool for smooth manoeuvering possibilities: this is an important claimed contribution of our controller design approach. Our bifurcation control approach are applied on two nonlinear ship course models to illustrate the applicability of our results. Our computer simulations confirm our theoretical results and accurate predictions.

\end{document}